\documentclass[twoside]{irmaems}
\usepackage{amssymb} 
\usepackage{amsmath} 
\usepackage{latexsym}

\usepackage{makeidx}
\makeindex

\usepackage[mathscr]{eucal}
\usepackage[english]{babel}
\input pstricks
\input pst-node

\usepackage{pstricks, pst-node}\input xy
\usepackage[all]{xy}

\renewcommand\theequation{\arabic{section}.\arabic{equation}}

\theoremstyle{definition} 

 \newtheorem{definition}{Definition}[section]
 \newtheorem{remark}[definition]{Remark}
 \newtheorem{example}[definition]{Example}

\theoremstyle{plain}      

 \newtheorem{proposition}[definition]{Proposition}
 \newtheorem{theorem}[definition]{Theorem}
 \newtheorem{corollary}[definition]{Corollary}
 \newtheorem{lemma}[definition]{Lemma}

\newtheorem*{conjecture}{Conjecture}

\newcommand{\bqn}{\begin{equation*}}
\newcommand{\eqn}{\end{equation*}}
\newcommand{\bq}{\begin{equation}}
\newcommand{\eq}{\end{equation}}
\newcommand{\ba}{\begin{aligned}}
\newcommand{\ea}{\end{aligned}}
\newcommand{\be}{\begin{enumerate}}
\newcommand{\ee}{\end{enumerate}}

\newcommand{\tr}{\operatorname{tr}}

\newcommand{\aut}{\operatorname{Aut}}
\newcommand{\End}{\operatorname{End}}

\newcommand{\SU}{\operatorname{SU}}
\newcommand{\SL}{\operatorname{SL}}
\newcommand{\Sp}{\operatorname{Sp}}
\newcommand{\PSp}{\operatorname{PSp}}
\newcommand{\PSL}{\operatorname{PSL}}
\newcommand{\PO}{\operatorname{PO}}
\newcommand{\SO}{\operatorname{SO}}
\newcommand{\GL}{\operatorname{GL}}

\newcommand{\PU}{\operatorname{PU}}

\newcommand{\essgr}{\operatorname{Ess\,Gr}}
\newcommand{\essim}{\operatorname{Ess\,Im}}

\newcommand{\CC}{{\mathbb C}}
\newcommand{\DD}{{\mathbb D}}
\newcommand{\FF}{{\mathbb F}}

\newcommand{\NN}{{\mathbb N}}
\newcommand{\PP}{{\mathbb P}}
\newcommand{\QQ}{{\mathbb Q}}
\newcommand{\RR}{{\mathbb R}}

\newcommand{\ZZ}{{\mathbb Z}}

\newcommand{\Bb}{{\mathcal B}}
\newcommand{\Cc}{{\mathcal C}}
\newcommand{\Dd}{{\mathcal D}}

\newcommand{\Hh}{{\mathcal H}}

\newcommand{\Mm}{{\mathcal M}}
\newcommand{\Nn}{{\mathcal N}}

\newcommand{\Uu}{{\mathcal U}}

\newcommand{\Yy}{{\mathcal Y}}

\newcommand{\Zz}{{\mathcal Z}}

\newcommand{\frakg}{{\mathfrak g}}

\newcommand{\fraks}{{\mathfrak s}}
\newcommand{\frakl}{{\mathfrak l}}

\newcommand{\g}{\gamma}

\newcommand{\gG}{{\mathbf G}}
\newcommand{\hH}{{\mathbf H}}

\newcommand{\lL}{{\mathbf L}}

\newcommand{\<}{\langle}
\renewcommand{\>}{\rangle}

\def\h{{\rm H}}
\def\hb{{\rm H}_{\rm b}}

\def\hc{{\rm H}_{\rm c}}
\def\hcb{{\rm H}_{\rm cb}}
\def\hdr{{\rm H}_{\rm dR}}

\def\linfty{L^\infty}

\def\la{L^\infty_{\mathrm{alt}}}

\def\T{\operatorname{T}}

\def\one{\mathbf{1\kern-1.6mm 1}}

\def\homeo#1{\operatorname{Homeo}^+\!\left(#1\right)}

\def\id{{\it I\! d}}

\def\C{{\operatorname{C}}}

\def\arg{{\operatorname{arg}}}

\def\h2{{\operatorname{H_2}}}
\def\h1{{\operatorname{H_1}}}

\def\rk{{\operatorname{rank}}}
\def\tr{{\operatorname{tr}}}

\def\d{{\operatorname{d}}}

\def\id{{\operatorname{Id}}}
\def\ker{{\operatorname{ker}}}
\def\im{{\operatorname{im}}}
\def\r{\operatorname{r}}
\def\qhch{\operatorname{QH}_{\rm c}^{\rm h}}

\def\PSL{\operatorname{PSL}}
\def\SL{\operatorname{SL}}
\def\Sp{\operatorname{Sp}}
\def\SU{\operatorname{SU}}
\def\SO{\operatorname{SO}}

\def\PSU{\operatorname{PSU}}

\def\chiext{\chi_{\rm ext}}
\def\cs{{\check S}}

\def\eb{{e^{\rm b}}}
\def\erb{{e^{\rm b}_\RR}}

\def\kg{\kappa_G}

\def\kgb{\kappa_G^{\rm b}}

\def\ksb{\kappa_S^{\rm b}}

\def\ksrb{\kappa_{S,\RR}^{\rm b}}
\def\ninn{\Nn_{\operatorname{i}}^+}

\def\Tb{\T_{\rm b}}
\def\tb{{\rm t}_{\rm b}}

\def\hcb{{\rm H}_{\rm cb}}
\def\to{\rightarrow}

\def\la{L^\infty_{\mathrm{alt}}}
\def\hb{{\rm H}_{\rm b}}
\def\hc{{\rm H}_{\rm c}}
\def\h{{\rm H}}

\renewcommand{\phi}{\varphi}

\def\No{N\raise4pt\hbox{\tiny o}\kern+.2em}
\def\no{n\raise4pt\hbox{\tiny o}\kern+.2em}

\newcommand{\rot}{\mathsf{rot}}

\newcommand{\R}{\textup{Hom}}
\newcommand{\RH}{\textup{Hom}_{Hit}}
\newcommand{\RM}{\textup{Hom}_{max}}
\newcommand{\RP}{\textup{Hom}_{pos}}
\renewcommand{\O}{\textup{O}}

\newcommand{\orn}{\textup{or}}

\newcommand{\hyp}{\textup{Hyp}}

\renewcommand{\hom}{\textup{Hom}}
\newcommand{\homc}{\textup{Hom}_{\textup c}}
\newcommand{\homs}{\textup{Homeo}_+(S^1)}

\newcommand{\hommax}{\textup{Hom}_{\textup{max}}}
\newcommand{\homcs}{\textup{Hom}_\cs}
\newcommand{\Sym}{\textup{Sym}}
\newcommand{\out}{\textup{Out}}
\newcommand{\Diff}{\textup{Diff}}

\newcommand{\inn}{\textup{Inn}}
\newcommand{\map}{\textup{Map}}

\begin{document}

\title{Higher Teichm\"uller Spaces:\\ from $\SL(2,\RR)$ to other Lie groups}

\author{M.~Burger and  A.~Iozzi\thanks{Work partially
supported by the Swiss National Science Foundation project 2000021-127016/2} and A.~Wienhard\thanks{Work partially supported by NSF Grant No.DMS-0803216 and NSF CAREER Grant No. DMS-0846408}}

\address{
Department Mathematik\\ ETH Zentrum\\ R\"amistrasse 101, CH-8092 Z\"urich, Switzerland\\
email:\,\tt{burger@math.ethz.ch}
\\[4pt]
Department Mathematik\\ ETH Zentrum\\ R\"amistrasse 101, CH-8092 Z\"urich, Switzerland\\
email:\,\tt{iozzi@math.ethz.ch}
\\[4pt]
Department of Mathematics\\ Princeton University\\ Fine Hall, Washington Road, Princeton, NJ 08540, USA\\
email:\,\tt{wienhard@math.princeton.edu}
}

\maketitle
%


%


\setcounter{tocdepth}{2} 
\tableofcontents   

\section{Introduction}\label{sec:intro}

Let $S$ be a connected surface of finite topological type. The
Teichm\"uller space $\mathcal{T}(S)$ is the moduli space of marked
complex structures on $S$. It is isomorphic to the moduli space of
marked complete hyperbolic structures on $S$, sometimes called the
Fricke space\index{Fricke space} $\mathcal{F}(S)$. Associating to a hyperbolic structure 
its holonomy representation naturally embeds the Fricke space
$\mathcal{F}(S)$ into the variety of representations $\R
\big(\pi_1(S), \PSU(1,1) \big)/\PSU(1,1)$.

This realization of the classical Teichm\"uller space as a subset of the representations variety 
is the starting point of this article. 
We begin in \S~\ref{sec:hyp_str} by defining the space $\hyp(S)$ of
hyperbolic structures on $S$ and constructing in some details the map
\bqn \delta: \hyp(S) \to \R \big(\pi_1(S), \PSU(1,1) \big)\, , \eqn as
well as the embedding \bqn \delta': \mathcal{F}(S) =
\operatorname{Diff}^+_0(S) \backslash \hyp(S) \to \R \big(\pi_1(S),
\PSU(1,1) \big)/\PSU(1,1)\,, \eqn where $\operatorname{Diff}^+_0(S)$
is the group of orientation preserving diffeomorphisms which are
homotopic to the identity. Then \S\S~\ref{sec:rep_var} and
\ref{sec:inv_miln_gold} are devoted to various descriptions of the
subset $\delta \big(\hyp(S) \big) \subset \R \big(\pi_1(S), \PSU(1,1)
\big)$.  When $S$ is a compact surface, $\delta(\hyp(S))$ is:
\begin{enumerate}
\item\label{item1} the set of injective orientation preserving
  homomorphisms with discrete image (see Theorem~\ref{thm:1.5} and
  Corollary~\ref{cor:2.3});
\item\label{item2} identified with one connected component of $\R
  \big(\pi_1(S), \PSU(1,1) \big)$ (see \S~\ref{sec:rep_var});
\item\label{item3} the (maximal value) level set of numerical invariants described in \S~\ref{sec:inv_miln_gold}
 (see Theorem~\ref{sec:rep_var});
\item\label{item4} the solution set of a commutator equation (see (\ref{eq:goldman}));
\item\label{item5} characterized in terms of bounded cohomology classes (see Corollary~\ref{cor:4.5}).
\end{enumerate}

When $S$ is a noncompact surface of finite type, the description of
$\delta \big(\hyp(S) \big)$ is more involved, and the
characterizations (\ref{item1}) and (\ref{item2}) do not hold in this
case. In \S~\ref{sec:surf_fin_eul} we define suitable (bounded
cohomological) analogues of the numerical invariants described in
\S~\ref{sec:inv_miln_gold}, which allow us to give in
\S~\ref{sec:hyp_str_nc} characterizations of $\delta \big(\hyp(S)
\big)$ for noncompact surfaces $S$, generalizing (\ref{item3}),
(\ref{item4}) and (\ref{item5}) above.

In the second part we ask how much of this ``$\PSU(1,1)$ picture''
generalizes to an arbitrary Lie group $G$. We discuss two classes of
(semi)simple Lie groups for which one can make this question precise
by defining, in very different ways, components (or specific subsets
when $S$ is not compact) of $\hom\big(\pi_1(S),G\big)$ which play the
role of Teichm\"uller space.

The terminology Higher Teichm\"uller spaces,
coined by Fock and Goncharov, has now come to mean subsets of $\R \big (\pi_1(S), G
\big)$, where $G$ is a simple Lie group, which share essential
geometric and algebraic properties with classical Teichm\"uller space
considered as a subset of $\R \big(\pi_1(S), \PSU(1,1) \big)$. Up to
now higher Teichm\"uller spaces are defined for two classes of Lie
groups, namely for split real simple Lie groups, e.g. $\SL(n,\RR)$,
$\Sp(2n,\RR)$, $\SO(n,n+1)$ or $\SO(n,n)$ and for Lie groups of
Hermitian type, e.g. $\Sp(2n,\RR)$, $\SO(2,n)$, $\SU(p,q)$ or
$\SO^*(2n)$.

The invariants defined in \S~\ref{sec:inv_miln_gold} and
\S~\ref{sec:surf_fin_eul} can be defined for homomorphisms from
$\pi_1(S)$ with values in any Lie group $G$ but, when $G$ is a Lie group
of Hermitian type these invariants are particularly meaningful. We describe
how the basic objects available in the case of $\PSU(1,1)$
can be generalized to higher rank in \S~\ref{sec:max_rep}. Considering
the maximal value level set of the numerical invariant thus
constructed leads us to consider the space of maximal representations
\bqn
\RM \big(\pi_1(S), G \big) \subset \R \big(\pi_1(S), G \big)\,,
\eqn 
some of whose properties are discussed in \S~\ref{sec:max_rep}.  In
particular we state a result ("structure theorem") which describes the
Zariski closure in $G$ of the image of a maximal representation; a
major part of \S\S~\ref{subsec:tight} 
\ref{subsec:bdry_rot_repvar} then offers a guided tour showing the
various aspects of the proof of the structure theorem. \medskip

The space of maximal representations is an example of a higher
Teichm\"uller space\index{Higher Teichm\"uller space}.  Hitchin components and spaces of positive
representations are other examples of a higher Teichm\"uller space,
defined when $G$ is a split real Lie group. We review the definition
of these spaces shortly (\S~\ref{subsec:hitchin_positive}), and then
discuss important structures underlying both families of higher
Teichm\"uller spaces (\S~\ref{sec:higher}). In the case of compact
surfaces the quest for common structures leads us to consider the
concept of Anosov structures (\S~\ref{subsec:anosov}). This more
general notion provides an important framework within which one can
start to understand the geometric significance of higher Teichm\"uller
spaces, their quotients by the mapping class group\index{mapping class group}
(\S~\ref{subsec:quotients}), the geometric structures parametrized by
higher Teichm\"uller spaces (\S~\ref{subsec:geom_struc}), as well as
further topological invariants (\S~\ref{subsec:invariants}).  In
\S~\ref{subsec:questions} we conclude by mentioning further directions
of study.  \medskip

There are many aspects of higher Teichm\"uller spaces which we do not
touch upon for lack of space and expertise.  One of them is the huge
body of work studying maximal representations from the point of view
of Higgs bundles; such techniques give in particular precise
information about the number of connected components of maximal
representations, as well as information on the homotopy type of those
components (see for example \cite{Biswas_etal,
  Bradlow_GarciaPrada_Gothen, Bradlow_GarciaPrada_Gothen_survey,
  Bradlow_GarciaPrada_Gothen_sp4, GarciaPrada_Gothen_Mundet,
  GarciaPrada_Mundet, Gothen}).

As a guide to the reader, we mention that the first part of the paper
and the discussion of maximal representations is very descriptive and
should be read linearly, while starting from the definition of Hitchin
representations the paper is a pure survey.
 
\medskip {\bf Acknowledgments:} We thank A.~Papadopoulos for
undertaking this project. We thank also
O.~Guichard and T.~Hartnick for carefully reading a preliminary
version of this paper providing many helpful comments, W.~Goldman for
many bibliographical comments and F.~Labourie for positive feedback.
Our thanks go also to D.~Toledo and N.~A'Campo for useful discussions
on this general topic over the years and to T.~Delzant for helpful
comments concerning central extensions of surface groups. Finally, 
we thank the referee for helpful comments. \medskip

\part{Teichm\"uller Space and Hyperbolic Structures}
\section{Hyperbolic structures and representations}

\subsection{Hyperbolic structures}\label{sec:hyp_str}
In this section we review briefly how one associates to a hyperbolic
structure\index{hyperbolic structure} on a surface a homomorphism of
its fundamental group into the group of orientation preserving
isometries of the Poincar\'e disk, and how an appropriate quotient of
the set of hyperbolic structures injects into the representation
variety.  \medskip

Let $\DD=\{z\in\CC:\,|z|<1\}$ be the unit disk endowed with the
Poincar\'e metric $\frac{4|dz|^2}{\big(1-|z|^2\big)^2}$, and let 
$G:=\PSU(1,1)=\SU(1,1)/\{\pm\id\}$ 
denote the quotient of $\SU(1,1)$ by its center.
The group $G$ acts on $\widehat\CC=\CC\cup\{\infty\}$ by linear fractional
transformations preserving $\DD$ and hence can be identified  
with the group of orientation preserving isometries of $\DD$.

Given a surface\index{surface} $S$, that is a two-dimensional smooth
manifold\footnote{Note that all manifolds here are without boundary.
  In particular a compact surface is necessarily closed.}  a
hyperbolic metric on $S$ is a Riemannian metric with sectional
curvature $-1$.  A $(G,\DD)$-structure\index{$(G,\DD)$-structure} on
$S$ is an atlas on $S$ consisting of charts taking values in $\DD$,
whose change of charts are locally restrictions of elements of $G$,
\cite{Thurston_book}.
Assuming from now on that $S$ is orientable, we observe that by the
local version of Cartan's theorem, an orientation together with a
hyperbolic metric on $S$ determines a $(G,\DD)$-structure on $S$ (the
converse is also true and straightforward). Also, the hyperbolic
metric is complete if and only if the same is true for the
corresponding $(G,\DD)$-structure, i.e. if the developing map
$\widetilde S\to\DD$ is a diffeomorphism.

The group $\operatorname{Diff}(S)$, and hence its subgroup
$\operatorname{Diff}^+(S)$ consisting of orientation preserving
diffeomorphisms of $S$, act on the set $\hyp(S)$ of complete
hyperbolic metrics on $S$ in a contravariant way.  In the sequel let
$\widetilde{S}=D$ be a smooth oriented disk with a basepoint $\ast\in
D$ and let us fix once and for all a base tangent vector $v\in T_\ast
D$, $v\neq0$.  By the correspondence between hyperbolic metrics and
$(G,\DD)$-structures, let us also consider, for every $h\in\hyp(D)$, 
the unique orientation preserving isometry 
\bqn
f_h:(D,\ast)\to(\DD,0) 
\eqn 
with $df_h(v)\in\RR^+e$,
where $e=1\in\CC$. If $\phi\in\operatorname{Diff}^+(D)$, then for any
$h\in\hyp(D)$, $\varphi $ is by definition an orientation preserving
isometry between the hyperbolic metrics $\varphi^\ast(h)$ and $h$.
Therefore \bqn c(\phi,h):=f_h\circ\phi\circ f_{\phi^\ast(h)}^{-1} \eqn
is an element of $G$.  In this way we obtain a map \bqn
c:\operatorname{Diff}^+(D)\times\hyp(D)\to G \eqn which verifies the
cocycle relation \bqn
c(\phi_1\phi_2,h)=c(\phi_1,h)c\big(\phi_2,\phi_1^\ast(h)\big)\,. \eqn
Let now $(S,\ast)$ be a connected oriented surface with base point
$\ast$ and assume that $\hyp(S)\neq\emptyset$.  Let $(\widetilde
S,\ast)=(D,\ast)$, let $p:D\to S$ be the canonical projection and \bqn
\Gamma=\{T_\gamma:\,\gamma\in\pi_1(S,\ast)\}<\operatorname{Diff}^+(D)
\eqn the group of covering transformations. Then the pullback via $p$
gives a bijection between $\hyp(S)$ and the set $\hyp(D)^\Gamma$ of
$\Gamma$-invariant elements in $\hyp(D)$. Furthermore, it follows from
the cocycle identity that, for every $h\in\hyp(D)^\Gamma$, the map
\bqn 
\ba
\rho_h:\pi_1(S,\ast)&\longrightarrow\,\, G\\
\gamma\quad&\mapsto c(T_\gamma,h) 
\ea 
\eqn 
is a homomorphism with
respect to which the isometry $f_h$ is equivariant. Thus we obtain the
map $\delta$, assigning to a hyperbolic structure its holonomy
homomorphism\index{holonomy homomorphism} 
\bqn 
\ba
\delta: \, \hyp(S)&\to\R\big(\pi_1(S,\ast),G\big)\\
h\quad&\longmapsto\qquad\rho_{p^\ast(h)}\,, 
\ea 
\eqn 
which has certain
important equivariance properties which we now explain.

To this end, let $\Nn^+$ be the normalizer of $\Gamma$ in
$\operatorname{Diff}^+(D)$. Then we have the diagram with exact line
\bqn 
\xymatrix{ \{e\}\ar[r] &\Gamma \ar[r] &\Nn^+\ar[r]^-\pi\ar[d]^a
  &\operatorname{Diff}^+(S)\ar[r]
  &\{e\}\\
  & &\aut\big(\pi_1(S,\ast) \big)\,,& & 
} 
\eqn 
where $\pi$ associates to
every $\phi\in\Nn^+$ the corresponding diffeomorphism of $S$ obtained
by observing that $\phi$ permutes the fibers of $p$; the fact that
$\pi$ is surjective follows from covering theory. The homomorphism $a$
is the one associating to $\phi$ the automorphism $a_\phi$ of $\Gamma$, or
rather of $\pi_1(S,\ast)$, obtained by conjugation.  With these
definitions, a computation gives \bq\label{eq:1.1}
\rho_{\phi^\ast(h)}(\gamma)=c(\phi,h)^{-1}\rho_h\big(a_\phi(\gamma)\big)c(\phi,h)\,,
\eq for every $\phi\in\Nn^+$, $h\in\hyp(D)^\Gamma$ and $\gamma\in
\Gamma$.

In view of \eqref{eq:1.1}, it is important to determine those $\phi\in\Nn^+$ 
such that $a_\phi$ is an inner automorphism of $\Gamma$.
Let $\ninn$ be the subgroup consisting of all such diffeomorphisms.  
Then we have:
\begin{lemma}\label{lem;1.1}
The map $\pi$ induces an isomorphism
\bqn
\Gamma\backslash\ninn\cong\operatorname{Diff}^+_0(S)\,,
\eqn
where $\operatorname{Diff}^+_0(S)$ is the subgroup of $\operatorname{Diff}^+(S)$ 
consisting of those diffeomorphisms
which are homotopic to the identity.
\end{lemma}
\begin{proof}  If $f:S\to S$ is homotopic to the identity,
then by covering theory the conjugation of $\Gamma$ by any lift $\tilde f$ of $f$ gives
an inner automorphism of $\Gamma$.  

Conversely, if $\phi T_\gamma\phi^{-1}=T_{\eta\gamma\eta^{-1}}$ 
for some $\eta$ and all $\gamma$, then the diffeomorphism
$T_\eta^{-1}\phi:D\to D$ commutes with the $\Gamma$-action on $D$;
if we fix $h\in\hyp(D)^\Gamma$ then the geodesic homotopy 
from $T_\eta^{-1}\phi$ to $\id_D$
is $\Gamma$-equivariant and hence descends to a homotopy between
$\pi(T_\eta^{-1}\phi)=\pi(\phi)$ and $\id_S$.
\end{proof}

Thus combining the inverse of $\pi$ with $a$ we obtain an injective
homomorphism \bqn \xymatrix{
  \map(S):=\operatorname{Diff}^+_0(S)\backslash\operatorname{Diff}^+(S)\ar[r]^-\alpha&\out\big(\pi_1(S,\ast)\big)\,
} \eqn of the {\em mapping class group}\index{mapping class group} $\map(S)$ of $S$
into the group of outer automorphisms of $\pi_1(S)$. It follows then
from \eqref{eq:1.1} that the map which to $h\in\hyp(S)$ associates the
class of the homomorphism $[\rho_{p^\ast(h)}]$ and which takes values
in the quotient $\R\big(\pi_1(S,\ast),G\big)/G$ by the $G$-conjugation
action on the target, is invariant under the
$\operatorname{Diff}^+_0(S)$-action so that finally we obtain a map
from the {\em Fricke space}\index{Fricke space}  $\mathcal{F}(S) =
\operatorname{Diff}^+_0(S)\backslash\hyp(S)$ to the representation
variety\index{representation variety} 
\bqn 
\ba
\delta':\operatorname{Diff}^+_0(S)\backslash\hyp(S)&\to\R\big(\pi_1(S,\ast),G\big)/G\\
[h]\qquad\quad&\longmapsto\quad\quad[\rho_{p^\ast(h)}] 
\ea 
\eqn 
which is $\alpha$-equivariant.

\begin{proposition}\label{prop:1.2}  If $S$ is a connected, oriented surface 
admitting a complete hyperbolic structure, then $\delta'$ is injective. 
\end{proposition}

\begin{proof}  If $h_1,h_2\in\hyp(D)^\Gamma$ are such that $\rho_{h_1}$ and $\rho_{h_2}$
are conjugated by $g\in G$, then it follows from the definitions that 
$f_{h_2} ^{-1} gf_{h_1}$ is an orientation preserving diffeomorphism sending $h_1$ to $h_2$,
which furthermore is $\Gamma$-equivariant; by the argument used in Lemma~\ref{lem;1.1},
we get that $\pi(f_{h_2} ^{-1}gf_{h_1})\in\operatorname{Diff}^+_0(S)$.
\end{proof}

\medskip

We now describe the image of the homomorphism $\alpha$ and of the map
$\delta'$ in the case in which $S$ is a compact oriented surface of
genus $g\geq2$.  This latter condition guarantees that the classifying
map 
\bq\label{eq:1.2}
 S\to B\pi_1(S,\ast) 
\eq 
is a homotopy
equivalence; we use this fact to equip
$\h_2\big(\pi_1(S,\ast),\ZZ\big)$ with the canonical generator, image
of the fundamental class $[S]$ via the isomorphism
$\h_2(S,\ZZ)\to\h_2\big(\pi_1(S,\ast),\ZZ\big)$ induced by
\eqref{eq:1.2}. An isomorphism between the fundamental groups of two
compact oriented surfaces $S_1$ and $S_2$ is said to be orientation
preserving if the generator of $\h_2\big(\pi_1(S_1,\ast),\ZZ\big)$ is
mapped to the generator of $\h_2\big(\pi_1(S_2,\ast),\ZZ\big)$.

\begin{theorem}\label{thm:1.3}  
Let $S_1$ and $S_2$ be compact oriented surfaces of genus $g\geq1$.  
Then any orientation preserving isomorphism $\pi_1(S_1,\ast)\to\pi_1(S_2,\ast)$
is induced by an orientation preserving diffeomorphism $S_1\to S_2$.
\end{theorem}

Let us denote by $\aut^+\big(\pi_1(S,\ast)\big)$
the group of the orientation preserving automorphisms  
of a compact orientable surface $S$ of genus $g\geq1$, and
by $\out^+\big(\pi_1(S,\ast)\big)$  its quotient by the group of  inner automorphisms.  
From Theorem~\ref{thm:1.3} we conclude:

\begin{corollary}[Dehn-Nielsen--Baer Theorem\index{Dehn--Nielsen--Baer Theorem}\index{Theorem!Dehn--Nielsen--Baer}, see \cite{Farb_Margalit} for a proof]\label{cor:1.4}  If $S$ is a compact orientable surface of genus $g\geq 2$, 
  the map \bqn \alpha:\map(S) =
  \operatorname{Diff}^+(S)/\operatorname{Diff}^+_0(S)\to\out^+\big(\pi_1(S,\ast)\big)
  \eqn is an isomorphism.
\end{corollary}

Let now $\R_{d,i}$ denote the subset of $\R$ consisting of 
all injective homomorphisms with discrete image.  
The following classical identification of the image of $\delta$ uses the Nielsen realization:

\begin{theorem}\label{thm:1.5}  If $S$ is a compact orientable surface of genus $g\geq1$, 
then the image of $\delta$ consists precisely of 
\bqn
\left\{
\rho\in\R_{d,i}\big(\pi_1(S,\ast),G)\big):\,
\im\rho\backslash\DD\text{ is compact and }
\rho\text{ is orientation preserving}
\right\}\,.
\eqn
\end{theorem}
\begin{remark}
  Since $\pi_1(S)$ is the fundamental group of a compact surface and
  $\rho$ is a discrete embedding, the quotient $\im\rho\backslash\DD$
  is automatically compact. Here, we include this property explicitly
  in order to stress the similarity with the definition of
  $\hom_0(\Gamma,G)$ in \S~\ref{sec:rep_var}.
  
  As above, the representation $\rho$ is orientation preserving if the
  induced map $\rho_*$ maps the generator of
  $\h_2\big(\pi_1(S,\ast),\ZZ\big)$ to the generator of
  $\h_2\big(\pi_1(\im\rho\backslash\DD, \ast) ,\ZZ\big)$.

\end{remark}

\begin{proof} Given $h\in\hyp(S)$, it is immediate, by using $f_h$, that $\rho_h$ belongs
to the above set.

Conversely, apply Nielsen realization to the orientation preserving isomorphism $\rho$ 
to get an orientation preserving diffeomorphism $f:S\to\im\rho\backslash \DD$;
if $h=f^\ast(h_P)$, where $h_P$ is the Poincar\'e metric on $\im\rho\backslash\DD$,
then one verifies that $[\rho]=[\rho_h]$.
\end{proof}

\begin{remark}\label{rem:2.6}
  Contrary to what happens in the compact case, if $S$ is a
  noncompact orientable surface of negative Euler characteristic, the
  inclusion
 \bqn
\delta\big(\hyp(S)\big)  \cup \delta\big(\hyp(\overline{S})\big)\subset\R_{d,i}\big(\pi_1(S, \ast), \PSU(1,1)\big)\,, 
\eqn
where $\overline{S}$ denotes the surface $S$ with the opposite orientation 
is always proper. 
In fact, if $\rho:\pi_1(S,\ast)\to G$ is just discrete and injective, 
the surfaces $\im\rho\backslash\DD$ and $S$ need not be diffeomorphic,
although they have the "same" fundamental group.  
For example, the once punctured torus and
the thrice punctured sphere have isomorphic fundamental groups $\FF_2$
and admit complete hyperbolic structures.
We will see in \S~\ref{sec:surf_fin_eul} one way to remedy this problem.




\end{remark}

\subsection{Representation varieties} \label{sec:rep_var}
In this section we review some basic properties of the set of discrete
and faithful representations in $\hom(\Gamma, G)$ in the general
context of a finitely generated group $\Gamma$ and a
connected reductive Lie group $G$.

One way to approach the problem of determining the image of $\hyp(S)$
under the map $\delta$ is to equip $\hom(\Gamma,G)$ with a topology.
Quite generally if $\Gamma$ is a discrete group and $G$ is a
topological group, $\hom(\Gamma,G)$ inherits the topology of the
product space $G^\Gamma$. In case $\Gamma$ is finitely generated with
finite generating set $F\subset\Gamma$, let $p:\FF_{|F|} \to\Gamma$ be the
corresponding presentation and $R$ a set of generators of the relators
$\ker\, p$. Since every $r\in R$ is a word in $\FF_{|F|}$, it
determines a product map $m_r:G^F\to G$ by evaluation on $G$. The map
\bqn
 \ba
\hom(\Gamma,G)&\longrightarrow\, G^F\\
\pi\qquad&\mapsto\big(\pi(s)\big)_{s\in F} 
\ea 
\eqn 
identifies the topological space $\hom(\Gamma,G)$ with the closed subset 
$\cap_{r\in
  R}m_r^{-1}(e)\subset G^F$. In particular, $\hom(\Gamma,G)$ is
locally compact if $G$ is so, and a real algebraic set if $G$ is a
real algebraic group. We record the following

\begin{proposition}[\cite{Benedetti_Risler, Whitney}]\label{prop:2.1} If $\Gamma$ is finitely generated and 
$G$ is a real algebraic group,
then $\hom(\Gamma,G)$ has finitely many connected components and 
each of them is a real semialgebraic set,
that is, it is defined by a finite number of polynomial equations and inequalities.
\end{proposition}

\begin{remark}
  Proposition~\ref{prop:2.1} fails if $G$ is not algebraic. An example
  for this, given in \cite{Goldman_Hirsch}, is the quotient of the three
  dimensional Heisenberg group by a cyclic central subgroup, where a
  simple obstruction class detects infinitely many connected
  components in the representation variety.
\end{remark}

In order to proceed further we assume that $\Gamma$ is finitely generated, 
$G$ is a Lie group and introduce 
(see \cite{Goldman_Millson, Weil}) the following subset of $\hom(\Gamma,G)$
\bqn
\hom_0(\Gamma,G)=\{\rho\in\hom_{d,i}(\Gamma,G)\hbox{ such that }
\rho(\Gamma)\backslash G \hbox{ is compact}\}\,,
\eqn
where, as in \S~\ref{sec:hyp_str}, $\hom_{d,i}$ refers to the set of injective homomorphisms
with discrete image, 
so that $\hom_0(\Gamma,G)\subset\hom_{d,i}(\Gamma,G)\subset\hom(\Gamma,G)$.

The first result on the topology of $\hom_{d,i}(\Gamma,G)$ requires a hypothesis on $\Gamma$:
\begin{definition}
We say that $\Gamma$ has property $(\h)$ 
if every normal nilpotent subgroup of $\Gamma$ is finite.
\end{definition}

Observe that this condition is fulfilled 
by every nonabelian free group and every fundamental group 
of a compact surface of genus $g\geq2$.  With this we can now state the following

\begin{theorem}[\cite{Goldman_Millson}]\label{thm:2.2}
Let $\Gamma$ be a finitely generated group with property $(\h)$ and $G$ a connected Lie group.
Then $\hom_{d,i}(\Gamma,G)$ is closed in $\hom(\Gamma,G)$.
\end{theorem}
\begin{proof} The essential ingredient is 
the theorem of Kazhdan--Margulis--Zassenhaus \cite[Theorem~8.16]{Raghunathan_book} 
saying that 
there exists an open neighborhood $\Uu\subset G$ of $e$ such that 
whenever $\Lambda<G$ is a discrete subgroup,
then $\Uu\cap\Lambda$ is contained in a connected nilpotent group. 
We fix now such an open neighborhood and assume in addition that 
[it does not contain any nontrivial subgroup of $G$; 
let also $\ell$ be an upper bound on the degree of
nilpotency of connected nilpotent Lie subgroups of $G$.

Let now $\{\rho_n\}_{n\geq1}$ be a sequence in $\hom_{d,i}(\Gamma,G)$ with limit $\rho$.
We show that $\rho$ is injective.
For every finite set $E\subset\ker\rho$, we have that 
$\rho_n(E)\subset\Uu$ for $n$ large, which, by \cite[Theorem~8.16]{Raghunathan_book}, implies 
that for all $k \geq \ell$, the $k$-th iterated commutator of $\rho_n(E)$ is trivial,
and the same holds therefore for $E$ since $\rho_n$ is injective.

As a result, $\ker\rho$ is nilpotent and hence, by property $(\h)$, finite; 
thus $\rho_n(\ker\rho)\subset\Uu$ for large $n$ which, by the choice of $\Uu$ implies 
that $\rho_n(\ker\rho)=e$ and hence that $\rho$ is injective.  

We prove now that $\rho$ is discrete.  To this end, let $L=\overline{\rho(\Gamma)}$
be the closure of $\rho(\Gamma)$;
then $L$ is a Lie subgroup of $G$ and $L^0$ is open in $L$.  
Let $V$ be an open neighborhood of the identity on $L^0$ with $V\subset\Uu$;
then $\rho(\Gamma)\cap V$ is dense in $V$ and $V$ generates $L^0$, 
from which we conclude that $\rho(\Gamma)\cap V$ generates a dense subgroup of $L^0$.
For every finite set $F'\subset\Gamma$ with $\rho(S)\subset V\subset\Uu$ we have that 
$\rho_n(S)\subset\Uu$ for $n$ large, which implies as before that for all $k\geq l$
the $k$-th iterated commutator of $F'$, and hence of $\rho(F')$ is trivial;
thus $L^0$ is nilpotent and so is $\rho^{-1}(L^0)$ since $\rho$ is injective.
But then $\rho^{-1}(L^0)$ is finite and hence $L^0=\{e\}$,
which shows that $\rho(\Gamma)$ is discrete and concludes the proof.
\end{proof}

Next we turn to the set $\hom_0(\Gamma,G)$ of faithful, 
discrete and cocompact realizations of $\Gamma$ in $G$;
this set was considered by A.~Weil as a tool in his celebrated local rigidity theorem 
in which the following general result played an important role.

\begin{theorem}[\cite{Raghunathan_book}]\label{thm:2.3}
Assume that $\Gamma$ is finitely generated and that $G$ is a connected Lie group. 
Then $\hom_0(\Gamma,G)$ is an open subset of $G$.
\end{theorem}

There are by now several approaches available: we refer to the paper by Bergeron and Gelander
\cite{Bergeron_Gelander}, where the geometric approach due essentially to 
Ehresmann and Thurston \cite{Thurston_book} is explained (see also \cite{Goldman_structures,Canary_etal, Lok});
this approach, based on a reformulation of the problem in terms of variations of 
$(G,X)$-structures leads to a more general stability result 
also valid for manifolds with boundary.

We content ourselves with noticing the following consequence:

\begin{corollary}\label{cor:2.3}  Assume that $\Gamma$ is  finitely generated, torsion-free and 
  has property $(\h)$.  Assume that $G$ is a connected reductive Lie
  group and that there exists a discrete, injective and cocompact
  realization of $\Gamma$ in $G$.  Then \bqn
  \hom_0(\Gamma,G)=\hom_{d,i}(\Gamma,G) \eqn and both sets are
  therefore open and closed, in particular a union of connected
  components of $\hom(\Gamma,G)$.
\end{corollary}
\begin{proof}  Let $\rho_0\in\hom_0(\Gamma,G)$ and let $K<G$ be a maximal compact subgroup;
since by the Iwasawa decomposition $X:=G/K$ is contractible 
and since  $\rho_0(\Gamma)$ acts on $X$ as a group of covering transformations
with compact quotient, we have that, for $n=\dim  X$, 
$\h^n(\Gamma,\RR)=\h^n\big(\rho_0(\Gamma),\RR\big)\neq0$.
Therefore, if $\rho:\Gamma\to G$ is any discrete injective embedding, 
we have that $\h^n\big(\rho(\Gamma)\backslash X,\RR\big)$ does not vanish
and hence $\rho(\Gamma)\backslash X$ is compact, 
thus implying that $\rho\in\hom_0(\Gamma,G)$.
\end{proof}


Applying the preceding discussion to our compact surface $S$ of genus $g\geq2$,
we conclude using Theorem~\ref{thm:1.5} and Corollary~\ref{cor:2.3} that 
\bqn
\delta\big(\hyp(S)\big)\subset\hom\big(\pi_1(S,\ast),\PSU(1,1)\big)
\eqn
is a union of components of the representation variety $\hom\big(\pi_1(S,\ast),\PSU(1,1)\big)$.

\section{Invariants, Milnor's inequality and Goldman's theorem}\label{sec:inv_miln_gold}
In this section we will discuss various aspects of a fundamental invariant 
attached to a representation
$\rho:\pi_1(S,\ast)\to G$, where $G=\PSU(1,1)$, namely the {\em Euler number}\index{Euler number}\index{number!Euler} 
of $\rho$.
This leads to a quite different way of characterizing the image of
\bqn
\delta:\hyp(S)\to\hom\big(\pi_1(S,\ast),G\big)
\eqn
in the case in which $S$ is compact.  This invariant can also be
defined for targets belonging to a large class of Lie groups $G$, and
this leads to natural generalizations of Teichm\"uller space (see the
discussion in \S~\ref{sec:max_rep} and \S~\ref{sec:higher}).

In the sequel $S$ denotes a compact surface of genus $g\geq 2$ and fixed orientation.  
We drop moreover the basepoint in the notation $\pi_1(S,\ast)$ and we set $D=\widetilde S$.

\subsection{Flat $G$-bundles\index{flat $G$-bundles}}\label{subsec:flat_gb} Given a connected Lie group 
$G$ and a homomorphism
$\rho:\pi_1(S)\to G$, we obtain, in the notation of
\S~\ref{sec:hyp_str}, a proper action without fixed points on $D\times
G$ by \bqn \gamma_\ast(x,g)=\big(T_\gamma x,\rho(\gamma)g\big) \eqn
whose quotient $\pi_1(S)\backslash (D\times G)$ is the total space
$G(\rho)$ of a flat principal (right) $G$-bundle over $S$, where the
projection map comes from the projection $D\times G\to D$ on the first
factor.

%
Given a $G$-bundle $\mathcal{E}$ over $S$ the first obstruction to find a continuous section of 
$\mathcal{E} \to S$ lies in $\h^2\big(S,\pi_1(G)\big)$.
Namely, let $K$ be a triangulation of $S$; choose preimages in  $\mathcal{E}$ 
for the vertices of $K$
and extend this section over the $0$-skeleton of $K$ to the $1$-skeleton 
by using that $G$ is connected;
for each $2$-simplex $\sigma$ we have thus a section over its boundary $\partial\sigma$.
Using the flat connection, this section of  $\mathcal{E}$ 
over $\partial\sigma$ can be deformed
into a loop lying in a single fixed fiber;
identifying this fiber with $G$ we get for every $\sigma$ a free homotopy class of loops in $G$
and hence a well defined element $c(\sigma)\in\pi_1(G)$, since the latter is Abelian.  
The map $c$ is a simplicial $2$-cocycle on $K$ with values in $\pi_1(G)$ and hence defines 
an element in $\h^2\big(S,\pi_1(G)\big)$ which depends only on $\rho$.
In this way we obtain a map
\bqn
o_2:\hom\big(\pi_1(S),G\big)\to\h^2\big(S,\pi_1(G)\big)\, 
\eqn
which assigns to $\rho$ the obstruction\index{obstruction class} 
$o_2(\rho) \in \h^2\big(S,\pi_1(G)\big)$ of the flat $G$-bundle 
$G(\rho)$. 
An important observation is that if $\rho_1$ and $\rho_2$ lie in the same component
of $\hom\big(\pi_1(S),G\big)$, then the associated $G$-bundles $G(\rho_1)$ and $G(\rho_2)$
are isomorphic.  As a result, the invariant $o_2$ is constant 
on connected components of $\hom\big(\pi_1(S),G\big)$.
(See also \cite{Goldman_82} for a discussion of characteristic classed and representations.)

\subsection{Central extensions}\label{subsec:centr_ext}
An invariant closely related to the one defined above is obtained by considering
the central extension\index{central extension} of $G$ given by the universal covering
\bqn
\xymatrix{
 \{e\}\ar[r]
&\pi_1(G)\ar[r]
&\widetilde G\ar[r]^p
&G\ar[r]
&\{e\}\,,
}
\eqn
where the neutral element is taken as basepoint.

A homomorphism $\rho:\pi_1(S)\to G$ then gives a central extension $\Gamma_\rho$ of $\pi_1(S)$
by $\pi_1(G)$ in the familiar way
\bqn
\Gamma_\rho=\big\{(\gamma,g)\in\pi_1(S)\times\widetilde G:\,\rho(\gamma)=p(g)\big\}\,.
\eqn
Observing now that the isomorphism classes of central extensions of $\pi_1(S)$ by $\pi_1(G)$
are classified by $\h^2\big(\pi_1(S),\pi_1(G)\big)$, we get  a map
\bqn
c_2:\hom\big(\pi_1(S),G\big)\to\h^2\big(\pi_1(S),\pi_1(G)\big)\,.
\eqn
So far the discussion in \S\S~\ref{subsec:flat_gb} and \ref{subsec:centr_ext} 
applies to any connected Lie group $G$.
In case $G=\PSU(1,1)$, we get a canonical generator of $\pi_1(G)$ 
from the orientation of $\DD\subset \CC$;
by considering the loop
\bqn
\ba
{[}0,1]&\to\quad\PSU(1,1)\\
s\,\,\,&\mapsto \begin{pmatrix}e^{i\pi s}&0\\0&e^{-i\pi s}\end{pmatrix}\,,
\ea
\eqn
we identify $\pi_1\big(\PSU(1,1)\big)$ with $\ZZ$; 
we will denote by $t\in\pi_1\big(\PSU(1,1)\big)$ the image of $1\in\ZZ$.

\subsection{Description of $\h^2\big(\pi_1(S),\ZZ\big)$, a digression}\label{subsec:descr}
Let $g\geq1$ be the genus of $S$. Then $\pi_1(S)$ admits as presentation
\bqn
\left\<a_1,b_1,\dots,a_g,b_g:\,\prod_{i=1}^g[a_i,b_i]=e\right\>\,.
\eqn
The orientation of $S$ is built in, in that, when drawing the lifts to $\widetilde S$ of the loop
$a_1b_1a_1^{-1}b_1^{-1}a_2b_2\dots a_g^{-1}b_g^{-1}$, one gets a $4g$-gone
whose boundary is traveled through in the positive sense.

Now define
\bqn
\overline\Gamma_g:=
\left\<A_1,B_1,\dots,A_g,B_g,z:\,\prod_{i=1}^g[A_i,B_i]=z\hbox{ and }[z,A_i]=[z,B_i]=e
\right\>\,.
\eqn
This group $\overline\Gamma_g$ surjects onto $\pi_1(S)$ with kernel the cyclic subgroup
generated by $z$, which incidentally is central.  
In order to see that $z$ has infinite order, 
observe first that $\overline\Gamma_1$ is isomorphic to the integer Heisenberg group
\bqn
\left\{
\begin{pmatrix}
1&x&z\\
0&1&y\\
0&0&1
\end{pmatrix}:\,
x,y,z\in\ZZ
\right\}
\eqn
by 
\bqn
A_1\mapsto\begin{pmatrix}1&1&0\\0&1&0\\0&0&1\end{pmatrix}
\qquad
B_1\mapsto\begin{pmatrix}1&0&0\\0&1&1\\0&0&1\end{pmatrix}
\qquad
z\mapsto\begin{pmatrix}1&0&1\\0&1&0\\0&0&1\end{pmatrix}\,.
\eqn
Then conclude by considering the surjection 
\bqn
\overline\Gamma_g\to\overline\Gamma_1
\eqn
obtained by sending $A_i$ and $B_i$ to $e$ for $i\geq2$.
Thus $\overline\Gamma_g$ gives a central extension of $\pi_1(S)$ by $\ZZ$;
denoting by $\big[\overline\Gamma_g\big]$ its image in $\h^2\big(\pi_1(S_g),\ZZ\big)$, 
we have the following

\begin{proposition}\label{prop:3.1} $\h^2\big(\pi_1(S_g),\ZZ\big)=\ZZ\big[\overline\Gamma_g\big]$.
\end{proposition}

In fact if
\bqn
\xymatrix{
   0\ar[r]
 &\ZZ\ar[r]^i
 &\Lambda\ar[r]
 &\pi_1(S_g)\ar[r]
 &\{e\}
 }
 \eqn
 is any central extension by $\ZZ$, take lifts $\alpha_j,\beta_j\in\Lambda$ of $a_j, b_j$: then
 $\prod_{j=1}^g[\alpha_j,\beta_j]$ is independent of all 
 choices and the image under $i$ of a well defined $n\in\ZZ$.  
 Using the Baer product of extensions \cite{Baer} 
 one shows, by recurrence on $n$, that
 \bqn
 [\Lambda]=n\,[\overline\Gamma_g]\,.
 \eqn 
Now we come back to the invariant associated in \S~\ref{subsec:centr_ext} 
to $\rho\in\hom\big(\pi_1(S),G\big)$ and use the identification
\bqn
\ba
\pi_1(G)&\to\ZZ\\
t\quad&\mapsto 1
\ea
\eqn
to get $c_2(\rho)\in\h^2\big(\pi_1(S),\ZZ\big)$.
In terms of central extension we have then that 
\bqn
c_2(\rho)=z_2(\rho)[\overline\Gamma_g]\,,
\eqn
where $z_2(\rho)\in\ZZ$ is defined by the formula 
\bqn
\prod_{i=1}^g[\alpha_i,\beta_i]=t^{z_2(\rho)}\,,
\eqn
where $\alpha_1,\beta_1,\dots,\alpha_g,\beta_g\in\widetilde G$ are lifts of 
$\rho(a_1),\rho(b_1),\dots,\rho(a_g),\rho(b_g)$.

\subsection{The Euler class\index{Euler class}}\label{subsec:euler}
The following discussion is specific to the case where $G=\PSU(1,1)$; it takes as point
of departure the observation that the action of $G$ by homographies on $\DD$ gives an action 
on the circle $\partial\DD$
bounding $\DD$, by orientation preserving homeomorphisms. 
It will become apparent that considering homomorphisms with values in $G$ 
as homomorphisms with target the group $\homs$ of orientation preserving homeomorphisms
of the circle gives additional flexibility.

Let us take the quotient $\ZZ\backslash\RR$ of $\RR$ by the group 
generated by the integer translations
$T(x)=x+1$ of the real line, and consider, as model of $S^1$,
\bqn
\Hh_\ZZ^+(\RR)=\{f:\RR\to\RR:\,\hbox{increasing homeomorphisms commuting with }T\}\,.
\eqn
We obtain then the central extension 
\bqn
\xymatrix{
 0\ar[r]
 &\ZZ\ar[r]
 &\Hh^+_\ZZ(\RR)\ar[r]^-p
 &\homs\ar[r]
 &0
 }
\eqn
which realizes $\Hh^+_\ZZ(\RR)$ as universal covering of the group $\homs$,
the latter being endowed with the compact open topology.  
One obtains a section of $p$ by associating to every $f\in\homs$ the unique lift 
$\overline f:\RR\to\RR$ with $0\leq\overline f(0)<1$. 
The extent to which $f\mapsto\overline f$ is not a homomorphism is measured by an integral 
$2$-cocycle $\epsilon$ given by 
\bqn
\overline f\circ\overline g=\overline{f\circ g}\circ T^{\epsilon(f,g)}\,,
\eqn
where $T$ is the image in $\Hh^+_\ZZ(\RR)$ of the generator $1\in\ZZ$.
The Euler class\index{Euler class} is then the cohomology class $e\in\h^2\big(\homs,\ZZ\big)$ defined by $\epsilon$.

\begin{definition}\label{defi:3.2} The {\em Euler number}\index{Euler number} $e(\rho)$ of a representation 
\bqn
\rho:\pi_1(S)\to\homs
\eqn 
is the integer
$\<\rho^\ast(e),[S]\>$ obtained by evaluation of the pullback $\rho^\ast(e)\in\h^2\big(\pi_1(S),\ZZ\big)$
of $e$ on the fundamental class $[S]$, or rather on its image under the isomorphism
\bqn
\h_2(S,\ZZ)\to\h_2\big(\pi_1(S),\ZZ\big)
\eqn
considered in \eqref{eq:1.2}.
\end{definition}

\subsection{K\"ahler form and Toledo number}\label{subsec:kaeh_tol}
Contrary to \S~\ref{subsec:euler} the viewpoint we present here emphasizes the fact that 
the Poincar\'e disk $\DD$ is an instance
of a {\em Hermitian symmetric space}\index{Hermitian symmetric space} with $G$-invariant K\"ahler form
\bqn
\omega_\DD:=\frac{dz\wedge d\overline z}{(1-|z|^2)^2}\,,
\eqn
where $G=\PSU(1,1)$.
Given a homomorphism $\rho:\pi_1(S)\to G$, consider then the bundle
with total space the quotient $\DD(\rho):=\pi_1(S)\backslash( D\times\DD)$
of $D\times\DD$ by the properly  discontinuous and fixed point free action 
$\gamma(x,z):=\big(T_\gamma x,\rho(\gamma)z\big)$,
and with basis $S=\pi_1(S)\backslash D$.
Since the typical fiber $\DD$ is contractible, one can construct, 
adapting the procedure described in \S~\ref{subsec:flat_gb}, a continuous and even a 
a smooth section.
Equivalently there is a smooth equivariant map $F:D\to\DD$.
As a result, the pullback $F^\ast(\omega_\DD)$ is a $\pi_1(S)$-invariant $2$-form on $D$
which gives a $2$-form on $S$ denoted again, with a slight abuse of notation, by $F^\ast(\omega_\DD)$.
The {\em Toledo number}\index{Toledo number} $\T(\rho)$ of the representation $\rho$ is then 
\bqn
\T(\rho):=\frac{1}{2\pi} \int_S F^\ast(\omega_\DD)\,.
\eqn
Recall that we have fixed an orientation on $S$ once and for all.

\begin{remark}\label{rem:3.3} One verifies, using again geodesic homotopy,
that any two $\rho$-equivariant smooth maps $D\to\DD$ are homotopic
and hence, by Stokes' theorem, one concludes that the de Rham cohomology class
$\big[F^\ast(\omega_\DD)\big]\in\hdr^2(S,\RR)$ is independent of $F$.
This shows that $\T(\rho)$ is independent of the choice of $F$.
\end{remark}

\subsection{Toledo number and first Chern classes}\label{subsec:tol_no_chern_class}
%
%
%

Let $L\to\DD$ be  a Hermitian complex line bundle over the Poincar\'e disk $\DD$
and $G'$ a finite covering of $\PSU(1,1)$ acting by bundle isomorphisms on $L$;
then the curvature form $\Omega_L$ is a $G'$-invariant $2$-form on $\DD$.
Given a representation $\rho:\pi_1(S)\to G'$ and a smooth equivariant map 
$F:\DD\to\DD$, $\pi_1(S)$ acts by bundle automorphisms on $F^\ast L\to\DD$
and, by passing to the quotient, we get a  complex line bundle $L(\rho)$
over $S$.  Then $\frac1{2\pi\imath}F^\ast\Omega_L$ descends to a $2$-form
$\omega_{L(\rho)}$ on $S$ which, by Chern-Weil theory, 
represents the first Chern class of $L(\rho)$, i.\,e.
\bq\label{eq:first_chern}
c_1\big(L(\rho)\big)=\int_S\omega_{L(\rho)}\in\ZZ\,.
\eq
Applying this to specific line bundles we obtain integrality
properties for the Toledo number.
Namely, let $\theta \to \DD\subset\CC\PP^1$ be the restriction of the tautological
bundle over $\CC\PP^1$ and $\theta^2$ be its square. Then
$\Omega_\theta = \frac{1}{2i} \omega_\DD$ and $\Omega_{\theta^2} =\frac{1}{i} \omega_\DD$. 
The group $\PSU(1,1)$ acts by isomorphisms on
$\theta^2$, and (\ref{eq:first_chern}) implies 
\bqn 
T(\rho)
= \frac{1}{2\pi} \int_S F^* \omega_\DD = -\int_S \omega_{\theta^2_\rho} =
-c_1(\theta^2_\rho) \, \in \ZZ\,. 
\eqn 
The group $\SU(1,1)$ acts
naturally on $\theta$, so for representations $\rho: \pi_1(S) \to
\SU(1,1)$ the relation in (\ref{eq:first_chern}) gives 
\bqn 
T(\rho) =
\frac{1}{2\pi} \int_S F^* \omega_\DD= - 2\int_S \omega_{\theta_\rho}= -2
c_1(\theta_\rho) \, \in 2 \cdot \ZZ\,. 
\eqn 
In particular, a
representation $\rho: \pi_1(S) \to \PSU(1,1)$ lifts to $\SU(1,1)$ if and only if 
its Toledo number is divisible by $2$. 

\subsection{Relations between the various invariants}\label{subsec:rel}
For $G=\PSU(1,1)$ we identify in the sequel $\pi_1(G)$ with $\ZZ$ as described in \S~\ref{subsec:centr_ext}
and obtain for a representation $\rho\in\hom\big(\pi_1(S),G\big)$ the obstruction class\index{obstruction class}
$o_2(\rho)\in\h^2(S,\ZZ)$ and the class $c_2(\rho)\in\h^2\big(\pi_1(S),\ZZ\big)$;
using the specific description of the latter in terms of central extensions 
as $\ZZ[\overline\Gamma_g]$,
we get the invariant $z_2(\rho)\in\ZZ$ by setting $c_2(\rho)=z_2(\rho)[\overline\Gamma_g]$.  
Then 
\bq\label{eq:3.7.1}
\<o_2(\rho),[S]\>=-z_2(\rho)\,,
\eq
(see \cite[Lemma 2]{Milnor} and \cite{Wood}).
Turning to the Euler class, 
we observe that the injection $\PSU(1,1)\hookrightarrow \homeo{S^1}$
is a homotopy equivalence as both groups retract on the (common) group of rotations.
Therefore the restriction $e|_{\PSU(1,1)}\in\h^2\big(\PSU(1,1),\ZZ\big)$ 
classifies the universal covering
of $\PSU(1,1)$ and hence for $\rho:\pi_1(S)\to\PSU(1,1)$ we have 
\bqn
\rho^\ast(e)=z_2(\rho)[\overline\Gamma_g]\,,
\eqn
which implies that 
\bq\label{eq:3.7.2}
e(\rho)=\<\rho^\ast(e),[S]\>=-z_2(\rho)\,.
\eq
To relate the previous invariant to the Toledo number 
we will recall the very general principle that 
invariant forms on a symmetric space form a complex, with $0$ as derivative,
which equals the continuous cohomology of the connected group of isometries.
In our special case of the Poincar\'e disk, this takes the form
\bqn
\Omega^2(\DD)^G\cong\hc^2(G,\RR),
\eqn
where, given $\omega_\DD$, we get a continuous cocycle
\bqn
c(g_1,g_2):=\frac{1}{2\pi}\int_{\Delta\big(0,g_1(0), g_1g_2(0)\big)}\omega_\DD\,, 
\eqn
where $\Delta\big(0,g_1(0), g_1g_2(0)\big)$ denotes the oriented geodesic triangle having vertices at 
the points $0,g_1(0), g_1g_2(0) \in \DD$.
We call the resulting class $\kg$ the {\em K\"ahler class}\index{K\"ahler class}.  In fact,
it is not difficult to show that under the change of coefficients
\bqn
\hc^2(G,\ZZ)\to\hc^2(G,\RR)
\eqn
the Euler class $e$ goes to the K\"ahler class $\kg$.  
If $\rho\in\hom\big(\pi_1(S),G\big)$, this implies  that
\bq\label{eq:3.7.3}
e(\rho)=\<\rho^\ast(e),[S]\>=\frac{1}{2\pi}\int_SF^\ast(\omega_\DD)=\T(\rho)\,.
\eq

\subsection{Milnor's inequality and Goldman's theorem}\label{subsec:miln_gold}
In his seminal paper \cite{Milnor}, 
J.~Milnor treated the problem of characterizing those classes in $\h^2(S,\ZZ)$ 
which are Euler classes of flat principal $\GL_2^+$-bundles.
The fact that, in general there are restrictions on the characteristic classes of
flat principal $G$-bundles and in particular on $o_2(\rho)\in\h^2\big(S,\pi_1(G)\big)$
comes from the following observation:
$o_2$ is constant on connected components of $\hom\big(\pi_1(S),G\big)$ 
and the latter is a real algebraic set when $G$ is a real algebraic group,
thus possesses only finitely many connected components
(see Proposition~\ref{prop:2.1}).
To get explicit restrictions, however, is not a trivial matter.
In the case of $G=\PSU(1,1)$-bundles this restriction, 
known as the {\em Milnor--Wood inequality}\index{Milnor--Wood inequality}\index{inequality!Milnor--Wood}, is the following:

\begin{theorem}[\cite{Milnor, Wood}]\label{thm:Milnor_Wood}  Let $\rho\in\hom\big(\pi_1(S),G\big)$
and let $g$ be the genus of $S$.  Then
\bqn
\big|\big\<o_2(\rho),[S]\big\>\big|\leq 2g-2\,.
\eqn
\end{theorem}
In light of subsequent generalizations of this inequality 
it is instructive to give an outline of the original arguments.
Consider the retraction
\bqn
 r:\PSU(1,1)\to K
=\left\{\pm\begin{pmatrix}e^{\imath s\pi}&0\\0&e^{-\imath s\pi}\end{pmatrix}:\,
s\in\RR/\ZZ\right\}
\eqn
given by decomposing $g=r(g)h(g)$ as a product of a rotation $r(g)$ 
with a Hermitian matrix $h(g)$.
Now lift $r$ to the universal covering
\bqn
\tilde r:\widetilde{\PSU(1,1)}\to\RR
\eqn
in such a way that $\tilde r(e)=0$.  Then: 
\begin{enumerate}
\item\label{item:3.8.1} $\tilde r(t^n g) = n+\tilde r(g)$, 
where $t$ is the generator of $\pi_1\big(\PSU(1,1)\big)$;
\item\label{item:3.8.2} $\tilde r(g^{-1})=-\tilde r(g)$;
\item\label{item:3.8.3}  $\big|\tilde r(ab)-\tilde r(a)-\tilde r(b)\big|<\frac12$ 
for all $a,b\in\PSU(1,1)$.
\end{enumerate}
This construction as well as the proof of these properties are given in \cite{Milnor} in the case
of $\GL_2^+(\RR)$ (see also \cite{Wood}).

Given $\rho\in\hom\big(\pi_1(S),G\big)$ let now
$\alpha_i,\beta_i$ be lifts to $\widetilde{\PSU(1,1)}$ of $\rho(a_i)$ and $\rho(b_i)$
(see \S~\ref{subsec:descr}).  Then
\bqn
t^{z_2(\rho)}=\prod_{i=1}^g[\alpha_i,\beta_i]\,.
\eqn
On applying the above properties several times 
we obtain
\bqn
\big|z_2(\rho)\big| = | \tilde r({t^{z_2(\rho)}})| = 
|\tilde r({\prod_{i=1}^g[\alpha_i,\beta_i]})|\leq (4g-1)\frac12=2g-\frac12
\eqn
which, since $z_2(\rho)$ is an integer, implies that 
\bqn
\big|z_2(\rho)\big|\leq 2g-1\,.
\eqn
This is not quite the announced result.  An additional argument is needed 
and can be found in \cite{Wood};
we present instead an argument in the spirit of Gromov's trick to compute 
the simplicial area of a surface.  Namely, let $p:S'\to S$
be a covering of degree $n\geq1$ and let $p_\ast:\pi_1(S')\to\pi_1(S)$
be the resulting morphism.  The inequality above, applied to $\rho\circ p_\ast$ gives
\bqn
\big|\<o_2(\rho\circ p_\ast),[S']\>\big|\leq 2g'-\frac12\,,
\eqn
where $g'$ is the genus of $S'$.  Since $o_2$ is a characteristic class, we have
\bqn
o_2(\rho\circ p_\ast)=p^\ast\big(o_2(\rho)\big)\,,
\eqn
where $p^\ast:\h^2(S,\ZZ)\to\h^2(S',\ZZ)$ and thus
\bqn
  \big\<o_2(\rho\circ p_\ast),[S]\big\>
=\big\<p^\ast\big(o_2(\rho)\big),[S']\big\>
=\big\<o_2(\rho),p_\ast[S']\big\>
=n\big\<o_2(\rho),[S]\big\>\,,
\eqn
since $p$ is of degree $n$.  Using the relation $g'-1=n(g-1)$,
we obtain
\bqn
\big|\<o_2(\rho),[S]\big\>\big|
<\frac{4n(g-1)+3}{2n}\,,
\eqn
which gives the desired inequality as soon as $n\geq 2$, since the left hand side is an integer.

In Milnor's paper \cite{Milnor} the construction and 
the property (\ref{item:3.8.3}) of $\tilde r$
come as a complete surprise. With hindsight, 
it is an instance of a {\em quasimorphism}\index{quasimorphism} and it is in the context of bounded cohomology
that its relation to the Euler class and the specific constant $\frac12$ 
in (\ref{item:3.8.3}) are explained.

Concerning the optimality of the inequality, it is shown also in \cite{Wood}
that every integer between $-(2g-2)$ and $2g-2$ is attained.  
In particular the inequality is optimal and one way to see this is 
to compute the Toledo invariant
of a homomorphism $\rho_h:\pi_1(S)\to G$ corresponding to a hyperbolic structure $h$ on $S$.
For this we have at our disposal the orientation preserving isometry $f_h:D\to\DD$
and hence the form $f_h^\ast(\omega_\DD)$ on $S$ coincides with the area $2$-form
$\omega_h$ given by the hyperbolic structure. Thus
\bqn
\T(\rho_h)=\frac1{2\pi}\int_S\omega_h=\big|\chi(S)\big|=2g-2\,.
\eqn
It should be observed that the value of the area of $S$ can be obtained directly 
from the formula
of the area of a geodesic triangle in $\DD$ applied to the triangulation 
of a "standard" fundamental polygon, taking into account that the sum of the 
internal angles if $2\pi$.
In light of this computation it is a very natural question what is the nature 
of the homomorphisms $\rho$ for which $\T(\rho)=2g-2$.
The answer is given by Goldman in his thesis \cite{Goldman_thesis}:

\begin{theorem}[\cite{Goldman_thesis}]\label{thm:goldman_thesis}
A representation $\rho:\pi_1(S)\to\PSU(1,1)$ corresponds to a hyperbolic structure on $S$ 
if and only if $\T(\rho)=2g-2$.
\end{theorem}
A reformulation of Theorem~\ref{thm:goldman_thesis} is given in (\ref{eq:goldman}). 
In particular,  the image of $\hyp(S)$ in $\hom\big(\pi_1(S),G\big)$
being the preimage of $2g-2$ under $\T$ is hence a union of components.
In fact a little later Goldman proved that the preimages $\T^{-1}(n)$,
for $n\in\ZZ\cap\big[-(2g-2),2g-2\big]$ are exactly 
the components of $\hom\big(\pi_1(S),G\big)$ \cite{Goldman_components}. 
The component where $\T = 2-2g$ corresponds to hyperbolic structures on S with the reversed orientation. 
 
\subsection{An application to Kneser's theorem}\label{subsec:appl_knes}
This theorem takes its motivation in the question  of 
what are the possible degrees of continuous maps
from a compact oriented surface  $S$ to itself.  
If $S$ is either the sphere or the torus, then maps of arbitrarily high degree exist.
This is not the case anymore if the genus of $S$ is at least two, and more generally we have
the following:

\begin{theorem}[\cite{Kneser}]\label{thm:kneser} Let $f:S_1\to S_2$ be a continuous map
between compact oriented surfaces $S_i$ of genus at least $2$.
Then
\bqn
|\deg f|\leq\frac{|\chi(S_1)|}{|\chi(S_2)|}\,,
\eqn
with equality if and only if $f$ is homotopic to a covering map, 
necessarily of degree $\frac{\chi(S_1)}{\chi(S_2)}$.
\end{theorem}

\begin{proof} Let $f_\ast:\pi_1(S_1)\to\pi_1(S_2)$ be the homomorphism induced 
on the level of fundamental groups, and pick $\rho\in\hom\big(\pi_1(S_2),G\big)$.
Then 
\bq\label{eq:3.9.1}
\ba
    \big\<o_2(\rho\circ f_\ast),[S_1]\big\>
  &=\big\<f^\ast\big(o_2(\rho)\big),[S_1]\big\>
  =\big\<o_2(\rho),f_\ast\big([S_1]\big)\big\>\\
  &=\deg f\big\<o_2(\rho),[S_2]\big\>\,.
\ea
\eq
Specializing now to $\rho=\rho_h$, for $h\in\hyp(S_2)$, 
we get 
\bqn
\big\<o_2(\rho_h),[S_2]\big\>=\big|\chi(S_2)\big|
\eqn
while the Milnor-Wood inequality gives
\bqn
\big|\<o_2(\rho_h\circ f_\ast),[S_1]\>\big|\leq\big|\chi(S_1)\big|
\eqn
which, together with \eqref{eq:3.9.1}, gives the inequality on $|\deg f|$.

Assume now that we have equality and, without loss of generality, that
\bqn
\T(\rho_h\circ f_\ast)=\big|\chi(S_1)\big|\,.
\eqn
Then Goldman's theorem implies that $\rho_h\circ f_\ast$ corresponds to a hyperbolic 
structure on $S_1$ and, in particular, $f_\ast$ is injective.
Letting $p:T\to S_2$ denote the covering of $S_2$ corresponding 
to the image of $f_\ast$, we have that 
\bqn
f_\ast:\pi_1(S_1)\to\pi_1(T)
\eqn
is an isomorphism, which implies that $T$ is compact and
(by Nielsen's theorem) that $f_\ast$ is induced by a homeomorphism $F:S_1\to T$.

We have then that the homomorphisms $f_\ast$ and $(p\circ F)_\ast$ coincide;
let now $\tilde f:\tilde S_1\to\tilde S_2$ and 
$\widetilde{p\circ F}:\tilde S_1\to\tilde S_2$ be lifts
of $f$.  These are continuous maps which are equivariant with respect to the same
homomorphism $\pi_1(S_1)\to\pi_1(S_2)$.
Upon choosing a hyperbolic metric on $S_2$, we conclude by using a geodesic
homotopy that $\widetilde{p\circ F}$ and $\tilde f$ are equivariantly
homotopic and hence $p\circ F$ and $f$ are homotopic.
\end{proof}

\section{Surfaces of finite type and the Euler number}\label{sec:surf_fin_eul}

\subsection{Hyperbolic structures on surfaces of finite type \\  and semiconjugations}\label{subsec:hyp_surf_quasiconj} 
Let $S$ be a compact (connected, oriented) surface.
Then the image of $\hyp(S)$ in $\hom\big(\pi_1(S),G\big)$, $G  = \PSU(1,1)$,  
can be described by one equation in the image of the generators;
namely, letting $t$ be the generator of $\pi_1(G)$,
$a_1,b_1,\dots,a_g, b_g$ the generators of $\pi_1(S)$ defined in \S~\ref{subsec:descr}
and introducing the smooth map
\bqn
\ba
G\times G&\longrightarrow \widetilde G\\
(g,h)&\mapsto[g,h]^{\widetilde\,}
\ea
\eqn
where $[g,h]^{ \widetilde\,}$ is the commutator of any two lifts of $g$ and $h$, 
Goldman's theorem (Theorem~\ref{thm:goldman_thesis}) can be restated as 
\bq\label{eq:goldman}
 \delta\big(\hyp(S)\big)
=\left\{\rho\in\hom\big(\pi_1(S),G\big):\,
 \prod_{i=1}^g\big[\rho(a_i),\rho(b_i)\big]^{ \widetilde\,}=t^{2-2g}\right\}\,.
\eq
The aim of this section is to present a circle of ideas, rooted in the theory 
of bounded cohomology, which will, among other things, 
lead to an analogous explicit description of $\delta\big(\hyp(S)\big)$
in the case in which $S$ is not compact.
We will however always assume that $\pi_1(S)$ is finitely generated;
equivalently $S$ is diffeomorphic to the interior of a compact surface with boundary.
The genus $g$ of this surface and the number $n$ of boundary components together
determine $S$ up to diffeomorphism. We say that $S$ is of {\em finite topological type}\index{finite topological type}.

The first observation is that the invariants introduced in \S~\ref{sec:inv_miln_gold} are of no use 
when $S$ is not compact.  In fact, for a connected surface $S$ the following are equivalent:
\begin{enumerate}
\item $\h^2(S,\ZZ)=\h^2\big(\pi_1(S),\ZZ\big)=0$;
\item\label{item:4.1.2} $\pi_1(S)$ is a free group;
\item $S$ is not compact.
\end{enumerate}
Elaborating a little on \eqref{item:4.1.2}, if $r$ is the rank of $\pi_1(S)$ 
as a free group, we have clearly that $\hom\big(\pi_1(S),G\big)\cong G^r$
and, as a result, this space of homomorphisms is always connected.
Thus $\delta\big(\hyp(S)\big)$ will not be a connected component.

The second observation, and this will lead us in the right direction, 
is to consider more closely the inclusions
\bqn
       \delta\big(\hyp(S)\big)
\subset\hom\big(\pi_1(S),G\big)
\subset\hom\big(\pi_1(S),\homeo{S^1}\big)
\eqn
in the case in which $S$ is compact. 
For $h_1, h_2\in\hyp(S)$, the diffeomorphism 
$f_{h_1}\circ f_{h_2}^{-1}:\DD\to\DD$ clearly conjugates
$\rho_{h_2}$ to $\rho_{h_1}$ within $\Diff^+(\DD)$.
Since $S$ is compact, $f_{h_1}\circ f_{h_2}^{-1}$ is a quasi isometry;
it is then a fundamental fact in hyperbolic geometry that $f_{h_1}\circ f_{h_2}^{-1}$
extends to an (orientation preserving) homeomorphism of $S^1=\partial\DD$.
Thus any two representations in $\delta\big(\hyp(S)\big)$ are conjugate in
$\homeo{S^1}$ and it is an easy exercise to see that 
any $\rho\in\hom\big(\pi_1(S),G\big)$ that is conjugate to an element
in $\delta\big(\hyp(S)\big)$ in $\homeo{S^1}$
in fact belongs to $\delta\big(\hyp(S)\big)$;
indeed such a representation $\rho$ is injective with discrete image.
Thus a full invariant of conjugacy on $\hom\big(\pi_1(S),\homeo{S^1}\big)$
would lead to a characterization of $\delta\big(\hyp(S)\big)$
within $\hom\big(\pi_1(S),G\big)$!

We will now develop this line of thought 
in the case in which $S$ is of finite topological type.
We assume that $S$ has a fixed orientation and let $\Sigma$ denote 
a compact surface of genus $g$ with $n$ boundary components
such that $S=\operatorname{int}(\Sigma)$.
Then $\pi_1(S)$ admits a presentation
\bq\label{eq:fg}
\left\<a_1,b_1,\dots,a_g,b_g,c_1,\dots,c_n:
\prod_{i=1}^g[a_i,b_i]\prod_{j=1}^nc_j=e\right\>\,.
\eq
Here each $c_i$ is freely homotopic to the $i$-th component of $\partial\Sigma$
with orientation compatible with the chosen orientation on $\Sigma$.
Let now $h$ be a complete hyperbolic metric on $S$.
We have then two possibilities for $\rho_h(c_i)$:
\begin{enumerate}
\item $\rho_h(c_i)$ is parabolic: it has a unique fixed point $\xi_i\in\partial\DD$ 
and for the interior $C_i$ of an appropriate horocycle based at $\xi_i$,
the quotient $\<\rho_h(c_i)\>\backslash C_i$ is of finite area 
and embeds isometrically into $\rho_h\big(\pi_1(S)\big)\backslash\DD$.
It is a neighborhood of the $i$-th end of $S$.
\item $\rho_h(c_i)$ is hyperbolic: it has an invariant axis $a_i\subset\DD$ 
which determines a half plane $H_i\subset\DD$ such that 
$\partial H_i$ and $a_i$ have opposite orientation.
The quotient $\<\rho_h(c_i)\>\backslash H_i$ embeds isometrically into 
$\rho_h\big(\pi_1(S)\big)\backslash\DD$. 
It is of infinite area and a neighborhood of the $i$-th end.
\end{enumerate}

Let $\Lambda_h\subset\partial\DD$ be the limit set of $\rho_h\big(\pi_1(S)\big)$.
Then either $\Lambda_h=\partial\DD$, equivalently $(S,h)$ is of finite area,
or $\Lambda_h\neq\partial\DD$, in which cases it is a Cantor set;
the connected components of $\partial\DD\setminus\Lambda_h$ 
are then in bijective correspondence with the set of elements
in 
\bqn
\big\{\gamma\in\pi_1(S):\gamma\text{ is conjugate to a boundary loop }c_i\text{ s.\,t.\,}
\rho_h(c_i)\text{ is hyperbolic}\big\}
\eqn
Thus, if $h_1,h_2\in\hyp(S)$ are such that $h_1$ has finite area
while $h_2$ has infinite area, then $\rho_{h_1}$ gives a minimal action
of $\pi_1(S)$ on $\partial\DD$, while $\rho_{h_2}$ gives an action 
on $\partial\DD$ which admits $\Lambda_{h_2}$ as minimal set.
In particular $\rho_{h_1},\rho_{h_2}$ cannot be conjugated in 
$\homeo{S^1}$.  Let us however consider the diffeomorphism
$F:=f_{h_1}\circ f_{h_2}^{-1}:\DD\to\DD$.

\begin{proposition}\label{prop:4.1} Assume that $h_1$ is of finite area.
Then $F$ extends to a continuous map $\varphi:\partial\DD\to\partial\DD$ which 
is weakly monotone.
\end{proposition}

This proposition is a slight generalization of a classical result stating that 
if $h_1,h_2$ have finite area, then 
since the isomorphism $\rho_{h_1}\circ\rho_{h_2}^{-1}$ is ``type preserving'', 
$f_{h_1}\circ f_{h_2}^{-1}$ extends to a homeomorphism of $\partial\DD$.

To explain the statements of the proposition, recall that the circle $S^1 \cong \partial\DD$ is equipped
with its canonical positive orientation and 
this gives a natural notion for triples of points to be positively oriented.  
We have the following

\begin{definition} An (arbitrary) map $\varphi:S^1\to S^1$ is {\em weakly monotone}\index{weakly monotone} if whenever
$x,y,z\in S^1$ are such that $\varphi(x),\varphi(y),\varphi(z)$ are distinct,
then $(x,y,z)$ and $\big(\varphi(x),\varphi(y),\varphi(z)\big)$ have the same
orientation.
\end{definition}

Typically the map in Proposition~\ref{prop:4.1} is collapsing a connected component
in $\partial\DD\setminus\Lambda_{h_2}$ corresponding to $\gamma\in\pi_1(S)$ to the
corresponding fixed point in $\partial\DD$ of the parabolic element $\rho_{h_1}(\gamma)$.
We have for every $x\in\partial\DD$
\bqn
\varphi\circ\rho_{h_2}(\gamma)(x)=\rho_{h_1}(\gamma)\circ\varphi(x)
\eqn
and we say that $\varphi$ {\em semiconjugates}\index{semiconjugate} $\rho_{h_2}$ to $\rho_{h_1}$.
It is in order to reverse this process that
in the definition of weakly monotone map one allows discontinuous maps.
For instance $\varphi^{-1}(x)$ is always an interval and 
if we set $\psi(x)$ equal to the left endpoint of $\varphi^{-1}(x)$, then
$\psi$ is weakly monotone and 
\bqn
\rho_{h_2}(\gamma)\circ\psi(x)=\psi\circ\rho_{h_1}(\gamma)(x)\,.
\eqn
Thus given now arbitrary homomorphisms
$\rho_1,\rho_2:\Gamma\to\homeo{S^1}$ defined on a group $\Gamma$, 
we say that $\rho_1$ and $\rho_2$ are semiconjugate 
if there exists $\varphi:S^1\to S^1$ weakly monotone such that for every $\gamma\in\Gamma$
\bqn
\varphi\circ\rho_1(\gamma)=\rho_2(\gamma)\circ\varphi\,.
\eqn
It is now clear that semiconjugation is an equivalence relation.
In our specific situation we have then
\begin{corollary}\label{cor:4.3}  \begin{enumerate}
\item For any $h_1,h_2\in\hyp(S)$, $\rho_{h_1}$ and $\rho_{h_2}$ are semiconjugate.
\item If $h\in\hyp(S)$ and $\rho\in\hom\big(\pi_1(S),G\big)$ is semiconjugate to
$\rho_h$, then $\rho\in\delta\big(\hyp(S)\big)$.
\end{enumerate}
\end{corollary}
\begin{proof} The first assertion follows from the discussion above.  
We will now indicate the main points entering in the proof of the second one.

Let $\varphi:\partial\DD\to\partial\DD$ be weakly monotone with
\bq\label{eq:4.1.1}
\varphi\circ\rho(\gamma)=\rho_h(\gamma)\circ\varphi
\eq
and assume, in virtue of the first part, that $h$ has finite area.
Since $\rho_h\big(\pi_1(S)\big)$ acts minimally on $\partial\DD$,
we must have that $\overline{\im\varphi}=\partial\DD$.
It is easy to see that for a weakly monotone map 
this implies that $\varphi$ is continuous.  Then we deduce from \eqref{eq:4.1.1}
that $\rho$ is injective and with discrete image.
Thus $\Gamma:=\rho\big(\pi_1(S)\big)$ is a finitely generated discrete subgroup
of $\PSU(1,1)$ and hence $\Gamma\backslash\DD$ is topologically of finite type.

One uses then $\varphi$ to check that the isomorphism $\rho:\pi_1(S)\to\Gamma$
sends $\<c_i\>$, for each $i$, isomorphically into the fundamental group
of a boundary component of $\Gamma\backslash\DD$ 
and that each boundary component is so obtained.

Then an appropriate version of the Nielsen realization implies that 
$\rho$ is implemented by a diffeomorphism $S\to\Gamma\backslash\DD$, 
by means of which we produce the hyperbolic structure $h'$ for which $\rho=\rho_{h'}$.
\end{proof}

\subsection{The bounded Euler class}\label{subsec:bdd_eul}
The discussion of the preceding section shows that semiconjugation 
is a natural notion of equivalence for group actions by homeomorphisms of the circle,
at least in the framework of the questions regarding hyperbolic structures.
A different context is provided by a paraphrase of a famous theorem of Poincar\'e
concerning rotation numbers of homeomorphisms, namely two orientation preserving
homeomorphisms of the circle are semiconjugate if and only if 
they have the same rotation number.

Remarkably, there is an invariant generalizing 
the rotation number of a single homeomorphism to arbitrary group actions and 
which is a complete invariant of semiconjugacy: it is the 
bounded Euler class, introduced by Ghys in \cite{Ghys_87} 
and whose main features we now describe briefly.

For this let us recall that {\em bounded cohomology} can be defined by
restricting to bounded cochains in the (inhomogeneous) bar resolution.
Let $A= \RR$ or $\ZZ$, and $G$ be any group. Denote by $C^n(G,A)$ the
space of function from $G^n$ to $A$ and by $C^n_b(G,A) :=\big\{ f \in
C^n(G,A)\, :\, \sup_{\underline{g} = (g_1,\cdots, g_n)\in G^n}
|f(\underline{g})| < \infty\big\}$ the subspace of bounded functions.
Defining the boundary map \bqn d_n: C^n (G,A) \to C^{n+1} (G,A) \eqn
by \bqn \ba
d_n f (g_1, \cdots , g_{n+1}) &= f(g_2, \cdots, g_{n+1}) \\
&+ \sum_{i=1}^n (-1)^i f(g_1, \cdots, g_{i-1}, g_i g_{i+1}, g_{i+2}, \cdots, g_{n+1}) \\
&+ (-1)^{n+1} f(g_1, \cdots, g_n)\,, \ea \eqn we obtain the complex
$\big(C^\bullet(G,A), d_\bullet\big)$, whose cohomology is the group
cohomology $\h^\bullet (G,A)$, and the sub-complex $\big(C^\bullet_b(G,A),
d_\bullet\big)$, whose cohomology is the bounded cohomology\index{bounded
  cohomology} $\hb^\bullet (G,A)$ of $G$.

Bounded cohomology behaves very differently from usual cohomology, for example,
the second bounded cohomology $\hb^2(\FF_r,\RR)$ of a
nonabelian free group is infinite dimensional. This 
different behavior will allow us to
define bounded analogues of the invariants introduced in
\S~\ref{sec:inv_miln_gold}, which are meaningful when $S$ is
noncompact, and give finer information even in the case when $S$ is
compact (see e.g. Corollary~\ref{cor:4.5}).

Recall that, in the notation of \S~\ref{subsec:euler}, 
a representative cocycle $\epsilon$ for the Euler class
$e\in\h^2\big(\homeo{S^1},\ZZ\big)$
was given by 
\bqn
\overline f\circ\overline g=\overline{f\circ g}\circ T^{\epsilon(f,g)}\,,
\eqn
where $\overline f$ and $\overline g$ are the unique lifts to $\RR$
of $f,g\in\homeo{S^1}$ such that 
\bqn
0\leq \overline f(0),\overline g(0)<1\,,
\eqn
and $T:\RR\to\RR$ is defined by $T(x):=x+1$. 

Since $\overline f$ is increasing and commutes with $T$, we have
\bqn
\overline f\big(\overline g(0)\big)\in\big[\overline f(0),\overline f(0)+1\big)
\eqn
and since $\overline{fg}(0)\in[0,1)$, we obtain that $\epsilon(f,g)\in\{0,1\}$ 
and hence in particular $\epsilon$ is a bounded cocycle.  The class
\bqn
\eb\in\hb^2\big(\homeo{S^1},\ZZ\big)
\eqn
so obtained is called the {\em bounded Euler class}\index{bounded Euler class}\index{Euler class!bounded}
and given any homomorphism $\rho:\Gamma\to\homeo{S^1}$, 
$\rho^\ast(\eb)\in\hb^2(\Gamma,\ZZ)$ is called the {\em bounded Euler class of the action 
given by $\rho$}\index{bounded Euler class of an action}.  We have then the following
\begin{theorem}[\cite{Ghys_87}]\label{thm:ghys} The bounded Euler class 
of a homomorphism $\rho:\Gamma\to\homeo{S^1}$ is a full invariant of semiconjugation.
\end{theorem}

The relation with the classical {\em rotation number}\index{rotation number} is then the following.
Recall that the translation number $\tau(\varphi)\in\RR$ of a homeomorphism
$\varphi\in\Hh_\ZZ^+(\RR)$ is given by 
\bqn
\tau(\varphi):=\lim_{n\to\infty}\frac{\varphi^n(0)}{n}\,.
\eqn
Then $\tau$ has the following remarkable properties (compare to the
properties of $\tilde{r}$ in \S~\ref{subsec:miln_gold}):
\begin{enumerate}
\item $\tau$ is continuous;
\item\label{item:4.2.2.2} $\tau(\varphi\circ T^m)=\tau(\varphi)+m$, for $m\in\ZZ$;
\item $\tau(\varphi^k)=k\tau(\varphi)$;
\item $\big|\tau(\varphi\psi)-\tau(\varphi)-\tau(\psi)\big|\leq1$, 
for all $\varphi,\psi\in\Hh_\ZZ^+(\RR)$.
\end{enumerate}
In the language of bounded cohomology, this says that 
$\tau$ is a continuous homogeneous quasimorphism\index{homogeneous quasimorphism}\index{quasimorphism!homogeneous}.  
Then for $f\in\homeo{S^1}$ the rotation number\index{rotation number} of $f$ is
\bqn
\rot(f):=\tau(\overline f)\mod\ZZ
\eqn
which is well defined in view of \eqref{item:4.2.2.2}. 

Given now $f\in\homeo{S^1}$, consider
\bqn
\ba
h_f:\ZZ&\to\homeo{S^1}\\
n&\longmapsto\quad f^n
\ea
\eqn
to obtain an invariant $h_f^\ast(\eb)\in\hb^2(\ZZ,\ZZ)$.
Writing the long exact sequence in bounded cohomology \cite[Proposition~1.1]{Gersten} associated to
\bqn
\xymatrix{
 0\ar[r]
&\ZZ\ar[r]
&\RR\ar[r]
&\RR/\ZZ\ar[r]
&0
}
\eqn
we get 
\bqn
\xymatrix{
 0\ar[r]
&\hom(\ZZ,\RR/\ZZ)\ar[r]^-\delta
&\hb^2(\ZZ,\ZZ)\ar[r]
&0
}
\eqn
and then
\bq\label{eq:4.2.3}
\big(\delta^{-1}h_f^\ast(\eb)\big)(1)=\rot(f)\,.
\eq
It should be noticed in passing that the definition of $\rot(f)$ 
involves taking a limit, while the left hand side of \eqref{eq:4.2.3}
only involves purely algebraic constructions.\footnote{In fact we have used $\hb^1(\ZZ, \ZZ) = 0$ and $\hb^2(\ZZ,\RR) = 0$. The first equality follows from the fact that there are no (nonzero) bounded homomorphisms into $\RR$ while the second follows from the elementary fact that if $\psi: \ZZ \to \RR$ is a quasimorphism and $\alpha:= \lim_{n\to \infty} \frac{\psi(n)}{n}$, then $\psi$ is at bounded distance from the homomorphism $n\mapsto n\alpha$.}

The proof of the not straightforward implication of Ghys' theorem goes as follows.
Let $\rho_1$, $\rho_2:\Gamma\to \homs$ be homomorphisms, with $\rho_1^\ast(\eb)=\rho_2^\ast(\eb)$.
Hence $\rho_1^\ast(e)=\rho_2^\ast(e)$ and, by replacing $\Gamma$ with a suitable central 
extension by $\ZZ$ we may assume that $\rho_1$ and $\rho_2$ lift to homomorphisms
$\widetilde{\rho_1}$, $\widetilde{\rho_2}:\Gamma\to\Hh^+_\ZZ(\RR)$.
But then 
\bqn
\widetilde{\rho_i}(\gamma)=\overline{\rho_i(\gamma)}T^{c_i(\gamma)}
\eqn
and the hypothesis that $\rho_1,\rho_2$ have the same bounded Euler class is equivalent
to saying that we may choose the lifts $\widetilde{\rho_1}$ and $\widetilde{\rho_2}$ such that 
\bqn
c_2-c_1:\Gamma\to\ZZ
\eqn
is bounded.  Thus:
\bqn
        \widetilde{\varphi}(x):=
        \sup_{\gamma\in\Gamma}
                \big\{\widetilde{\rho_1}(\gamma)^{-1}\widetilde{\rho_2}(\gamma)(x):\,\gamma\in\Gamma\big\}
                        <+\infty
\eqn
is well defined for every $x$ and gives a monotone map $\RR\to\RR$ commuting with $T$ and satisfying
\bqn
\widetilde\varphi\widetilde{\rho_2}(\eta)=\widetilde{\rho_1}(\eta)\widetilde\varphi\,,
\eqn
for all $\eta\in\Gamma$. This shows that $\rho_1$ and $\rho_2$ are semiconjugate.

\medskip 

In order to complete one of the descriptions of the image of $\hyp(S)$ under the map
$\delta$ in $\hom\big(\pi_1(S),G\big)$ let 
now $S$ be again an oriented surface of finite topological type,
where we do not exclude the case in which $S$ is compact.
We have seen that for any two $h_1,h_2\in\hyp(S)$, 
the homomorphisms $\rho_{h_1}$ and $\rho_{h_2}$ are semiconjugate
in $\homeo{S^1}$ and hence, by the easy direction of Ghys' theorem,
we have that
\bqn
\rho_{h_1}^\ast(\eb)=\rho_{h_2}^\ast(\eb)\,.
\eqn
Let $\ksb\in\hb^2\big(\pi_1(S),\ZZ\big)$ denote the class so obtained.
Then
\begin{corollary}\label{cor:4.5} 
\bqn
\delta\big(\hyp(S)\big)=\big\{\rho\in\hom\big(\pi_1(S),G\big):\,\rho^\ast(\eb)=\ksb\big\}\,.
\eqn
\end{corollary}
\begin{proof}
The inclusion $\subset$ has already been discussed.  

If now $\rho^\ast(\eb)=\ksb$, then Ghys' theorem implies that 
$\rho$ is semiconjugate to an element in $\delta\big(\hyp(S)\big)$
and the assertion follows from Corollary~\ref{cor:4.3}.
\end{proof}

\subsection{Bounded Euler number and bounded Toledo number}\label{subsec:bbd_eul_bdd_tol}
In this section we describe two ways in which one can associate a (real) number
to the bounded Euler class; this will give the two invariants mentioned in the title.
The fact that they coincide is then an essential result containing a lot 
of information.

Recall that $S$ is a surface of finite topological type and hence we
may consider it as the interior of a compact surface $\Sigma$ with
boundary $\partial\Sigma$. Let now $\rho:\pi_1(\Sigma)\to\homeo{S^1}$
be a homomorphism and 
\bqn
\rho^\ast(\eb)\in\hb^2\big(\pi_1(\Sigma),\ZZ\big) 
\eqn 
its bounded Euler class.  We proceed now to define the {\em bounded Euler
number}\index{bounded Euler number}\index{Euler number!bounded} of
$\rho$. First we use that the classifying map $\Sigma\to
B\pi_1(\Sigma)$ is a homotopy equivalence in order to obtain a natural
isomorphism
$\hb^2\big(\pi_1(\Sigma),\ZZ\big)\to\hb^2\big(\Sigma,\ZZ\big)$ by
means of which we consider, keeping the same notation, the class
$\rho^\ast(\eb)$ as a bounded singular class on $\Sigma$.  (See
\cite{Gromov_bounded} for the definition of singular bounded
cohomology.) The inclusion $\partial\Sigma\hookrightarrow\Sigma$ gives
in a straightforward way a long exact sequence in bounded cohomology
with coefficients in $A=\ZZ,\RR$, whose relevant part for us reads
\bqn
\xymatrix{
 \hb^1(\partial\Sigma,A)\ar[r]
&\hb^2(\Sigma,\partial\Sigma,A)\ar[r]^-{f_A}
&\hb^2(\Sigma,A)\ar[r]
&\hb^2(\partial\Sigma,A)
}
\eqn
which gives for $A=\ZZ$
\bqn
\xymatrix{
 0\ar[r]
&\hb^2(\Sigma,\partial\Sigma,\ZZ)\ar[r]^-{f_\ZZ}
&\hb^2(\Sigma,\ZZ)\ar[r]
&\hb^2(\partial\Sigma,\ZZ)
}
\eqn
and for $A=\RR$
\bqn
\xymatrix{
 0\ar[r]
&\hb^2(\Sigma,\partial\Sigma,\RR)\ar[r]^-{f_\RR}
&\hb^2(\Sigma,\RR)\ar[r]
&0
}
\eqn
where we have used the following facts (see \cite{Burger_Iozzi_Wienhard_toledo} or see footnote in equality \eqref{eq:4.2.3}):
\begin{enumerate}
\item $\hb^1(\partial\Sigma,A)=0$ for $A=\RR,\ZZ$;
\item $\hb^2(\partial\Sigma,\RR)=0$.
\end{enumerate}
As a result we have that if we consider $\rho^\ast(\eb)$ 
as a real bounded class on $\Sigma$, it corresponds to a unique
relative class
\bqn
f_\RR^{-1}\big(\rho^\ast(\eb)\big)\in\hb^2(\Sigma,\partial\Sigma,\RR)\,.
\eqn
The latter can then be seen as an ordinary singular relative class and 
hence can be evaluated on the relative fundamental class,
thus leading to the {\em bounded Euler number}\index{bounded Euler number}\index{Euler number!bounded}:
\bq\label{eq:4.4.1}
\eb(\rho):=\big\<f_\RR^{-1}\big(\rho^\ast(\eb)\big),[\Sigma,\partial\Sigma]\big\>\,.
\eq

Two important remarks are in order here.
First, the definition of this invariant not only involves $\pi_1(S)$ but also 
the surface $S$ itself;  this is essential if this invariant is to detect hyperbolic structures on $S$
(see Remark~\ref{rem:2.6}).  
Second, let us denote by $\rho^\ast(\erb)\in\hb^2\big(\homeo{S^1},\RR\big)$
the real bounded class obtained by considering the cocycle $\epsilon$ as 
taking values in $\RR$.  Then $\eb(\rho)$ depends in fact only on the real class 
$\rho^\ast(\erb)$; the extent to which this (real) class determines $\rho$ (up to semiconjugation),
is completely understood (see \cite{Burger_zimmer}).

The {\em bounded Euler number} $\eb(\rho)$ is in general not an integer.
Remarkably, one can give an explicit formula for the "fractional part"
of $\eb(\rho)$; indeed, combining the long exact sequence associated to
$\partial\Sigma\to\Sigma$ together with the one associated to the
short exact sequence 
\bqn \xymatrix{ 0\ar[r] &\ZZ\ar[r] &\RR\ar[r]
  &\RR/\ZZ\ar[r] &0 } 
\eqn 
leads to the following congruence relation

\bqn
\eb(\rho)=-\sum_{i=1}^n\rot\rho(c_i)\mod\ZZ\,.
\eqn
In fact, using this, one can establish a general formula for $\eb(\rho)$:

\begin{theorem}[\cite{Burger_Iozzi_Wienhard_toledo}]\label{thm:4.6}
 Let $S$ be an oriented surface of finite topological type
  with presentation of its fundamental group 
  \bqn
  \pi_1(S)=\left\<a_1,b_1,\dots,a_g,b_g,c_1,\dots,c_n:
  \prod_{i=1}^g[a_i,b_i]\prod_{j=1}^nc_j=e\right\> 
  \eqn 
  as defined in \eqref{eq:fg}.  Let $\rho:\pi_1(S)\to\homeo{S^1}$ be a
  homomorphism and let $\tau$ denote the translation quasimorphism.
  Then
\begin{enumerate}
\item If $S$ is compact (that is $n=0$), then
\bqn
e(\rho)=\eb(\rho)=\tau\left(\prod_{i=1}^g{\big[{\rho(a_i)},{\rho(b_i)}\big]^{ \widetilde\,}}\right)\,.
\eqn
\item if $S$ is noncompact (that is $n\geq1$), then
\bqn
\eb(\rho)=-\sum_{i=1}^n\tau\big(\tilde\rho(c_i)\big)\,,
\eqn
where $\tilde\rho:\pi_1(S)\to\Hh_\ZZ^+(\RR)$ denotes a homomorphism lifting $\rho$.
\end{enumerate}
\end{theorem}

Now we will turn to the description of the bounded Toledo number.
Its definition is based on the use of a very general operation in bounded cohomology
called "transfer", together with a description of the second bounded cohomology 
of $G=\PSU(1,1)$.

Let $\Gamma<G$ be a lattice in $G$.  One has the isomorphism
\bq\label{eq:eckmann_shapiro}
\hb^\bullet(\Gamma,\RR)\cong\hcb^\bullet\big(G,\linfty(\Gamma\backslash
G)\big) \eq analogous to the Eckmann--Shapiro isomorphism in ordinary
cohomology. Here $\hcb^\bullet$ denotes the bounded continuous
cohomology for whose definition the reader is referred to
\cite{Monod_book} or also
\cite[\S~2.3]{Burger_Iozzi_Wienhard_toledo}. Thus the bounded
cohomology of the discrete group $\Gamma$ can be computed via the
bounded continuous cohomology of the ambient Lie group $G$, but at the
expense of replacing the trivial $\Gamma$-module $\RR$ by the quite
intractable $G$-module $\linfty(\Gamma\backslash G)$.  This principle
is very general and does not require $\Gamma$ to be a lattice, but
this hypothesis will now allow us to "simplify" the coefficients:
indeed, let $\mu$ be the $G$-invariant probability measure on
$\Gamma\backslash G$. Then \bq\label{eq:5.6} \ba
\linfty(\Gamma\backslash G)&\longrightarrow\qquad\RR\\
f\quad\,\,&\longmapsto \int_{\Gamma\backslash G}f(x)\d\mu(x) \ea \eq
is a morphism of $G$-modules, where $\RR$ is then the trivial
$G$-module. Composing the induction isomorphism
\eqref{eq:eckmann_shapiro} with the morphism in cohomology induced by
the morphism of coefficients \eqref{eq:5.6} and specializing to degree
2 leads to a map, called the {\em transfer map}\index{transfer map}
\bqn \Tb:\hb^2(\Gamma,\RR)\to\hcb^2(G,\RR) 
\eqn 
which is linear and norm decreasing.
The interest of this construction lies in the fact that, while
$\hb^2(\Gamma,\RR)$ is infinite dimensional, say when $G$ is a real
rank one group, the space $\hcb^2(G,\RR)$ is finite dimensional if $G$
is a connected Lie group and in fact one dimensional for
$G=\PSU(1,1)$.  Considering the cocycle in \S~\ref{subsec:rel}
defining the K\"ahler class, we see that $c$ is bounded by $\frac12$,
as the area of geodesic triangles in $\DD$ is bounded by $\pi$, and
therefore we can use $c$ to define a bounded continuous class
$\kgb\in\hcb^2(G,\RR)$ called the bounded K\"ahler class\index{bounded
  K\"ahler class}\index{K\"ahler class!bounded}.

We have then:
\begin{proposition}\label{prop:4.7} Let $G=\PSU(1,1)\hookrightarrow\homeo{\partial\DD}$
be the natural inclusion.
Then
\begin{enumerate}
\item  $\hcb^2(G,\RR)=\RR\kgb$;
\item The restriction $\erb|_G$ to $G$ of the real bounded Euler class equals 
the bounded K\"ahler class $\kgb$ in $\hcb^2(G,\RR)$.
\end{enumerate}
\end{proposition}

The first assertion is in fact a very special case of a more general result 
and we will treat this later in its proper context; suffices it to say here that 
we already know that the comparison map 
\bqn
\hcb^2(G,\RR)\to\hc^2(G,\RR) = \RR\kg
\eqn
is surjective as $\kgb$ is sent to $\kg$;
the kernel of this map is then described by the space of continuous quasimorphisms
on $\PSU(1,1)$ and it is easy to see that they must be bounded.  
Hence the comparison map is injective.

For the second statement one needs an explicit relation between the cocycle $\epsilon$
used to define the Euler class and the orientation cocycle on $S^1$.  
Recall that the {\em orientation cocycle}\index{orientation cocycle}
\bqn
\orn:S^1\times S^1\times S^1\to\ZZ
\eqn
is defined by 
\bqn
\orn(x,y,z):=\begin{cases}
1&\hbox{ if }x,y,x\hbox{ are cyclically positively oriented}\\
0&\hbox{ if at least two coordinates coincide}\\
-1&\hbox{ if }x,y,x\hbox{ are cyclically negatively oriented.}
\end{cases}
\eqn
A formula relating the orientation cocycle directly to hyperbolic geometry is given by
\bqn
\orn(x,y,z)=\frac{1}{\pi}\int_{\Delta(x,y,z)}\omega_\DD
\eqn
where now $x,y,z\in\partial\DD$ and $\Delta(x,y,z)$ denotes the oriented geodesic
ideal triangle with vertices $x,y,z$.
Here we have taken $\partial\DD$ as a model of $S^1$
with the identification 
\bqn
\ba
\ZZ\backslash\RR&\to\partial\DD\\
t&\mapsto e^{2\pi\imath t}
\ea
\eqn
and we denote again by $\epsilon$ the corresponding cocycle on $\homeo{\partial\DD}$.

\begin{lemma}[\cite{Iozzi_ern}]\label{lem:4.8} For $f,g\in\homeo{\partial\DD}$
\bqn
\epsilon(f,g)=-\frac12\orn\big(1,f(1),fg(1)\big)+d\beta(f,g)\,,
\eqn
where 
\bqn
\beta(f):=\begin{cases}
\,\,\,\,\,0&\hbox{ if } f(1)=1\\
-\frac12&\hbox{ if }f(1)\neq1\,.
\end{cases}
\eqn
\end{lemma}
The proof of Proposition~\ref{prop:4.7}(2) is now straightforward:
using Stokes' theorem one shows that 
\bqn
(g_1,g_2)\mapsto c(g_1,g_2)=\frac1{2\pi}\int_{\Delta(0,g_10,g_1g_20)}\omega_\DD
\eqn
and 
\bqn
(g_1,g_2)\mapsto\frac{1}{2\pi}\int_{\Delta(1,g_11,g_1g_21)}\omega_\DD
\eqn
are cohomologous in the complex of bounded (Borel) cochains;
since the second cocycle is then essentially $\frac12\orn\big(1,g_1(1),g_1g_2(1)\big)$,
Lemma~\ref{lem:4.8} allows to conclude.

Now we are in the position to define the {\em bounded Toledo number}\index{bounded Toledo number}\index{Toledo number!bounded}.
Define, using Proposition~\ref{prop:4.7}(1), the linear form
$\tb:\hb^2(\Gamma,\RR)\to\RR$,
\bqn
\Tb(\alpha)=\tb(\alpha)\kgb\,.
\eqn
Then given a surface $S$ of finite topological type as before, 
fix a hyperbolization, i.e. a homomorphism corresponding to a complete hyperbolic structure on $S$, 
$h:\pi_1(S)\to G$ with image a lattice
$\Gamma=h\big(\pi_1(S)\big)$ in $G$.
Given now any homomorphism
$\rho:\pi_1(S)\to\homeo{S^1}$,
we define the bounded Toledo number of $\rho$ as
\bqn
\Tb(\rho,h):=\tb\big((\rho\circ h^{-1})^\ast(\erb)\big)\,.
\eqn
Observe that the hyperbolization $h$ is involved in the definition,
but we will see that $\Tb(\rho,h)$ is independent of $h$ 
as a consequence of the relation between the bounded Toledo and the bounded Euler numbers.

Concerning this relation consider the following diagram

\bqn
\xymatrix{
 \hb^2\big(\pi_1(S),\RR\big)\ar[r]^\cong
&\hb^2(\Sigma,\RR)
&\hb^2(\Sigma,\partial\Sigma,\RR)\ar[l]^{f_\RR}\ar[d]^{\<\,\cdot\,,[\Sigma,\partial\Sigma]\>}\\
\hb^2(\Gamma,\RR)\ar[u]^{h^\ast}\ar[d]_{\tb}
& &\RR\\
\RR& & 
}
\eqn
where, as before, $h:\pi_1(S)\to\Gamma$ is a hyperbolization with finite area.

\begin{theorem}[{\cite[Theorem~3.3]{Burger_Iozzi_Wienhard_toledo}}]\label{thm:4.9} 
For every $\alpha\in\hb^2\big(\pi_1(S),\RR\big)\cong\hb^2(\Sigma,\RR)$,
we have that 
\bqn
\tb\big((h^\ast)^{-1}(\alpha)\big) \big|\chi(S)\big|= \big\<f_\RR^{-1}(\alpha),[\Sigma,\partial\Sigma]\big\>\,.
\eqn
\end{theorem}

Specializing to $\alpha = \rho^*(e^{\rm b})$ this provides the desired
equality between the bounded Toledo number and the bounded Euler
number.

\subsection{Computations in bounded cohomology}\label{subsec:comp_bdd_coh}
In computing bounded cohomology one faces {\em a priori} 
the same difficulties as for the usual cohomology, 
namely that the bar resolution contains many coboundaries;
the ideal situation then would be if one had a complex giving bounded
cohomology and where all differentials are zero.
While this can be achieved for various ordinary cohomology theories,
we do not know of an analogue of either Hodge theory or Van Est isomorphism
for bounded cohomology. What can be achieved for the moment is a good
model in degree two.  
This follows from the theory developed in \cite{Burger_Monod_JEMS, Burger_Monod_GAFA, Monod_book}
and of which we recall a few consequences in our case at hand.

\begin{proposition}\label{prop:4.9}  Let $G=\PSU(1,1)$ and $L<G$ a closed subgroup whose
action on $\partial\DD\times\partial\DD$ is ergodic.  Then there is a canonical isomorphism
\bqn
\hcb^2(L,\RR)\cong\Zz\la\big((\partial\DD)^3\big)^L\,.
\eqn
\end{proposition}

Here 
\bqn
\ba
\Zz\la\big((\partial\DD)^3\big)^L:=
\big\{f:(\partial\DD)^3\to\RR:\, 
f\hbox{ is measurable, essentially bounded, }&\\ 
\hbox{alternating, $L$-invariant and }&\\
f(x_2,x_3,x_4)-f(x_1,x_3,x_4)+f(x_1,x_2,x_4)-f(x_1,x_2,x_3)=0&\\
\hbox{ for a.\,e. }(x_1,x_2,x_3,x_4)\in(\partial\DD)^4&
\big\}\,.
\ea
\eqn

In particular for $L=G=\PSU(1,1)$ it is plain that $\Zz\la\big((\partial\DD)^3\big)^L$ is one-dimensional, 
generated by the orientation cocycle.  In addition one can verify that under
the isomorphism in Proposition~\ref{prop:4.9}, 
$\kgb$ is sent to $\frac12\orn$.  This implies immediately the following

\begin{corollary}\label{cor:4.10} $\|\kgb\|=\frac12$.
\end{corollary}

This, in turn, together with the fact that $\Tb$ is norm decreasing, implies:

\begin{corollary}\label{cor:4.11} For every $\alpha\in\hb^2(\Gamma,\RR)$,
\bqn
\big|\tb(\alpha)\big| \leq2\|\alpha\|\,.
\eqn
\end{corollary}

Another feature of this model for bounded cohomology is that the transfer $\Tb$
takes a particularly simple and useful form:

\begin{proposition}[\cite{Monod_book}]\label{prop:4.12} Let $\Gamma$ be a lattice and $\mu$
the $G$-invariant probability measure on $\Gamma\backslash G$.  Then
\bqn
\Tb:\hb^2(\Gamma,\RR)\to\hb^2(G,\RR)
\eqn
is given by the map
\bqn
\ba
\Zz\la\big((\partial\DD)^3\big)^\Gamma&\to\Zz\la\big((\partial\DD)^3\big)^G\\
\alpha\qquad\quad&\longmapsto\quad\Tb(\alpha)\,,
\ea
\eqn
where
\bqn
\Tb(\alpha)(x,y,z)=\int_{\Gamma\backslash G}\alpha(gx,gy,gz)\,d\mu(g)\,.
\eqn
In particular
\bqn
\int_{\Gamma\backslash G} \alpha(gx,gy,gz)\,d\mu(g)=\frac{\tb(\alpha)}{2}\orn(x,y,z)
\eqn
for almost every $(x,y,z)\in(\partial\DD)^3$.
\end{proposition}

With this at hand we can now deduce a characterization of the bounded K\"ahler class
which lies at the heart of our approach:

\begin{theorem}\label{thm:4.13} Let $\Gamma<G $ be a lattice.  For every $\alpha\in\hb^2(\Gamma,\RR)$
\bqn
\big|\tb(\alpha)\big|\leq2\|\alpha\|
\eqn
with equality if and only if $\alpha$ is proportional to the restriction
$\kgb|_\Gamma$ to $\Gamma$ of the bounded K\"ahler class.
\end{theorem}

\begin{proof} Taking up the formula in Proposition~\ref{prop:4.12} in terms of measurable cocycles, 
$\big|\tb(\alpha)\big|=2\|\alpha\|$ reads 
\bqn
\int_{\Gamma\backslash G}\alpha(gx,gy,gz)\,d\mu(x)=\|\alpha\|_\infty\orn(x,y,z)\,,
\eqn
or 
\bqn
\int_{\Gamma\backslash G}\big(\|\alpha\|_\infty\orn(gx,gy,gz)-\alpha(gx,gy,gz)\big)\,d\mu(g)=0\,.
\eqn
For positively oriented triples $(x,y,z)$ this implies that for almost every $g$
\bqn
\|\alpha\|_\infty\orn(gx,gy,gz)=\alpha(gx,gy,gz)
\eqn
and hence $\alpha=\|\alpha\|_\infty \orn$ in $\Zz\la\big((\partial\DD)^3\big)^G$.
\end{proof}

Let $\ksrb\in\hb^2\big(\pi_1(S),\RR\big)$ denote the class obtained by considering 
 $\ksb\in\hb^2\big(\pi_1(S),\ZZ\big)$ (see the end of \S~\ref{subsec:bdd_eul}) as a real class.
 Using the results of the previous section, we obtain the following important 
 characterization of $\ksrb$.


\begin{corollary}\label{cor:4.14} Let $S$ be of finite topological type realized
as the interior of $\Sigma$.  Then for every $\alpha\in\hb^2\big(\pi_1(S),\RR\big)=\hb^2(\Sigma,\RR)$
\bqn
\big|\big\<f_\RR^{-1}(\alpha),[\Sigma,\partial\Sigma]\big\>\big|\leq2\|\alpha\|\,\big|\chi(S)\big|\,,
\eqn
with equality if and only if $\alpha$ is a multiple of $\ksrb$.
\end{corollary}

\subsection{Hyperbolic structures and representations: \\ the noncompact case}\label{sec:hyp_str_nc}
In this section we fulfill the promise to give explicit equations for the image 
in $\hom\big(\pi_1(S),G\big)$ of $\hyp(S)$ under the map $\delta$,
in the case where $S$ is a surface of finite topological type.

Let thus $\rho:\pi_1(S)\to\homeo{S^1}$ be a homomorphism; 
then we have the following result which is a first characterization 
of the maximality of the bounded Euler number of $\rho$.

\begin{corollary}\label{cor:4.15} We have
\bqn
\big|\eb(\rho)\big|\leq\big|\chi(S)\big|\,,
\eqn
where equality holds if and only if $\rho^\ast(\erb)=\pm\ksrb$.
\end{corollary}
\begin{proof} Combine Theorem~\ref{thm:4.9} with Corollary~\ref{cor:4.14}.
\end{proof}

In fact, if $S$ has $n$ punctures and is of genus $g$, then we have that
$\eb(\rho)=2g-2+n$ if and only if $\rho^\ast(\erb)=\ksrb$;
observe that for every $h\in\delta\big(\hyp(S)\big)$,
we have $\ksrb=h^\ast(\erb)$, so that $\rho$ and $h$ have 
the same real bounded Euler class. 
Keeping in mind that $\rho^\ast(\eb)\in\hb^2\big(\pi_1(S),\ZZ\big)$
determines $\rho$ up to semiconjugacy, this is a rather strong conclusion 
and in fact we have the following

\begin{theorem}\label{thm:4.16}  Let $\rho:\pi_1(S)\to\homeo{S^1}$ be a homomorphism with
\bqn
\rho^\ast(\erb)=\ksrb\,.
\eqn
Then $\rho$ is semiconjugate to an element in $\delta\big(\hyp(S)\big)$.
In particular
\bqn
\delta\big(\hyp(S)\big)=\big\{\rho\in\hom\big(\pi_1(S),\homs\big):\,\rho^\ast(\erb)=\ksrb\big\}\,.
\eqn
\end{theorem}

Note that Theorem~\ref{thm:4.16} combined with
Corollary~\ref{cor:4.15} gives a generalization of Matsumoto's theorem
proved in \cite{Matsumoto} for compact surfaces to surfaces of finite
topological type (see also \cite{Iozzi_ern} for a different proof in
the case when $S$ is a compact surface).

\begin{theorem}\label{thm:matsumoto}
Let $S$ be a surface of finite type and let $\rho_i: \pi_1(S) \to \homeo{S^1}$, $ i=1,2$, be 
homomorphisms with 
\bqn
\big|\eb(\rho_i)\big|=\big|\chi(S)\big|\,.
\eqn 
Then $\rho_1$ and $\rho_2$ are semiconjugate. 

In particular, every $\rho: \pi_1(S) \to \homeo{S^1}$ with
$\big|\eb(\rho)\big|=\big|\chi(S)\big|$is injective with discrete
image.
\end{theorem}

Below we will give the proof of this theorem in the case in which 
$\rho$ takes values in $G = \PSU(1,1)$.  In general 
this theorem follows immediately from a recent result in \cite{Burger_zimmer}
stating that the real bounded Euler class of a group homomorphism 
is a complete invariant of semiconjugacy provided its Gromov norm equals $\frac12$.

\begin{proof} Let $h,\rho:\pi_1(S)\to G$ be homomorphisms, and suppose that 
$h$ is a hyperbolization of finite area and $\rho$ satisfies the hypotheses of the theorem.
Then 
\bqn
\rho^\ast(\erb)=h^\ast(\erb)\,.
\eqn
Consider now the exact sequence
\bqn
\xymatrix{
 \hom\big(\pi_1(S),\RR/\ZZ\big)\ar[r]^b
&\hb^2\big(\pi_1(S),\ZZ\big)\ar[r]
&\hb^2\big(\pi_1(S),\RR\big)
}
\eqn
in order to conclude that there is a homomorphism 
$\chi:\pi_1(S)\to\RR/\ZZ$ with 
\bq\label{eq:4.5.1}
\rho(\eb)-h^\ast(\eb)=b(\chi)\,.
\eq
We now proceed to show that $\chi$ is trivial.
Since $\chi|_{[\Gamma,\Gamma]}=0$, we deduce that
\bqn
\big(\rho|_{[\Gamma,\Gamma]}\big)^\ast(\eb)=\big(h|_{[\Gamma,\Gamma]}\big)^\ast(\eb)
\eqn
and hence, by Ghys' theorem, there exists a weakly monotone map 
$\varphi:\partial\DD\to\partial\DD$
with 
\bqn
\varphi\big(\rho(\eta)x\big)=h(\eta)\big(\varphi(x)\big)\,,
\eqn
for all $\eta\in\big[\pi_1(S),\pi_1(S)\big]$ and all $x\in\partial\DD$.

If now $\varphi$ had a point of discontinuity, there would be
a nonempty open interval in the complement of $\im\varphi$ and \
in particular $\overline{\im\varphi}\neq\partial\DD$;
but $h\big(\pi_1(S)\big)$ and hence $h\big([\pi_1(S),\pi_1(S)]\big)$
act minimally on $\partial\DD$ and 
$\overline{\im\varphi}$ is invariant under the latter subgroup,
which is a contradiction.
Therefore $\varphi$ is continuous and surjective.
This implies in a straightforward way that $\rho\big([\pi_1(S),\pi_1(S)]\big)$
is a discrete subgroup of $G$; 
since the limit set of $\rho\big([\pi_1(S),\pi_1(S)]\big)$
is either a Cantor set or $\partial\DD$, we deduce that 
$\rho\big([\pi_1(S),\pi_1(S)]\big)$, 
and hence $\rho\big(\pi_1(S)\big)$, is Zariski dense.
Thus $\rho\big(\pi_1(S)\big)$ is either dense or discrete in $G$
but since it normalizes a nontrivial discrete subgroup and 
since $G$ is simple, $\rho\big(\pi_1(S)\big)$ must be discrete.

Restricting the equality $\rho(\eb)-h^\ast(\eb)=b(\chi)$ to a cyclic subgroup we deduce that 
\bqn
\rot\rho(\gamma)-\rot h(\gamma)=\chi(\gamma) 
\eqn
for all $\gamma\in\pi_1(S)$.
But since $h(\gamma)$ has at least one fixed point in $\partial\DD$, 
we have that $\rot h(\gamma)=0$ for all $\gamma\in\pi_1(S)$, and hence
\bqn
\rot\rho(\gamma)=\chi(\gamma)
\eqn
for all $\gamma\in\pi_1(S)$.
In particular, $\ker\rho\subset\ker\chi$ and hence
$\rho|_{\ker\rho}$ is semiconjugate to $h|_{\ker\rho}$,
which implies that $h(\ker\rho)$ has a fixed point in $\partial\DD$.
But since $h$ is a hyperbolization and $\ker\rho$ is normal in $\pi_1(S)$,
we deduce that $h(\ker\rho)$ is trivial and hence $\ker\rho$ is trivial,
thus showing that $\rho$ is injective.

Now we show that $\chi$ is trivial.  Let $\gamma\in\pi_1(S)$.  We distinguish then three cases:
\begin{enumerate}
\item {\em $\rho(\gamma)$ is hyperbolic or parabolic}.
Hence $\rho(\gamma)$ has a fixed point in $\partial\DD$ and hence 
$\chi(\gamma) = \rot\rho(\gamma)=0$;
\item {\em $\rho(\gamma)$ is elliptic and $\rot\rho(\gamma)=\chi(\gamma)\notin\QQ/\ZZ$}.
Then $\rho(\gamma)$ is conjugate in $G$ to an irrational rotation 
contradicting the fact that $\rho\big(\pi_1(S)\big)$ is discrete;
\item {\em $\rho(\gamma)$ is elliptic and $\rot\rho(\gamma)=\chi(\gamma)\in\QQ/\ZZ$}.
Let $n\in\NN$ be such that $n\chi(\gamma)=0$.  Hence $\rot\rho(\gamma^n)=\chi(\gamma^n)=0$
and, since $\rho(\gamma^n)$ is elliptic, it is hence the identity.
Thus $\gamma^n\in\ker\rho=e$ and since $\pi_1(S)$ has no torsion, $\gamma=e$.
\end{enumerate}
Thus we conclude that $\chi$ is trivial, $\rho^\ast(\eb)=h^\ast(\eb)$ and 
hence $\rho$ is semiconjugate to $h$.
\end{proof}

Now we will put to use the explicit formula for $\eb(\rho)$ 
together with the above result in order to restore, in a sense,
the setting of the case of surfaces without boundary and 
interpret $\delta\big(\hyp(S)\big)$ as a union of connected components.
Namely let us introduce
\bqn
\ba
\hom_\Cc\big(\pi_1(S),G\big)\hphantom{XXXXXXXXXXXXXXXXXXXXXXXXXXX}&\\
=\big\{\rho\in\hom\big(\pi_1(S),G\big):\,\rho(c_i)
\hbox{ has at least one fixed point in }\partial\DD\big\}&\\
=\big\{\rho:\pi_1(S)\to G:\,\big(\tr\rho(c_i)\big)^2\geq4,\hbox{ for }1\leq i\leq n\big\}&
\ea
\eqn
which is a real semialgebraic subset of $\hom\big(\pi_1(S),G\big)$.
Clearly for $\rho\in\hom_\Cc\big(\pi_1(S),G\big)$ 
we have that $\rot\rho(c_i)=0$, 
and hence taking into account that $\rho\mapsto\eb(\rho)$ is continuous, 
we have the following

\begin{corollary}\label{cor:4.17} If $\rho\in\hom_\Cc\big(\pi_1(S),G\big)$, then
\bqn
\eb(\rho)=-\sum_{i=1}^n\tau\big(\tilde\rho(c_i)\big)
\eqn
takes integer values and is constant on connected components.
\end{corollary}

This is in contrast with the fact that on $\hom\big(\pi_1(S),G\big)$ and when $S$ is not compact, 
the image of $\rho\mapsto\eb(\rho)$ is the whole interval $\big[-|\chi(S)|,|\chi(S)|\big]$.
In any case we obtain finally:

\begin{theorem}\label{thm:4.18} In the notation of Theorem~\ref{thm:4.6}, we have 
\bqn
 \delta\big(\hyp(S)\big)
=\left\{
\rho:\pi_1(S)\to G:\,\sum_{i=1}^n\tau\big(\widetilde\rho(c_i)\big)=2g-2+n
\right\}\,.
\eqn
Thus $\delta\big(\hyp(S)\big)$ is a union of connected components of $\hom_\Cc\big(\pi_1(S),G\big)$
and, in particular, a semialgebraic set.
\end{theorem}

\subsection{Relation with quasimorphisms}\label{sec:scl}
In connection with Corollary~\ref{cor:4.14}, we would like to present a different viewpoint, 
coming from Ch.~Bavard \cite{Bavard} and developed by D.~Calegari \cite{Calegari_scl},
which relates the second bounded cohomology of $\pi_1(S)$ to the stable commutator length\index{stable commutator length} (scl)
via quasimorphisms.  In order to simplify the discussion, we assume in this section that $\Gamma$
is a group with 
\bq\label{eq:vanishing}
\h^2(\Gamma,\RR)=0\,.
\eq
This applies in particular to $\Gamma=\pi_1(S)$, where $S$ is a non-compact surface.
In the notation of \S~\ref{subsec:bdd_eul}, every class in $\hcb^2(\Gamma,\RR)$ admits then a representative 
of the form $d^1f$, with $f\in\C^1(\Gamma,\RR)$. This leads us to make the following definition

\begin{definition}\label{def:qm}  A quasimorphism\index{quasimorphism} on $\Gamma$ is a function $f:\Gamma\to \RR$ such that
\bqn
D(f):=\sup_{a,b\in\Gamma}\big|f(ab)-f(a)-f(b)\big|<+\infty\,,
\eqn
and $D(f)$ is called the {\em defect}\index{defect} of $f$.
\end{definition}

The vector space $Q(\Gamma,\RR)$ of all quasimorphisms on $\Gamma$ contains always the subspace $\ell^\infty(\Gamma,\RR)$
of bounded functions, as well as the subspace of all homomorphisms $\hom(\Gamma,\RR)$.
It is clear that $d^1$ induces an isomorphism of vector spaces
\bqn
\xymatrix{
 Q(\Gamma,\RR)/\ell^\infty(\Gamma,\RR)\oplus\hom(\Gamma,\RR)\ar[r]^-\cong
&\hb^2(\Gamma,\RR)}
\eqn
and the Gromov norm\index{Gromov norm} $\|\alpha\|$ of a class $\alpha\in\hcb^2(\Gamma,\RR)$ is given by
\bqn
\|\alpha\|=\inf\big\{D(f):\,[d^1f]=\alpha,\,f\in Q(\Gamma,\RR)\big\}\,.
\eqn
There are various ways of choosing a "special" quasimorphism representing a given class.
One is to fix a finite symmetric generating set of $\Gamma$ and establish the existence of a harmonic representative
$F$ for each class $\alpha\in\hcb^2(\Gamma,\RR)$: this quasimorphism then minimizes the defect, that is 
\bqn
\|\alpha\|=D(f)\,,
\eqn
(\cite{Burger_Monod_JEMS}; see also  \cite{Bjorklund_Hartnick} for a recent application).

Another way to get a representative, this time canonical, is to consider homogeneous quasimorphisms,
that is quasimorphisms satisfying the condition
\bqn
f(x^n)=nf(x)\qquad\text{ for all }n\in\ZZ,\,,x\in\Gamma\,.
\eqn
A simple argument (see \cite{Polya_Szego}), shows that for $f\in Q(\Gamma,\RR)$, the limit
\bqn
F(x):=\lim_{n\to\infty}\frac{f(x^n)}{n}
\eqn
exists for all $x$ and gives a homogeneous quasimorphism $F$ with the property that $f-F\in\ell^\infty(\Gamma,\RR)$.
Denoting by $Q_{\mathrm h}(\Gamma,\RR)$ the subspace of homogeneous quasimorphisms, 
it is thus clear that $d^1$ induces an isomorphism of vector spaces
\bqn
\xymatrix{
Q_{\mathrm h}(\Gamma,\RR)/\hom(\Gamma,\RR)\ar[r]^-\cong
&\hb^2(\Gamma,\RR)\,.}
\eqn

Obviously the defect $D$ gives a norm on the left hand side, while the right hand side is endowed with the (canonical) Gromov norm\index{Gromov norm}.
There is then the following non-trivial relation between these two norms:

\begin{theorem}[\cite{Calegari_scl}]\label{thm:norms} For every homogeneous quasimorphism $f$ on $\Gamma$,
\bqn
\frac12 D(f)\leq\big\|d^1f\big\|\leq D(f)\,.
\eqn
\end{theorem}

This inequality is based on the following relation between defect of a homogeneous quasimorphism $f$ and commutators
\bqn
D(f)=\sup_{a,b\in\Gamma}\big|f\big([a,b]\big)\big|\,.
\eqn

\begin{example} \begin{enumerate}
\item The function $\tilde r:\widetilde{\PU(1,1)}\to\RR$ in \S~\ref{subsec:miln_gold} leading to Milnor's inequality is a quasimorphism.
\item The translation number $\tau:\Hh_\ZZ^+(\RR)\to\RR$ defined in \S~\ref{subsec:bdd_eul} is a continuous quasimorphism.  
Identifying $\widetilde{\PU(1,1)}$ with a subgroup of $\Hh_\ZZ^+(\RR)$, we have 
\bqn
D(\tau)=D\big(\tau|_{\widetilde{\PU(1,1)}}\big)=1\,.
\eqn
\end{enumerate}
\end{example}

We let as usual $[\Gamma,\Gamma]$ denote the subgroup of $\Gamma$ generated by the set 
$\big\{[x,y]:\,x,y\in\Gamma\big\}$  of all commutators and let $\mathrm{cl}(\gamma)$ denote the word length
with respect to this generating set.  The stable commutator length\index{stable commutator length}
is defined by
\bqn
\mathrm{scl}(\gamma):=\lim_{n\to\infty}\frac{\mathrm{cl}(\gamma^n)}{n}\,,
\eqn
where the existence of the limit follows from the fact that the map $n\mapsto \|\gamma^n\|$ is subadditive.

The following result then puts the geometry of the commutator subgroup $[\Gamma,\Gamma]$ in direct relation
with quasimorphisms, in fact, homogeneous ones:

\begin{theorem}[\cite{Bavard}]\label{thm:bavard}  For every $\gamma\in[\Gamma,\Gamma]$, we have
\bqn
\mathrm{scl}(\gamma)=\sup\left\{\frac{|\varphi(\gamma)|}{2D(\varphi)}:\,\varphi\in Q_{\mathrm{h}}(\Gamma,\RR)\right\}\,.
\eqn
\end{theorem}

Now every element $\gamma\in[\Gamma,\Gamma]$, seen as a $1$-chain, is an element in the vector space $\mathrm{B}_1(\Gamma,\RR)$
of $1$-boundaries in the bar resolution defining group homology.  In \cite{Calegari_faces} the author extends 
$\mathrm{scl}$ to a seminorm on the vector space $\mathrm{B}_1(\Gamma,\RR)$ and obtains an extension of the Bavard duality theorem.
We refer to the monograph \cite{Calegari_scl} for the details and interesting developments, and proceed directly to state a corollary
of the main result in \cite{Calegari_faces}.  If now $\Gamma=\pi_1(S)$, where $S$ is a non-compact surface and
  \bqn
  \pi_1(S)=\left\<a_1,b_1,\dots,a_g,b_g,c_1,\dots,c_n:
  \prod_{i=1}^g[a_i,b_i]\prod_{j=1}^nc_j=e\right\> \,,
  \eqn 
  then $\sum_{i=1}^n c_i$ is clearly in $\mathrm{B}_1(\Gamma,\RR)$ and
  
  \begin{proposition}[\cite{Calegari_faces}] With the above assumptions and notation
  \bqn
  \mathrm{scl}\left(\sum_{i=1}^n c_i\right)=\frac{-\chi(S)}{2}\,.
  \eqn
  \end{proposition}
  If now $\tilde\rho:\Gamma\to\widetilde{\PU(1,1)}$ is the lift of a fixed hyperbolization $\rho$ of $S$,
  we can pullback the translation quasimorphism and obtain $\rot_\rho:=\tau\circ\rho$,
  which defines on $\Gamma$ a homogeneous quasimorphism taking values in $\ZZ$; 
  in fact, $\rot_\rho$ changes by an element of $\hom(\Gamma,\ZZ)$ if one takes a different lift of $\rho$.
  A corollary to the main result in \cite{Calegari_faces} is then the following
  
  \begin{theorem}[\cite{Calegari_faces}]\label{thm:6} For any homogeneous quasimorphism $f$ on $\Gamma$, we have the inequality
  \bqn
  \left| \sum_{i=1}^nf(c_i)\right|\leq D(f)\,|\chi(S)|\,,
  \eqn
  with equality  if and only if $f$ differs from $\rot_\rho$ by an element of $\hom(\Gamma,\RR)$.
  \end{theorem}
  
The relation with Corollary~\ref{cor:4.14} is the following.  Given a class $\alpha\in\hb^2\big(\pi_1(S),\RR\big)=\hb^2(\Sigma,\RR)$,
let $d^1F$ be a representative of $\alpha$, where $F:\Gamma\to\RR$ is a homogeneous quasimorphism.  Then:

\begin{lemma}  If $f_\RR:\hb^2(\Sigma,\partial\Sigma,\RR) \to \hb^2(\Sigma,\RR)$ is the isomorphism  in \S~\ref{subsec:bbd_eul_bdd_tol}, then
\bqn
\big\<f_\RR^{-1}(\alpha),[\Sigma,\partial\Sigma]\big\>=-\sum_{i=1}^nF(c_i)\,.
\eqn
\end{lemma}
Taking into account the inequality $D(F)\leq2\|\alpha\|$ in Theorem~\ref{thm:norms}, one sees that Theorem~\ref{thm:6} implies 
Corollary~\ref{cor:4.14}.  For our purposes however, both results contain the same information as far as the characterization of equality is concerned.  
An intriguing question in this context is whether if $\Gamma$ is a finitely generated group, $F:\Gamma\to\RR$ a homogeneous quasimorphism
and $[d^1F]\in\hb^2(\Gamma,\RR)$ the bounded class it defines, then the equality
\bqn
\big\|[d^1F]\big\|=\frac12 D(F)
\eqn
holds.

It would suffice to show this for nonabelian free groups; in this case all known examples of quasimorphisms satisfy the above
equality  (\cite{Calegari_scl, Rolli}).

\part{Higher Teichm\"uller Spaces\index{Higher Teichm\"uller space}}
Our considerations started with the study of how the set of hyperbolic structures
on a surface $S$, of finite topological type, is related to the set of representations 
of $\pi_1(S)$ into $G=\PSU(1,1)$; the problem of characterizing the image of the map
\bqn
\delta:\hyp(S)\to\hom\big(\pi_1(S),G\big)
\eqn
constructed in \S~\ref{sec:hyp_str} led us to introduce an invariant (see \S~\ref{sec:inv_miln_gold})
\bqn
\T:\hom\big(\pi_1(S),G\big)\to\RR
\eqn
with various incarnations and 
whose maximal fiber $\T^{-1}\big(-\chi(S)\big)$ coincides with the image of $\delta$.
As indicated in \S~\ref{sec:inv_miln_gold} this invariant can be defined for homomorphisms 
from $\pi_1(S)$ with values in any Lie group $G$.
This leads to the vague question of how much of the ``$\PSU(1,1)$ picture'' generalizes
to an arbitrary Lie group $G$.
Interestingly enough, there are two classes of (semi)simple Lie groups for which 
one can make this question precise in defining, in very different ways, components
(or specific subsets when $S$ is not compact) of $\hom\big(\pi_1(S),G\big)$
which should play the role of Teichm\"uller space: 
those two classes are on the one hand the {\em split real groups} (Definition~\ref{defi:split}) 
and, on the other hand, the {\em Lie groups of Hermitian type} (Definition~\ref{defi:5.1}).  

%
Because of the various properties these connected
components share with Teichm\"uller space, we will call them {\em
higher Teichm\"uller spaces}.  

We will now describe in some detail the class given by Lie groups
of Hermitian type, and the corresponding subset, namely the {\em
space of maximal representations} 
\bqn
  \RM \big(\pi_1(S), G \big) \subset \R \big(\pi_1(S), G \big)\,.
\eqn
%
%
We will discuss the other class in \S~\ref{subsec:hitchin_positive}.

\section{Maximal representations into Lie groups \\ of Hermitian type}\label{sec:max_rep}

\begin{definition}\label{defi:5.1} A Lie group $G$ is of 
{\em Hermitian type}\index{Hermitian type}\index{Lie group!Hermitian type} if it is connected
semisimple with finite center without compact factors and the associated symmetric
space $X$ has a $G$-invariant complex structure.
\end{definition}

In this setting we will be able to define a {\em Toledo invariant}\index{Toledo invariant}
(and a {\em bounded Toledo invariant}\index{bounded Toledo invariant}\index{Toledo invariant!bounded} 
when $S$ is not compact)
satisfying a Milnor--Wood type inequality\index{Milnor--Wood type inequality}\index{inequality!Milnor--Wood type} 
and this will lead us to consider 
the set $\hommax\big(\pi_1(S),G\big)$ of maximal representations into $G$.

\medskip

In this section we will describe a certain number of fundamental geometric properties 
of maximal representations and we will also have something to say 
about the structure of the set of such representations, 
all this in the context where $S$ is of finite topological type.


\subsection{The cohomological framework}\label{subsec:coh_fr}
Let $G$ be of Hermitian type, $X$ the associated symmetric space,
\bqn
\<\,\cdot\,,\,\cdot\,\>:X \to \Sym(TX)
\eqn
the Riemannian metric and 
\bqn
J:X \to \End(TX)
\eqn
the complex structure. Then $\<J\cdot\,,\,\cdot\,\>$ defines
a $G$-invariant Hermitian metric whose imaginary part $\omega_X$
is a real $G$-invariant $2$-form on $X$.
By a general lemma of E.~Cartan, the complex of $G$-invariant forms on any symmetric space $X$
consists of closed forms.  Thus $\omega_X$ is closed and the above Hermitian metric is K\"ahler.

Given $S$ compact and $\rho:\pi_1(S)\to G$ a homomorphism, 
we can then proceed as in \S~\ref{subsec:kaeh_tol} and with the help 
of a smooth equivariant map
\bqn
f:\widetilde S=D\to X
\eqn
define the {\em Toledo invariant}\index{Toledo invariant} of $\rho$,
\bqn
\T(\rho)=\frac{1}{2\pi}\int_Sf^\ast(\omega_X)\,.
\eqn
For the cohomological interpretation we can proceed as for $\PSU(1,1)$, 
namely, since $G$ acts properly on $X$, we have the Van Est isomorphism
\bqn
\Omega^2(X)^G\cong\hc^2(G,\RR)\,.
\eqn
To the K\"ahler form\index{K\"ahler form} $\omega_X$ we associate the continuous (inhomogeneous) $2$-cocycle
\bqn
c(g_1,g_2)=\frac{1}{2\pi}\int_{\Delta(x_0,g_1x_0,g_1g_2x_0)}\omega_X
\eqn
where $x_0\in X$ is a fixed base point and $\Delta(x,y,z)$ denotes a smooth simplex 
with geodesic sides connecting the vertices $x,y,z$; of course such a simplex is not unique,
but any two such simplices with fixed vertices have the same boundary and hence, 
since $\omega_X$ is closed, by Stokes' theorem the integral does not depend on it.  
We denote by $\kg\in\hc^2(G,\RR)$ the class defined by $c$ and 
call it the K\"ahler class.  
Then, given $\rho:\pi_1(S)\to G$, we have the equality
\bq\label{eq:7.1.1}
\T(\rho)=\big\<\rho^\ast(\kg),[S]\big\>\,.
\eq
In fact, the cocycle $c$ defining $\kg$ turns out to be bounded;
this is a consequence of a precise study of the K\"ahler area of triangles with geodesic sides,
due to Domic and Toledo \cite{Domic_Toledo}, and  Clerc and \O rsted \cite{Clerc_Orsted_2}, 
and which we will describe later in more details.
Here we deduce that $c$ defines a bounded class $\kgb\in\hcb^2(G,\RR)$, 
called the {\em bounded K\"ahler class} and we have the following

\begin{theorem}[{Clerc--\O rsted, \cite{Clerc_Orsted_2}}]\label{thm:5.2}
Assume that the metric on $X$ is normalized so that its minimal holomorphic
sectional curvature is $-1$.  Then the value of the Gromov norm
of the bounded K\"ahler class\index{bounded K\"ahler class}{K\"ahler class!bounded} is
\bqn
\|\kgb\|=\frac12\rk_X\,, 
\eqn
where $\rk_X$ is the rank of the symmetric space $X$. 
\end{theorem}
When $X$ is irreducible, $\Omega^2(X)^G=\RR\omega_X$ and hence the comparison map
\bq\label{eq:7.2}
\hcb^2(G,\RR)\to\hc^2(G,\RR)
\eq
is surjective; in general, $\Omega^2(X)^G$ is spanned by the pullbacks
under the projections of the K\"ahler forms of the irreducible factors of $X$ and
thus the comparison map \eqref{eq:7.2} is surjective as well; 
since $G$ has finite center (see Definition~\ref{defi:5.1}), it is injective in all cases.

Let now $S$ be of finite topological type realized as the interior of an oriented
compact surface $\Sigma$ with boundary and
$\rho:\pi_1(S)\to G$ a homomorphism.  
Then we can proceed as in the definition of the bounded Euler number and, 
in the notation of \S~\ref{subsec:bbd_eul_bdd_tol}, set
\bqn
\T(\rho):=\big\<f_\RR^{-1}\big(\rho^\ast(\kgb)\big),[\Sigma,\partial\Sigma]\big\>\, ,
\eqn
where we recall that
\bqn
f_\RR:\hb^2(\Sigma,\partial\Sigma,\RR)\to\hb^2(\Sigma,\RR)\,
\eqn
is the isomorphism given by the natural inclusion.

For $G=\PSU(1,1)$ we saw in \S~\ref{sec:surf_fin_eul}
that the K\"ahler class and the bounded K\"ahler class come from a (bounded) integral class,
namely the (bounded) Euler class, and this turned out to be essential 
in order to obtain explicit formulas for the Toledo invariant 
needed in particular when $S$ is noncompact. 
In the case of $\PSU(1,1)$ we have at our disposal the relation with rotation numbers 
and the translation quasimorphism\index{translation quasimorphism}\index{quasimorphism!translation}; 
these structures were given to us for free 
from the fact that $\PSU(1,1)$ acts by orientation preserving homeomorphisms of the circle.
For $G$ of Hermitian type one can construct (more sophisticated) analogues
of each of these objects;  in particular the integral structure 
on (bounded) cohomology and the analogues of rotation number can be described 
quite explicitly and this is what we turn to now.

\medskip
We denote by $\hc^2(G,\ZZ)$ and $\hcb^2(G,\ZZ)$ the (bounded) Borel cohomology, 
which is defined by considering the complex of (bounded) Borel functions from $G^n$ to $\ZZ$. 
We refer the reader to \cite[\S\S 2.3 and 7.2]{Burger_Iozzi_Wienhard_toledo} for more details. 

We have the following
\begin{lemma}\label{lem:5.3} The comparison map
\bqn
\hcb^2(G,\ZZ)\to\hc^2(G,\ZZ)
\eqn
is an isomorphism.
\end{lemma}
\begin{proof} The long exact sequences in cohomology associated to 
\bqn
\xymatrix{
 0\ar[r]
&\ZZ\ar[r]
&\RR\ar[r]
&\RR/\ZZ\ar[r]
&0
}
\eqn
read as 
\bqn
\xymatrix{
 0=\homc(G,\RR/\ZZ)\ar[r]\ar[d]^=
&\hcb^2(G,\ZZ)\ar[r]\ar[d]
&\hcb^2(G,\RR)\ar[r]\ar[d]
&\hcb^2(G,\RR/\ZZ)\ar[d]^=\\
 0=\homc(G,\RR/\ZZ)\ar[r]
&\hc^2(G,\ZZ)\ar[r]
&\hc^2(G,\RR)\ar[r]
&\hc^2(G,\RR/\ZZ)\,.
}
\eqn
The fact that the third vertical arrow is an isomorphism and the $5$-term lemma 
allow to conclude.
\end{proof}

Actually we will turn to an explicit implementation of the isomorphism in Lemma~\ref{lem:5.3}. 
This will also give
an alternative treatment of some material in \cite[\S~7]{Burger_Iozzi_Wienhard_toledo}.
We start with the observation that the isomorphism
\bq\label{eq:7.3}
\hc^2(G,\ZZ)\cong\hom\big(\pi_1(G),\ZZ\big)
\eq
is valid for any connected Lie group $G$; this follows easily 
from the fact that $\hc^2(G,\ZZ)$ classifies equivalence classes 
of topological central extensions of $G$ by $\ZZ$,
together with some covering theory; moreover this isomorphism is natural.
If now $K<G$ is a maximal compact subgroup, then by the Iwasawa decomposition
$K\hookrightarrow G$ is a homotopy equivalence and hence $\pi_1(K)=\pi_1(G)$;
this, together with \eqref{eq:7.3}, 
implies that the restriction map
\bqn
\hc^2(G,\ZZ)\to\hc^2(K,\ZZ)
\eqn
is an isomorphism.  Taking into account that continuous cohomology
of compact groups with real coefficients is trivial, 
we obtain, considering the long exact sequence associated to the coefficient sequence, that
\bqn
\xymatrix{
\homc(K,\RR/\ZZ)\ar[r]^-\epsilon&\hc^2(K,\ZZ)
}
\eqn
is an isomorphism.
As a result we obtain an isomorphism 
\bqn
\ba
\homc(K,\RR/\ZZ)&\longrightarrow\hc^2(G,\ZZ) \cong \hcb^2(G,\ZZ)\\
\chi\,\,\qquad&\longmapsto\qquad\qquad\kappa
\ea
\eqn
and we say that $\kappa$ corresponds to $\chi$ and viceversa.

Now we will assume that $G$ is real algebraic and semisimple.
In this case we have at our disposal the {\em refined Jordan decomposition}\index{refined Jordan decomposition}
namely every $g\in G$ is a product
\bqn
g=g_eg_hg_u
\eqn
of pairwise commuting elements, where $g_e$ is contained in a compact subgroup,
$g_h$ is in the connected component of the identity of a maximal real split torus
and $g_u$ is unipotent.  Given then
\bqn
\chi:K\to\RR/\ZZ
\eqn
a continuous homomorphism and denoting by $C(h)$ the conjugacy class of an element
$h\in G$, define for $g\in G$
\bqn
\chiext(g):=\chi\big(C(g_e)\cap K\big)
\eqn
Then, according to \cite{Borel_comm}, $\chiext$ is indeed 
a well defined {\em continuous} class function on $G$ extending $\chi$;
moreover it satisfies
\bqn
\chiext(g^n)=n\chiext(g)
\eqn
for all $n\in\ZZ$ and $g\in G$. Let $\widetilde\chiext:\widetilde G\to\RR$
denote the unique continuous lift to the universal covering $\widetilde G$ of $G$,
vanishing at $e$;  finally we denote by 
\bqn
\chi_\ast:\pi_1(G)=\pi_1(K)\to\ZZ
\eqn
the morphism on the level of fundamental groups induced by $\chi$. 
The following result then gives a precise description of the isomorphism in Lemma~\ref{lem:5.3}.

\begin{theorem}[\cite{Burger_Iozzi_Wienhard_toledo}]\label{thm:5.4} \be
\item The function $\widetilde\chiext:\widetilde G\to\RR$ is a homogeneous 
quasimorphism\index{homogeneous quasimorphism}\index{quasimorphism!homogeneous} and the map
\bqn
\ba
\homc(K,\RR/\ZZ)&\to\qhch(\widetilde G,\RR)_\ZZ\\
\chi\qquad&\longmapsto\quad\widetilde\chiext
\ea
\eqn
establishes an isomorphism with the space of continuous homogeneous quasimorphisms 
sending $\pi_1(G)$ to $\ZZ$; in  fact
\bqn
\chi_\ast=\widetilde\chiext|_{\pi_1(G)}\,.
\eqn
\item If $\kappa\in\hc^2(G,\ZZ)$ is the class corresponding to $\chi$, and denote by $[\cdot]$ the integer part, 
then 
\bqn
(g,h)\to\big[\widetilde\chiext(gh)\big]-\big[\widetilde\chiext(g)\big]
                 -\big[\widetilde\chiext(h)\big]\,
\eqn
descends to a well defined $\ZZ$-valued bounded cocycle on $G\times G$ 
representing the class $\kappa$.
\ee
\end{theorem}

\begin{remark}  The only not obvious statement in Theorem~\ref{thm:5.4} is the assertion that
$\widetilde\chiext$ is a quasimorphism; in fact, this follows easily from the fact (proved in \cite{Borel_comm})
that for every $k\in\NN$, $\widetilde\chiext$ is bounded on the elements in $\widetilde G$ 
which are products of $k$ commutators.  We mention  in addition that 
Theorem~\ref{thm:5.4} is also consequence of a different and more general approach taken 
in \cite{Burger_Iozzi_Wienhard_toledo}.
\end{remark}
In our context, {\em rotation numbers}\index{rotation number} arise in the following way.
Let $\kappa\in\hcb^2(G,\ZZ)$ and, for every $g\in G$ consider, as in \S~\ref{sec:surf_fin_eul},
the homomorphism
\bqn
\ba
h_g:\ZZ&\to G\\
n&\mapsto g^n
\ea
\eqn
by means of which we obtain a bounded integral class $h_g^\ast(\kappa)\in\hb^2(\ZZ,\ZZ)$
and finally by means of the canonical isomorphism
\bqn
\xymatrix{
0\ar[r]&\hom(\ZZ,\RR/\ZZ)\ar[r]^\delta&\hb^2(\ZZ,\ZZ)\ar[r]&0
}
\eqn
and element in $\RR/\ZZ$,
\bqn
\rot_\kappa(g):=\delta^{-1}\big(h_g^\ast(\kappa)\big)(1)\,.
\eqn

From standard homological considerations using the naturality of all constructions involved, 
one can deduce that if $\chi\in\homc(K,\RR/\ZZ)$ corresponds to the class $\kappa\in\hcb^2(G,\ZZ)$
then
\bqn
\rot_\kappa(g)=\widetilde\chiext(\tilde g)\mod\ZZ
\eqn
where $\tilde g\in\widetilde G$ is any lift of $g\in G$.
Now we fix $\kappa\in\hcb^2(G,\ZZ)$ and 
let as before $S$ be a surface of finite topological type
realized as the interior of $\Sigma$.
Given a homomorphism $\rho:\pi_1(S)\to G$, we define then
\bqn
\T_\kappa(\Sigma,\rho)=\big\<f_\RR^{-1}\big(\rho^\ast(\kappa)\big),[\Sigma,\partial\Sigma]\big\>\,.
\eqn
Of course if $\partial\Sigma=\emptyset$, that is if $S$ is compact,
the definition takes the simplified form
\bqn
\T_\kappa(S,\rho)=\big\<\rho^\ast(\kappa),[S]\big\>\,.
\eqn
Then we have the following:

\begin{theorem}[\cite{Burger_Iozzi_Wienhard_toledo}]\label{thm:5.5} 
Let $S$ be of finite topological type with presentation 
\bqn
\pi_1(S)=\left\<a_1,b_1,\dots,a_g,b_g,c_1,\dots,c_n:\prod_{i=1}^g[a_i,b_i]\prod_{j=1}^nc_j=e\right\>
\eqn
and $\rho:\pi_1(S)\to G$ a homomorphism.  Let  $\kappa\in\hcb^2(G,\ZZ)$ and $\chi\in\homc(K,\RR/\ZZ)$
the corresponding homomorphism. 
\be
\item Assume that $S$ is compact.  Then
\bqn
\T_\kappa(S,\rho)=-\chi_\ast\left(\prod_{i=1}^g[\rho(a_i),\rho(b_i)]^{\tilde\,}\right)\,,
\eqn
where $\chi_\ast:\pi_1(G)=\pi_1(K)\to\ZZ$ is the morphism induced by $\chi$ and 
$[\cdot, \cdot]^{\tilde\, }$ is the commutator map introduced in \S~\ref{subsec:hyp_surf_quasiconj} .

\item Assume that $S$ is not compact.  If 
$\widetilde\rho:\pi_1(S)\to\widetilde G$ is a lift of $\rho$ to $\widetilde G$, then
\bqn
\T_\kappa(\Sigma,\rho)=-\sum_{j=1}^n\widetilde\chiext\big(\widetilde\rho(c_j)\big)\,.
\eqn
\ee
\end{theorem}

\subsection{Maximal representations and basic geometric properties}\label{subsec:max_rep_geom_pr}
Let $S$ be of finite topological type and $\rho:\pi_1(S)\to G$ a homomorphism.
Based on our considerations in the case of $\PSU(1,1)$ in \S~\ref{sec:surf_fin_eul}
and the computation of the Gromov norm of the bounded K\"ahler class $\kgb\in\hcb^2(G,\RR)$
by Clerc and \O rsted, we can conclude

\begin{corollary}\label{cor:5.6}  The Toledo invariant $\T(\rho)$ defined in \eqref{eq:7.1.1}
satisfies the inequality
\bqn
\big|\T(\rho)\big|\leq\rk_X\big|\chi(S)\big|
\eqn
with equality if and only if
\bqn
\rho^\ast(\kgb)=\pm\rk_X\ksrb\,,
\eqn
where $\ksrb$ is the bounded real class defined in the context of Corollary~\ref{cor:4.14}.
\end{corollary}
\begin{proof}  Corollary~\ref{cor:4.14} with $\alpha=\rho^\ast(\kgb)$ implies 
that 
\bqn
\big|\T(\rho)\big|\leq2\big\|\rho^\ast(\kgb)\big\|\,\big|\chi(S)\big|\,.
\eqn
Using the fact that $\rho^\ast$ is norm decreasing and the value of
$\|\kgb\|$ (see Theorem~\ref{thm:5.2}) we obtain that
\bqn
\big|\T(\rho)\big|\leq\rk_X\big|\chi(S)\big|\,.
\eqn
Equality implies that 
\bqn
\big|\T(\rho)\big|=2\big\|\rho^\ast(\kgb)\big\|\,\big|\chi(S)\big|=\rk_X\big|\chi(S)\big|\,.
\eqn
It follows from the first equality that 
\bqn
\rho^\ast(\kgb)=\lambda\,\ksrb
\eqn
for some $\lambda\in\RR$ (Corollary~\ref{cor:4.14}), 
and from the second equality that $|\lambda|=\rk_X$.
\end{proof}

Thus we now introduce the following

\begin{definition}\label{def:5.7} A representation $\rho:\pi_1(S)\to G$ is
called {\em maximal}\index{maximal representation}\index{representation!maximal} if 
\bqn
\T(\rho)=\rk_X\big|\chi(S)\big|\,.
\eqn
\end{definition}

The basic example of a family of maximal representations is obtained 
via a geometric fact of fundamental importance called the polydisk theorem.
Recall that in a symmetric space of noncompact type there are maximal flat subspaces,
they are all $G$-conjugate and their common dimension is the rank $r=\rk_X$ of $X$.
When $X$ is Hermitian symmetric, a geometric version of a fundamental result of Harish-Chandra
says that the complexification of a maximal flat is a maximal polydisk,
or, in other words, that the image under the exponential map of the complexified tangent space of a maximal flat
is a totally geodesic holomorphic copy of $\DD^r$.
The fact that the normalized metric on $X$ is taken to be of minimal holomorphic 
sectional curvature $-1$ is equivalent to the property 
that any maximal polydisk embedding
\bqn
\varphi:\DD^r\to X
\eqn
is isometric.  To such a map corresponds a homomorphism
\bqn
\Phi:\SU(1,1)^r\to G
\eqn
with respect to which $\varphi$ is equivariant.
Given then $r$ hyperbolizations $\rho_1,\dots,\rho_r:\pi_1(S)\to\SU(1,1)$,
the homomorphism 
\bqn
\rho(\gamma)=\Phi\big(\rho_1(\gamma),\dots,\rho_r(\gamma)\big)
\eqn
is then maximal.  The main point is the fact that
\bqn
\Phi^\ast(\kgb)=\kappa_{\SU(1,1)^r}^{\rm b}\,.
\eqn
This follows from a Lie algebra computation which gives 
\bqn
\Phi^\ast(\omega_X)=\omega_{\DD^r}
\eqn
and the naturality of the isomorphisms
\bqn
\Omega^2(X)^G\cong\hc^2(G,\RR)\cong\hcb^2(G,\RR)\,.
\eqn
Observe that we could have taken an antiholomorphic embedding 
$\overline\varphi:\DD^r\to X$, in which case
\bqn
\overline{\rho}(\gamma)=\overline{\Phi}\big(\rho_1(\gamma),\dots,\rho_r(\gamma)\big)
\eqn
has then 
\bqn
\T(\overline\rho)=-\rk_X\,\big|\chi(S)\big|
\eqn
as Toledo invariant.

Now, the class $\kgb$ is not always integral, but there is a specific natural number $n_X$ 
depending on the root system of $G$ such that $\kappa = n_X\kgb$ is
an integral class.  From this and Theorem~\ref{thm:5.5} we deduce:

\begin{corollary}[\cite{Burger_Iozzi_Wienhard_toledo}]\label{cor:5.7} The map $\rho\mapsto\T(\rho)$ on 
$\hom\big(\pi_1(S),G\big)$ is continuous.
\be
\item If $S$ is compact, it takes values in $\frac{1}{n_X}\ZZ$ and
is constant on connected components.
\item If $S$ is not compact, its range is 
\bqn
\left[-\rk_X\big|\chi(S)\big|,\rk_X\big|\chi(S)\big|\right]\,.
\eqn
\ee
\end{corollary}
\begin{proof} The continuity follows from the formulas in Theorem~\ref{thm:5.5}.
Then (1) is clear and (2) follows from the fact that, since $\pi_1(S)$ is free,
then $\hom\big(\pi_1(S),G\big)\cong G^{2g+n-1}$
is connected and therefore the intermediate value theorem implies the statement.
\end{proof}

In the study of representations of $\pi_1(S)$, where $S$ is of finite topological type,
we were several times led to consider the corresponding ``completed'' compact surface $\Sigma$
with boundary, for which of course we have $\pi_1(S)=\pi_1(\Sigma)$. 
In the light of the Fenchel--Nielsen approach to Teichm\"uller theory, 
it is natural to ask what happens to Toledo invariants 
when one glues together two surfaces along a component of their boundary.
The answer is given by the following

\begin{proposition}[\cite{Burger_Iozzi_Wienhard_toledo}]\label{prop:5.8}  
Let $\Sigma$ be a compact oriented surface
with boundary and $\rho:\pi_1(\Sigma)\to G$ a homomorphism.
\be
\item If $\Sigma=\Sigma_1\cup_C\Sigma_2$ is the connected sum
of two subsurfaces $\Sigma_1$ and $\Sigma_2$ along a separating loop $C$, then
\bqn
\T(\Sigma,\rho)=\T(\Sigma_1,\rho_1)+\T(\Sigma_2,\rho_2)\,,
\eqn
where $\rho_i$ is the restriction of $\rho$ to $\pi_1(\Sigma_i)$.
\item If $\Sigma'$ is the surface obtained by cutting $\Sigma$ along a
non-separating loop $C$ and $i:\Sigma'\to\Sigma$ is the canonical map, then 
\bqn
\T(\Sigma',\rho\,i_\ast)=\T(\Sigma,\rho)\,.
\eqn
\ee
\end{proposition}

\begin{remark}\label{remark} This result holds also for the invariants $\T_\kappa(\Sigma,\rho)$
for $\kappa\in\hcb^2(G,\ZZ)$ introduced in \S\ref{subsec:coh_fr}.
\end{remark}

In the situation of Proposition~\ref{prop:5.8}, 
if we take into account that the Euler characteristic is additive
under connected sum, we obtain that $\rho\in\hom\big(\pi_1(\Sigma),G\big)$
is maximal if and only if $\rho_i\in\hom\big(\pi_1(\Sigma_i),G\big)$
are maximal for $i=1,2$.

Concerning the geometric properties of maximal representations, the first fundamental
result is:

\begin{theorem}[\cite{Burger_Iozzi_Wienhard_toledo}]\label{thm:5.9} 
Maximal representations are injective and with discrete image.
\end{theorem}
\begin{remark}
As soon as the Lie group $G$ is not locally isomorphic to $\PSU(1,1)$, 
there are injective representations with discrete image that are not maximal. 
\end{remark}

The proof of this result relies on the fact (see Corollary~\ref{cor:5.6})
that if $\rho:\pi_1(S)\to G$ is maximal and $\rk_X$ is the real rank of $G$, then 
\bqn
\rho^\ast(\kgb)=\rk_X\,\ksrb\,.
\eqn
The proof proceeds then by an appropriate reinterpretations of this equality 
in terms of homogeneous quasimorphisms; this approach works for a larger class
of representations, namely the class of {\it causal} representations 
which satisfy that $\rho^\ast(\kgb) = \lambda \, \ksrb$ for some $\lambda \neq 0$;
details will appear in \cite{BenSimon_Burger_Hartnick_Iozzi_Wienhard}.

While Theorem~\ref{thm:5.9} holds for all surfaces of finite type,
one has a substantially stronger result when $S$ is compact, namely:

\begin{theorem}[\cite{Burger_Iozzi_Wienhard_toledo}]\label{thm:5.9.2}
 Let $S$ be a compact surface, $\rho:\pi_1(S)\to G$ a maximal
representation and $X$ the symmetric space associated to $G$.  Then there are
constants $A>0$ and $B\geq0$ such that 
\bqn
A^{-1}\|\gamma\|-B\leq d_X\big(\rho(\gamma)x_0,x_0\big)\leq A\|\gamma\|+B\,,
\eqn
where $x_0\in X$ is a basepoint and $\|\,\cdot\,\|$ is a word metric on $\pi_1(S)$.
\end{theorem}  
This result is a consequence of the fact that maximal representations are {\em Anosov} (see \S~\ref{subsec:anosov});
the proof of this fact uses the structure theorem 
Theorem~\ref{thm:5.10}  presented in the next section. 

\subsection{The structure theorem and tube type domains}\label{subsec:str_thm_ttdom}
Almost from its beginning in the 80's, research on maximal representations was driven by
``irreducibility'' questions.  For instance D.~Toledo, using tools from the
Gromov--Thurston proof of Mostow rigidity for real hyperbolic manifolds, showed in \cite{Toledo_89} that 
a maximal representation from a compact surface group into $\SU(n,1)$ leaves invariant
a complex geodesic, or equivalently its image is contained in a conjugate of
$\operatorname{S}\big(\operatorname{U}(n-1)\times\operatorname{U}(1,1)\big)$.
Then L.~Hern\'andez showed in \cite{Hernandez} that 
if $\SU(n,2)$ (for $n\geq2$) is the target group,
the image must be contained in a conjugate of
$\operatorname{S}\big(\operatorname{U}(n-2)\times\operatorname{U}(2,2)\big)$.
In \cite{Bradlow_GarciaPrada_Gothen} S.~Bradlow, O.~Garc\'\i a-Prada and P.~Gothen
then showed that a reductive maximal representation with target $\SU(p,q)$,
with $p\leq q$, is contained in a conjugate of 
$\operatorname{S}\big(\operatorname{U}(p,p)\times\operatorname{U}(q-p)\big)$
using methods from the theory of Higgs bundles.

In its most general form the problem presents itself naturally in the following way:
given a maximal representations $\rho:\pi_1(S)\to G$ where $G:=\gG(\RR)^\circ$
consists of the real points of the connected component of a semisimple algebraic 
group $\gG$ defined over $\RR$, determine the Zariski closure
$\lL:=\overline{\rho\big(\pi_1(S)\big)}^{\rm Z}$ of the image of $\rho$.
In \cite{Burger_Iozzi_Wienhard_toledo} we gave a complete answer to this question
and most of this section is devoted to the description of the result
and the ingredients of the proof.

Recall that every Hermitian symmetric space $X$ (of noncompact type)
is biholomorphic to a bounded domain $\Dd\subset\CC^n$.  While this is 
the natural generalization of the Poincar\'e disk, the question 
of the generalization of the upper half plane leads to the notion of tube type domain.
We say that $X$ (or $\Dd$) is of {\em tube type}\index{tube type domain} if it is biholomorphic 
to a domain of the form $V+\imath\Omega$, where $V$ is a real vector space and
$\Omega\subset V$ is an open convex proper cone in $V$.
The groups corresponding to irreducible Hermitian symmetric spaces of tube type are
 $\Sp(2n,\RR)$, $\SU(p,p)$, $\SO^\ast(2n)$
(for $n$ even), $\SO(2,n)$ and one of the two exceptional ones. 
There are many known characterizations of tube type domains, mainly in terms of special 
geometric structures, or the topology of their Shilov boundary, and 
we will add a new one in Theorem~\ref{thm:5.20}. 

With the notion of tube type at hand, the structure of the Zariski closure of the image
of  a maximal representation is described by the following 

\begin{theorem}[\cite{Burger_Iozzi_Wienhard_tight}]\label{thm:5.10}  Let $G:=\gG(\RR)^\circ$ 
be a Lie group of Hermitian type
with associate symmetric space $X$.  Let $\rho:\pi_1(S)\to G$ 
be a maximal representation and $\lL:=\overline{\rho\big(\pi_1(S)\big)}^{\rm Z}$
the Zariski closure of its image.  Then:
\be
\item the Lie group $L:=\lL(\RR)^\circ$ is reductive with compact centralizer in $G$;
\item the semisimple part of $L$ is of Hermitian type;
\item the Hermitian symmetric space $\Yy$ associated to $L$ is of tube type and 
  the totally geodesic embedding $\Yy\hookrightarrow X$ is tight.
\ee
\end{theorem}

In statement (3) the embedding $\Yy\hookrightarrow X$ is not necessarily 
holomorphic but it is tight, a notion involving the area of geodesic triangles
in $\Yy$ and $X$ with respect to $\omega_X$. We will elaborate on this notion
in \S~\ref{subsec:tight}.

In order to relate this result to the ``irreducibility question'' described above, 
we recall that in every Hermitian symmetric space $X$, maximal tube type subdomains
exists, they are all conjugate and of rank equal to the rank of $X$.
We have then:

\begin{corollary}[\cite{Burger_Iozzi_Wienhard_toledo}]\label{cor:5.11} 
 Let $\rho:\pi_1(S)\to G$ be a maximal representation.
Then there is a maximal tube type subdomain which is $\rho\big(\pi_1(S)\big)$-invariant.
\end{corollary}

A special case of Theorem~\ref{thm:5.10} is when $\rho$ has Zariski dense image in $\gG$,
in which case $\Yy=X$ and hence $X$ is of tube type.
This result is optimal, in the sense that every tube type domain admits 
a maximal representation with Zariski dense image.  In order to be more specific,
we recall that a diagonal disk in $X$ is a holomorphic totally geodesic embedding
\bqn
d:\DD\to X
\eqn
obtained as the composition of a diagonal embedding $\DD\to\DD^r$ (where $r=\rk_X$)
and a maximal polydisk embedding $\DD^r\to X$.  Let 
\bqn
\Delta:\SU(1,1)\to G
\eqn
denote the homomorphism corresponding to $d$.  Let now $\rho:\pi_1(S)\to\SU(1,1)$
be a hyperbolization.  Then we have the following

\begin{theorem}[\cite{Burger_Iozzi_Wienhard_toledo}]\label{thm:5.12}  Assume that $X$ is of tube type.
Then there exists a path of homomorphisms
\bqn
\rho_t:\pi_1(S)\to G\,,
\eqn
for $t\geq0$, such that 
\be
\item $\rho_t$ is maximal for all $t\geq0$ and $\rho_0=\Delta\circ\rho$;
\item $\rho_t$ has Zariski dense image for $t>0$.
\ee
\end{theorem}

\begin{remark}
  Using the structure theory developed in
  \cite{Burger_Iozzi_Wienhard_tight} which is described in the
  following section, Kim and Pansu \cite{Kim_Pansu} recently showed
  that for fundamental groups of compact surfaces the global rigidity
  result for maximal representation into non-tube type Hermitian Lie
  groups given by Theorem~\ref{thm:5.10} arises only in this context.
  For a precise statement of their result see \cite[Corollary
  2]{Kim_Pansu}.
\end{remark}

In the next section we describe the various ingredients 
entering the proof of the structure theorem (Theorem~\ref{thm:5.10}).

\subsection{Tight homomorphisms, triangles \\ and the Hermitian triple product}\label{subsec:tight}
Maximal representations are a special case of more general type of homomorphism,
namely {\em tight homomorphisms}\index{tight homomorphism}; they are defined on any locally compact group $L$
and take values in a Lie group of Hermitian type $G$.

\begin{definition}\label{defi:5.10} A continuous homomorphism $\rho:L\to G$  
is {\em tight}\index{tight homomorphism}\index{homomorphism!tight} if
\bqn
\big\|\rho^\ast(\kgb)\big\|=\|\kgb\|\,.
\eqn
\end{definition}

By inspecting the proof of the Milnor--Wood type inequality in Corollary~\ref{cor:5.6}, 
one verifies easily that maximal representations are tight.

In the case in which also $L$ is a Lie group of Hermitian type and $\Yy$ is the
associated symmetric space, then a continuous homomorphism
$h:L\to G$ gives rise to a totally geodesic map $f:\Yy\to X$.
The geometric condition on $f$ for $\rho$ to be tight is then
\bq\label{eq:5.4.1}
\sup_{\Delta\subset\Yy}\int_\Delta f^\ast\omega_X=\sup_{\Delta\subset X}\int_\Delta\omega_X
\eq
and we call $f$ tight if it satisfies \eqref{eq:5.4.1}.  
A useful observation is that if $\rho:\pi_1(S)\to L$ is a homomorphism
such that $h\circ\rho$ is maximal, then $h$ is tight.
For the converse, we need to introduce an additional notion.
Namely, recall that the space 
\bqn
\hcb^2(L,\RR)\cong\hc^2(L,\RR)\cong\Omega^2(\Yy)^L
\eqn
is generated as a vector space by the pullback to $\Yy$ of 
the K\"ahler form of the irreducible factors of $\Yy$.
The open cone generated by the linear combination 
with strictly positive coefficients of these forms is called
the {\em cone of positive K\"ahler classes} 
and denoted by $\hcb^2(L,\RR)_{>0}$.

\begin{definition}\label{defi:5.11}  A continuous homomorphism $h:L\to G$ 
is positive\index{homomorphism!positive} if
\bqn
h^\ast(\kgb)\in\hcb^2(L,\RR)_{>0}\,.
\eqn
\end{definition}

With these definitions we have then:

\begin{proposition}[\cite{Burger_Iozzi_Wienhard_tight}]\label{prop:5.13}  
If $\rho:\pi_1(S)\to L$ is maximal and 
$h:L\to G$ is tight and positive, then $h\circ\rho:\pi_1(S)
\to G$ is maximal.
\end{proposition}

This is particularly useful in combination with the following geometric examples.

\begin{proposition}[\cite{Burger_Iozzi_Wienhard_tight}]\label{prop:5.14}  
Let $\Yy$ and $X$ be Hermitian symmetric spaces
with normalized metrics and let $f:\Yy\to X$ be a holomorphic and isometric map.
Then $f$ is tight if and only if $\rk_X=\rk_\Yy$, in which case it is also positive.
In particular:
\be
\item a maximal polydisk $t:\DD^r\to X$ is tight and positive;
\item the inclusion $T\hookrightarrow X$ of a maximal tube type subdomain is tight and positive;
\item a diagonal disk $d:\DD\to X$ is tight and positive.
\ee
\end{proposition}

We stress the fact that all Hermitian spaces involved carry the normalized metric,
that is the one with minimal holomorphic sectional curvature $-1$.  

There are many interesting tight embeddings which are not holomorphic,
as the following result shows.

\begin{proposition}\label{prop:5.15}  The $2n$-dimensional irreducible representation
\bqn
\rho_{2n}:\SL(2,\RR)\to\Sp(2n,\RR)
\eqn
is tight and corresponds to a holomorphic map only when $n=1$.
\end{proposition}

The main structure theorem concerning tight homomorphisms is then the following:
\begin{theorem}[\cite{Burger_Iozzi_Wienhard_tight}]\label{thm:5.16} 
Let $L$ be a locally compact second countable group,
$G=\gG(\RR)^\circ$ a Lie group of Hermitian type and 
let $\rho:L\to G$ be a continuous tight homomorphism.  Then:
\be
\item the Zariski closure $\hH:=\overline{\rho(L)}^{\rm Z}$ is reductive;
\item the centralizer of $H:=\hH(\RR)^\circ$ in $G$ is compact;
\item the semisimple part of $H$ is of Hermitian type and the associated symmetric space
$\Yy$ admits a unique $H$-invariant complex structure such that 
the inclusion $H\hookrightarrow G$ is tight and positive.
\ee
\end{theorem}

Setting $L=\pi_1(S)$ and assuming $\rho$ to be maximal, 
the above result accounts then for most of the statements in the structure Theorem~\ref{thm:5.10}
except the one, essential, that $\Yy$ is of tube type.  
This is specific to the hypothesis that $L=\pi_1(S)$ is a surface group.

An important ingredient of the structure theorem for tight homomorphisms is 
the work of Clerc and \O rsted on the characterization of ``ideal triangles with 
maximal symplectic area'' \cite{Clerc_Orsted_2}.
To describe some important features, we will assume for simplicity
that $G$ is simple (of Hermitian type), that is the associated symmetric space $X$
is irreducible.  
Let $\Dd\subset\CC^n$ be the bounded domain realization of $X$;
there is an explicit realization of $\Dd$, called
{\em Harish-Chandra realization}, and in which the Bergmann kernel
\bqn
K_\Dd:\Dd\times\Dd\to\CC^\ast
\eqn
of $\Dd$ can be computed rather explicitly.
We let $G$ act on $\Dd$ via the isomorphism $X\to\Dd$;
the Hermitian metric defined by $K_\Dd$ is of course $G$-invariant,
and it has the interesting feature that its K\"ahler form comes from 
an integral class in $\hc^2(G,\ZZ)$.  However,
this metric is not normalized.  One introduces then the normalized Bergmann kernel
\bqn
k_\Dd:=K_\Dd^{1/n_\Dd}
\eqn
where $n_\Dd = n_X$ is a specific integer (see Corollary~\ref{cor:5.7} and the discussion preceding it).
This leads to the normalized metric whose K\"ahler form is 
\bqn
\omega_\Dd=\imath\partial\overline\partial\log k_\Dd(z,z)\,.
\eqn
In fact, on the specific formulas for $k_\Dd$ one sees that 
it is defined and not vanishing on $\Dd^2$ of course and even on a certain
open dense subset $\overline\Dd^{(2)}\subset\overline\Dd^2$ 
of pairs of points satisfying a certain transversality condition.
The domain $\overline\Dd^{(2)}$ being star shaped, we let $\arg k_\Dd$ denote 
the unique continuous determination of the argument of $k_\Dd$ on $\overline\Dd^{(2)}$
vanishing on the diagonal of $\Dd\times\Dd$.  Then we have the following formula
for the area of a triangle with geodesic sides:

\begin{lemma}[\cite{Domic_Toledo, Clerc_Orsted_2}]\label{lem:5.17} For $x,y,z\in\Dd$
\bqn
\int_{\Delta(x,y,z)}\omega_\Dd=-\big(\arg k_\Dd(x,y)+\arg k_\Dd(y,z)+\arg k_\Dd(z,x)\big)\,.
\eqn
\end{lemma}

Guided by this, we introduce a fundamental object on the set $\overline\Dd^{(3)}$
of triples of pairwise transverse points, namely the {\em Bergmann cocycle}\index{Bergmann cocycle}
\bqn
\beta_\Dd(x,y,z)=-\frac{1}{2\pi}\big(\arg k_\Dd(x,y)+\arg k_\Dd(y,z)+\arg k_\Dd(z,x)\big)\,.
\eqn
Its role is to extend in a meaningful way the notion of area to ``ideal triangles''.
This function is continuous on $\overline\Dd^{(3)}$, $G$-invariant and 
satisfies an obvious cocycle property.
The following is then a summary of some work of Clerc and \O rsted.

\begin{theorem}[\cite{Clerc_Orsted_2}]\label{thm:5.18} 
Let $\cs$ and $\rk_\Dd$ be respectively the Shilov boundary\index{Shilov boundary} 
and the real rank of $\Dd$.
\be
\item Then
\bqn
-\frac{1}{2}\rk_\Dd\leq\beta_\Dd(x,y,z)\leq\frac{1}{2}\rk_\Dd\,,
\eqn
with strict inequality if $(x,y,z)\in\Dd^3$;
\item we have that $\beta_\Dd(x,y,z)=\frac{\rk_\Dd}{2}$ if and only if 
$x,y,z\in\cs$ and there exists a diagonal disk 
\bqn
d:\overline\DD\to\overline\Dd
\eqn
with $d(1)=x$, $d(\imath)=y$ and $d(-1)=z$.
\ee
\end{theorem}

In the sequel we call the restriction of the Bergmann cocycle to the Shilov boundary 
{\em Maslov cocycle}\index{Maslov cocycle} and we call 
a triple of points $(x,y,z)$ on the Shilov boundary
{\em maximal} if
\bqn
\beta_\Dd(x,y,z)=\frac{\rk_\Dd}{2}\,.
\eqn
Observe that a maximal triple is always contained in the boundary of 
a maximal tube type subdomain of $\Dd$.

One of the corollaries of the above result is the computation of the Gromov norm of $\kgb$ (Theorem~\ref{thm:5.2}).
This is based on the following
\begin{corollary}[\cite{Burger_Iozzi_supq, Burger_Iozzi_Wienhard_toledo}]\label{cor:5.19}  
Under the canonical map
\bqn
\h^\bullet\big(\linfty(\cs^\bullet)^G\big)\to\hcb^\bullet(G,\RR)
\eqn
the class defined by $\beta_\Dd$ goes to $\kgb$.
\end{corollary}

Finally we turn to the ingredient which leads to the conclusion in Theorem~\ref{thm:5.10}
that $\Yy$ is of tube type.  For this we construct an invariant of triples of points
on $\cs$ which we call the Hermitian triple product\index{Hermitian triple product} 
and whose definition goes as follows;
recall that the Bergmann kernel satisfies the relation
\bqn
K_\Dd(gz,gw)=j(g,z) K_\Dd(z,w)\overline{j(g,w)}\,,
\eqn
where $j(g,z)$ is the complex Jacobian of $g$ at the point $z\in\Dd$.
Then we define on the set $\cs^{(3)}$ of pairwise transverse points the Hermitian
triple product
\bqn
\<\<x,y,z\>\>:=K_\Dd(x,y) K_\Dd(y,z) K_\Dd(z,x)\mod\RR^\times\,.
\eqn
Recall that $\cs$ is of the form $G/Q$, where $Q$ is a maximal parabolic subgroup of $G$,
 and is hence in a natural way the set of real points
of a complex projective variety.

\begin{theorem}[\cite{Burger_Iozzi_Wienhard_kahler}]\label{thm:5.20}  The function 
\bqn
\<\<\,\cdot\,,\,\cdot\,\>\>:\cs^{(3)}\to\RR^\times\backslash\CC^\times
\eqn
is a $G$-invariant multiplicative cocycle
and, for an appropriate real structure on $\RR^\times\backslash\CC^\times$,
it is a real rational function.
Moreover the following are equivalent:
\be
\item $\Dd$ is not of tube type;
\item $\cs^{(3)}$ is connected;
\item the Hermitian triple product is not constant.
\ee
\end{theorem}

\begin{proof}[{\it Sketch of the proof that $\Yy$ is of tube type}]
Let $\rho:\pi_1(S)\to G:=\gG(\RR)^\circ$ be a maximal representation.
Using the structure theorem for tight homomorphisms we may assume that
$\rho\big(\pi_1(S)\big)$ is Zariski dense in $\gG$ and hence $\Yy=X$.
We have to show that $X$ is of tube type.
Realize $\pi_1(S)$ as a lattice in $\PSU(1,1)$ via an appropriate hyperbolization.
Since $\rho$ has Zariski dense image, the action on the Shilov boundary $\cs$ is strongly proximal.
 This together with the amenability of the action of $\pi_1(S)$ on $\partial \DD$
 via the chosen hyperbolization implies (according to \cite{Margulis_book}) 
the existence of an equivariant measurable map (see e.g. \cite{Burger_Iozzi_supq}
 for a description of the construction)
\bqn
\varphi:\partial\DD\to\cs
\eqn
into the Shilov boundary.  From the maximality assumption 
we deduce that 
\bqn
\rho^\ast(\kgb)=\rk_\Dd\ksrb
\eqn
and hence that 
\bqn
 \beta_\Dd\big(\varphi(x),\varphi(y),\varphi(z)\big)
=\rk_\Dd\beta_\DD(x,y,z)
\eqn
for almost every $(x,y,z)$.
Thus we obtained that for almost every $x,y,z\in\partial\DD$
\bqn
\ba
K_\Dd\big(\varphi(x),\varphi(y)\big) 
K_\Dd\big(\varphi(y),\varphi(z)\big)
K_\Dd\big(\varphi(z),\varphi(x)\big)&\\
=e^{2\pi\imath n_\Dd\beta_\Dd\big(\varphi(x),\varphi(y),\varphi(z)\big)}&\mod\RR^{\times+}\\
=e^{2\pm\pi\imath n_\Dd\rk_\Dd\frac12}&\mod\RR^{\times \,+}
\ea
\eqn
and as a result the square $\<\<\,\cdot\,,\,\cdot\,\>\>^2$ of the Hermitian triple 
product is equal to $1$ on $(\essim\varphi)^{(3)}\subset\cs^{(3)}$,
where $\essim\varphi\subset\cs$ is the essential image of $\varphi$.
But, being invariant under $\rho$, this set is Zariski dense in $\cs$ hence 
the rational function $\<\<\,\cdot\,,\,\cdot\,\>\>^2$ on $\cs^{(3)}$
is identically equal to $1$. If now $\Dd$ were not of tube type,
$\cs^{(3)}$ would be connected and hence $\<\<\,\cdot\,,\,\cdot\,\>\>$
would be identically equal to $1$ on $\cs^{(3)}$, which is a contradiction.
\end{proof}

\subsection{Boundary maps, rotation numbers \\ and representation varieties}\label{subsec:bdry_rot_repvar}
We have seen that if $\rho:\pi_1(S)\to G$ is a maximal representation 
into a group of Hermitian type and
$h:\pi_1(S)\to\PSU(1,1)$ is a hyperbolization of $S$ of finite area, 
then there exists a measurable $\rho\circ h^{-1}$-equivariant map $\varphi:\partial\DD\to\cs$
and furthermore
\bq\label{eq:5.5.1}
\beta_\Dd\big(\varphi(x),\varphi(y),\varphi(z)\big)=\r_\Dd\beta_\DD(x,y,z)
\eq
for almost every $(x,y,z)\in(\partial\DD)^3$.  Here $\beta_\Dd$ is the Maslov cocycle 
on the Shilov boundary of the bounded symmetric domain $\Dd$; 
observe that $\beta_\DD$ is just $\frac12$ of the orientation cocycle.

In the case where $G=\PSU(1,1)$ we have see that 
$h$ and $\rho$ are semiconjugate by using Ghys' theorem;
an alternative approach would be to use the equality \eqref{eq:5.5.1} 
to show that $\varphi$ coincides
almost everywhere with a weakly monotone map; 
this has been carried out in \cite{Iozzi_ern}.  
In fact, this way of ``improving" the regularity of $\varphi$
works in general and the basic idea is presented in \cite{Burger_Iozzi_Labourie_Wienhard}.  
One considers the essential graph of $\varphi$
\bqn
\essgr(\varphi)\subset\partial\DD\times\cs
\eqn
which is by definition the support of the direct image of the Lebesgue measure on $\partial\DD$
under the map
\bqn
\ba
\partial\DD&\to\partial\DD\times\cs\\
x&\mapsto\big(x,\varphi(x)\big)\,.
\ea
\eqn
Then one shows that there are exactly two sections $\varphi_-$ and $\varphi_+$ of the projection
of $\essgr(\varphi)$ on $\partial\DD$ such that:
\be
\item $\varphi_-$ and $\varphi_+$ are strictly equivariant;
\item $\varphi_-$ is right continuous while $\varphi_+$ is left continuous;
\item $\essgr(\varphi)=\big\{\big(x,\varphi_-(x)\big),\big(x,\varphi_+(x)\big):\,x\in\partial\DD\big\}$;
\item for every positive triple $x,y,z\in\partial\DD$, both triples 
        $\varphi_+(x),\varphi_+(y),\varphi_+(z)$ and $\varphi_-(x),\varphi_-(y),\varphi_-(z)$ 
        are maximal.
\ee
This generalizes exactly the $\PSU(1,1)$ picture and, remarkably, 
the discontinuities of $\varphi_-$ and $\varphi_+$ are simple.  

One can summarize the situation as follows:

\begin{theorem}[\cite{Burger_Iozzi_Wienhard_toledo}]\label{thm:5.21}
 The representation $\rho:\pi_1(S)\to G$ is maximal 
if and only if there exists a left continuous map $\varphi:\partial\DD\to\cs$ such that 
\be
\item $\varphi$ is strictly $\rho\circ h^{-1}$-equivariant;
\item $\varphi$ maps every positively oriented triple in $\partial\DD$ to a maximal triple on $\cs$.
\ee
\end{theorem}

The first obvious consequence is the following result on the existence of fixed points:

\begin{corollary}[\cite{Burger_Iozzi_Wienhard_toledo}]\label{cor:5.22}  Let $\rho:\pi_1(S)\to G$ be maximal.
Then:
\be
\item if $\gamma$ is freely homotopic to a boundary component, 
  $\rho(\gamma)$ has a fixed point in $\cs$;
\item if $\gamma$ is not conjugate to a boundary component, then $\rho(\gamma)$ has (at least)
        two fixed points in $\cs$, which are transverse. 
\ee
\end{corollary}

We use again the standard presentation of $\pi_1(S)$ (see (\ref{eq:fg})) and 
define the following subset of $\homcs\big(\pi_1(S),G\big)$:
\bqn
\ba
 \homcs\big(\pi_1(S),G\big)
=\big\{
\rho\in\hom\big(\pi_1(S),G\big);\, \hbox{for every }1\leq i\leq n,& \\
        \rho(c_i)\hbox{ has at least one fixed point in }\cs\big\}&
\ea
\eqn
Then $\homcs\big(\pi_1(S),G\big)$ is a semialgebraic subset of $\hom\big(\pi_1(S),G\big)$
and we have from Corollary~\ref{cor:5.22} that
\bqn
\hommax\big(\pi_1(S),G\big)\subset\hom_\cs\big(\pi_1(S),G\big)\,.
\eqn

\begin{theorem}[\cite{Burger_Iozzi_Wienhard_toledo}]\label{thm:5.23} Assume that $\Dd$ is of tube type.  
Then the Toledo invariant $\rho\mapsto \T(\Sigma,\rho)$ is locally constant on
$\homcs\big(\pi_1(S),G\big)$.  In particular, the subset of maximal representations
$\hommax\big(\pi_1(S),G\big)$ is a union of connected components of
$\homcs\big(\pi_1(S),G\big)$ and therefore semialgebraic.
\end{theorem}

This result is essentially a consequence of the formulas in \S~\ref{eq:fg} 
for the invariants $\T_\kappa(\Sigma,\rho)$, $\kappa\in\hcb^2(G,\ZZ)$
together with the lemma that if $Q$ is the stabilizer in $G$ of a point in $\cs$,
and if $\Dd$ is of tube type, then the restriction map
\bqn
\hcb^2(G,\RR)\to\hcb^2(Q,\RR)
\eqn
is identically zero.  

We end by mentioning a result which gives additional invariants for maximal representations.
Recall that for $\chi\in\homc(K,\RR/\ZZ)$ we have introduced a class function 
$\chiext:G\to\RR/\ZZ$ extending $\chi$.  We have:

\begin{theorem}[{\cite[Theorem~13]{Burger_Iozzi_Wienhard_toledo}}]\label{thm:5.24} 
Let $\chi\in\homc(K,\RR/\ZZ)$ and fix $\rho_0:\pi_1(S)\to G$ maximal.  
\be
\item For every maximal $\rho:\pi_1(S)\to G$, the map
\bqn
\ba
R_\chi(S):\pi_1(S)&\longrightarrow\qquad\quad\RR/\ZZ\\
\gamma\quad&\mapsto\chiext\big(\rho(\gamma)\big)-\chiext\big(\rho_0(\gamma)\big)
\ea
\eqn
is a homomorphism.
\item If $\Dd$ is of tube type, $R_\chi(S)$ takes values in $e_G^{-1}\ZZ/\ZZ$ and
\bqn
\hommax\big(\pi_1(S),G\big)\to\hom\big(\pi_1(S),e_G^{-1}\ZZ/\ZZ\big)
\eqn
is constant on connected components.
(Here $e_G$ is an explicit constant depending on $G$, not just on the symmetric space associated to $G$, 
e.\,g $e_{\SL(2,\RR)} = 2$.)
\ee
\end{theorem}

\section{Hitchin representations and positive representations}\label{subsec:hitchin_positive}
Hitchin representations and positive representations are defined when $G$ is a split real Lie group. 

\begin{definition}\label{defi:split}
A real simple Lie groups $G$ is {\em split}\index{split real Lie group}\index{Lie group!split} 
if its real rank equals the complex rank of its complexification $G_\CC$, 
i.\,e. the maximal torus is diagonalizable over $\RR$.
\end{definition}

\subsection{Hitchin representations}\index{Hitchin representation}\index{representation!Hitchin}
Let $S$ be a compact surface and $G$ a split real simple adjoint group,
e.g. $G= \PSL(n,\RR)$, $\PSp(2n,\RR)$, $\PO(n,n+1)$ or $\PO(n,n)$,
Hitchin \cite{Hitchin} singled out a connected component 
\bqn
\RH \big(\pi_1(S), G \big) \subset \R \big(\pi_1(S), G \big)\, 
\eqn 
which he called {\em
  Teichm\"uller component}, now it is usually called {\em Hitchin
  component}\index{Hitchin  component}.

In order to define the Hitchin component we recall that the Lie
algebra $\frakg$ of a split real simple adjoint Lie group $G$ contains
a (up to conjugation) unique principal three dimensional simple Lie
algebra.  That is an embedded subalgebra isomorphic to 
$\fraks\frakl(2,\RR)$ which
is the real form of a subalgebra $\fraks\frakl(2,\CC) \subset
\frakg_\CC$ given (via the theorem of Jacobson--Morozov)
by a regular nilpotent element in $\frakg_\CC$. Here $\frakg_\CC$
denotes the complexification of $\frakg$ and a nilpotent element is
regular if its centralizer is of dimension equal to the rank of
$\frakg_\CC$. For more details on principal three dimensional
subalgebra we refer the reader to \cite{Onishchik_Vinberg_III} or
Kostant's original papers \cite{Kostant1,Kostant2,Kostant3}.

The embedding $\fraks\frakl(2,\RR) \to \frakg$ gives rise to an
embedding $\pi: \SL(2,\RR) \to G$. Precomposition of $\pi$ with a
discrete (orientation preserving) 
embedding of $\pi_1(S)$ into $\SL(2,\RR)$ defines a
homomorphism $\rho_0: \pi_1(S) \to G$, which we call a principal
Fuchsian representation.  The  Hitchin component $\RH \big(\pi_1(S), G \big)$ is
defined as the connected component of $\R \big(\pi_1(S),G \big)$ containing
$\rho_0$.  By construction it contains a copy of Teichm\"uller space.

\begin{remark}\label{rem:irreducible}
  When $G = \PSL(n,\RR)$, $\PSp(2m,\RR)$ or $\PO(m,m+1)$ the embedding
  $\pi$ is given by the $n$-dimensional irreducible representation
  $\PSL(2,\RR) \to \PSL(n,\RR)$, which is contained in $\PSp(2m,\RR)$
  if $n= 2m$ and in $\PO(m,m+1)$ if $n = 2m+1$.
  For $G = \PO(m,m)$, the embedding $\pi$ is given by the composition of 
the $2m-1$-dimensional irreducible representation into $\PO(m,m-1)$ 
with the embedding $\PO(m,m-1)$ into $\PO(m,m)$.
\end{remark}

\begin{remark}\label{rem:not_adjoint}
When $G$ is a finite cover of a split real simple adjoint Lie group $G^{\rm Ad}$,
 one can define the Hitchin components 
$\RH\big(\pi_1(S), G\big)$ by taking lifts of the principal Fuchsian representations 
$\pi_1(S) \to G^{\rm Ad}$ and the corresponding connected components.
 Equivalently one can define a Hitchin representation $\rho: \pi_1(S) \to G$ 
as a representation whose projection $\rho: \pi_1(S) \to G^{\rm Ad}$ is a Hitchin representation. 
%
\end{remark}


Hitchin studied the Hitchin component following an analytic approach,
relying on the correspondence between irreducible representations of
$\pi_1(S)$ in $\PSL(n,\CC)$ and stable Higgs bundles. One direction of
this correspondence is due to Corlette \cite{Corlette_1988} and
Donaldson \cite{Donaldson_1985}, following ideas of Hitchin
\cite{Hitchin_1987}, the other direction is due to Simpson
\cite{Simpson_1988, Simpson_1992}.

The correspondence requires to fix a complex structure $j$ on $S$;
with this choice, there is an isomorphism (see \cite{Hitchin})
%
%
\bq\label{eq:hitchin_map}
h_{j}: \RH \big(\pi_1(S), G \big)/G \to 
H^0 \big(S,\oplus_{k=1}^r \Omega_{(S,j)}^{d_k} \big)\,,
\eq 
where $H^0 \big(S, \Omega_{(S,j)}^d \big)$ is
the vector space of holomorphic differentials (with respect to the
fixed complex structure $j$ on $S$) of degree $d$.  The coefficients
$d_k$, $ k = 1, \cdots, r = \rk(G)$ are the degrees of a basis of the
algebra of invariant polynomials on $\frakg$.
In particular, this proves 
\begin{theorem}[{\cite[Theorem A]{Hitchin}}]\label{thm:hitchin}
  Let $S$ be a compact surface and $G$ the adjoint group of a split
  real Lie group. Then the Hitchin component $\RH \big(\pi_1(S), G \big)/G$ is
  homeomorphic to $\RR^{|\chi(S)|\dim G}$.
\end{theorem}

Hitchin pointed out that the analytic approach via Higgs bundles gives
no indication about the geometric significance of the representations belonging to this component.
The only example supporting the idea that Hitchin components might
 parametrize geometric structures on $S$ available at that time was 
 given in work of Goldman \cite{Goldman_convex} and Choi and Goldman
 \cite{Goldman_Choi}, who showed that for $G=\PSL(3,\RR)$ the Hitchin
 component parametrizes convex real projective structures on $S$. Now
 we have a better, but not yet satisfactory understanding of the
 geometric significance of the Hitchin components beyond
 $\PSL(3,\RR)$, which will be described in more detail in
 \S~\ref{subsec:geom_struc} below.

%
 
 A direct consequence of Choi and Goldman's result is that any Hitchin
 representation $\rho: \pi_1(S) \to \PSL(3,\RR)$ is a discrete
 embedding with the additional property that for any $\gamma \in
 \pi_1(S) -\{1\}$, the element $\rho(\gamma)$ is diagonalizable with
 distinct real eigenvalues.  These properties have been generalized to
 all Hitchin representations into $\PSL(n,\RR)$ by Labourie
 \cite{Labourie_anosov}.

\begin{theorem}[\cite{Labourie_anosov}]\label{thm:Labourie}
  Let $\rho: \pi_1(S) \to \PSL(n,\RR)$ be a Hitchin representation.
  Then $\rho$ is a discrete embedding, and for any $\gamma \in
  \pi_1(S) -\{1\}$, the element $\rho(\gamma)$ is diagonalizable with
  distinct real eigenvalues.
\end{theorem}

Note that by Remark~\ref{rem:irreducible} similar results hold for
Hitchin representations into $\Sp(2m,\RR)$ and $\SO(m,m+1)$. \medskip

\subsection{Positive representations} 
In the case when $S$ is a noncompact surface (of finite type), the
generalization of Hitchin's work, which is based on methods from the
theory of Higgs bundles, has only been partially carried out
as it presents some additional analytic difficulties
(see for example \cite{Biswas_etal}). But when $G$ is an adjoint split real
Lie group and $S$ is a noncompact surface, there is a completely different approach to
define a special subset of $\R \big(\pi_1(S), G \big)$ due to Fock
and Goncharov \cite{Fock_Goncharov} which leads to the set of {\em
  positive representations}\index{positive representation}\index{representation!positive} 
$\RP \big(\pi_1(S),G \big) \subset \R \big(\pi_1(S), G \big)$.

In order to describe the definition let us identify $S$ with a
punctured surfaces of the same topological type.  As recalled in \S~\ref{subsec:flat_gb} 
a homomorphism $\rho: \pi_1(S) \to G$ corresponds to a
flat principal $G$-bundle $G(\rho)$ on S.
The definition of the space of positive representations 
relies on considering the space of framed $G$-bundles on $S$. 
Let $\Bb = G/B$ be the space of Borel subgroups of $G$. 

A framed $G$-bundle on $S$ is a pair $\big(G(\rho), \beta \big)$, where
$G(\rho)$ is a flat principal $G$-bundle on $S$ and $\beta$ is a flat
section of the associated bundle $G(\rho) \times _G \Bb$ restricted to
the punctures.  There is a natural forgetful map from the space of
framed $G$-bundles to the space of flat principal $G$-bundles sending
$\big(G(\rho), \beta \big) \to G(\rho)$. Since there always exists a flat
section of $G(\rho) \times _G \Bb$ over the punctures, this map is
surjective.

Given an ideal triangulation\index{ideal triangulation}\index{triangulation!ideal} 
of the surface $S$, i.e. a triangulation
whose vertices lie at the punctures of $S$, one can use the
information provided by the section $\beta$ to define a coordinate
system on the space of framed $G$-bundles. Fock and Goncharov
show that these coordinate systems form a positive atlas. This means
in particular that the coordinate transformations are given by rational functions, 
involving only positive coefficients. Hence the set of positive framed
$G$-bundles, i.e. the set where for a given triangulation all coordinate functions are positive
real numbers, is well defined and independent of the chosen
triangulation.  The space of positive representations $\RP \big(\pi_1(S),G \big)$ 
is the image of the space of positive framed $G$-bundles under the
forgetful map.

The construction of the coordinates involves the notion of positivity\index{positivity}
in Lie groups introduced by Lusztig
\cite{Lusztig_groups,Lusztig_base}, to which we will come back in
\S~\ref{subsec:boundarymaps}.  When $G = \PSL(n,\RR)$ one can give
an elementary description of the coordinates in terms of projective
invariants of triples and quadruples of flags (see \cite[\S~9]{Fock_Goncharov}).  
In the case when $G = \PSL(2,\RR)$ the
coordinates correspond to shearing coordinates constructed first by Thurston
and Penner \cite{Penner} and similar coordinates constructed by Fock \cite{Fock}.

\begin{theorem}[{\cite[Theorem 1.13, Theorem 1.9 and Theorem 1.10]{Fock_Goncharov}}]\label{thm:Fock_Goncharov_1}
  The space $\RP \big(\pi_1(S), G \big)/G$ of positive representations and
  $\RR^{|\chi(S)| \dim G}$ are homeomorphic.  Every positive
  representation is a discrete embedding, and every nontrivial
  element $\gamma \in \pi_1(S)$ which is not homotopic to a loop
  around a boundary component of $S$, is sent to a positive hyperbolic
  element.
\end{theorem}

The notion of positive representations can be extended to the
situation when $S$ is compact, using the characterization of positive
representations in terms of equivariant positive boundary maps, which
is described in the next section.

\begin{theorem}[{\cite[Theorem 1.15]{Fock_Goncharov}}]\label{thm:Fock_Goncharov_3}
When $S$ is compact,   
$\RP \big(\pi_1(S), G \big) = \RH \big(\pi_1(S),G \big)$.
\end{theorem}
\begin{remark}
  In the case when $G= \PSL(n,\RR), \PSp(2n,\RR)$ or $\PO(n,n+1)$ this
  theorem also follows from Labourie's work \cite{Labourie_anosov}.
\end{remark}
When $S$ is a noncompact surface, the set of positive framed $G$-bundles carries many more interesting structures.  
It is a cluster variety,
admits a mapping class group invariant Poisson structure and natural quantizations. We will
not discuss any of these interesting structures and refer the reader
to \cite{Fock_Goncharov} and \cite{Fock_Goncharov_htt} for further
reading.

\section{Higher Teichm\"uller spaces -- a comparison}\label{sec:higher}

In this section 
we discuss common structures as well as differences between
the higher Teichm\"uller spaces introduced above -- comparing maximal representations 
on one hand and Hitchin representations on the other hand. 
We explain that, when $S$ is compact, representations in higher
Teichm\"uller spaces fit into the context of Anosov structures, 
which is a more general concept 
(for the definition see \S~\ref{subsec:anosov}). 
From this further geometric information about higher
Teichm\"uller spaces can be obtained.

\subsection{Boundary maps}\index{boundary map}\label{subsec:boundarymaps}
The higher Teichm\"uller spaces we are discussing here were defined
and studied with very different methods and so far we see no unified
approach to it.  Nevertheless there is a common theme in all these
works which highlights an important underlying structure for all
higher Teichm\"uller spaces: The existence of very special {\em
  boundary maps}.

Since $\pi_1(S)$ is a word hyperbolic group, the boundary
$\partial\pi_1(S)$ of $\pi_1(S)$ is a well defined compact metrizable space. 
When $S$ is compact $\partial \pi_1(S)$ identifies
naturally with a topological circle $S^1$ endowed with a canonical H\"older
structure. When $S$ is not compact there is no natural identification
of $\partial \pi_1(S)$ with a subset of $S^1$ (see the discussion in
\S~\ref{sec:surf_fin_eul}). 
%

All higher Teichm\"uller spaces can be characterized as the set of representations for which 
there exist special equivariant (semi)continuous maps into a flag variety.  The
special boundary maps all satisfy some positivity condition, where the
notion of positivity depends on the context.

For maximal representations, i.e in the case when $G$ is a Lie group
of Hermitian type, we saw in \S~\ref{sec:max_rep} that the Maslov
cocycle on the Shilov boundary $\cs$ of the symmetric space associated
to $G$ gives rise to the notion of a maximal triple of points in
$\cs$; and Theorem~\ref{thm:5.21} characterizes maximal representations as those 
which admit an equivariant boundary map $\phi:\partial \pi_1(S) \to \cs$, 
which sends every positively oriented triple to a maximal triple in $\cs$.
%
%

For Hitchin representations Labourie  constructed special boundary maps in \cite{Labourie_anosov}. 
In this case $S$ is compact and $\partial\pi_1(S) = S^1$.   
A map $\phi: S^1\to \RR\PP^{n-1}$ is said to be convex if for every $n$-tuple 
of distinct points $x_1, \cdots, x_n \in S^1$ the images $\phi(x_1) , \cdots, \phi(x_n) $ 
are in direct sum - or, equivalently, the map is injective and any hyperplane 
in $ \RR\PP^{n-1}$ intersects $\phi(S^1)$ in at most $(n-1)$ points. 

The characterization of Hitchin representation into $\PSL(n,\RR)$ in
terms of convex maps is due to a combination of the construction by
Labourie and a result by Guichard:
\begin{theorem}[{\cite{Labourie_anosov, Guichard_convex}}]\label{thm:convexmap}
  A representation $\rho: \pi_1(S) \to \PSL(n,\RR)$ lies in the Hitchin
  component if and only if there exists a $\rho$-equivariant continuous convex map 
$\phi: \partial\pi_1(S) \to
  \RR\PP^{n-1}$.
\end{theorem}
 
In the context of positive representations the notion of positivity for the
boundary maps relies on Lusztig's notion of positivity. Recall that a
matrix in $\GL(n,\RR)$ is totally positive if all its minors are
positive numbers. An upper triangular matrix is positive if all
not obviously zero minors are positive. The notion of positivity has
been extended to all split real semisimple Lie groups $G$ by Lusztig
\cite{Lusztig_groups,Lusztig_base}. This can be used to define a
notion of positivity for $k$-tuples in full flag varieties. Let $B^+$
be a Borel subgroup of $G$, $B^-$ an opposite Borel subgroup and $U$
the unipotent radical of $B^+$. Then the set of Borel groups in $\Bb$
being opposite to $B^+$ can be identified with the orbit of $B^-$
under $U$.  The notion of positivity gives us a well defined subset
$U(\RR_{>0}) \subset U$.  A $k$-tuple of points $ (B_1, \cdots, B_k)$
in $\Bb$ is said to be positive if (up to the action of $G$) it can be
written as $\big(B^+, B^-, u_1 B^-, \cdots (u_1 \cdots u_{k-2}) B^-\big)$,
where $u_i \in U (\RR_{>0})$ for all $i= 1, \cdots , k-2$.

\begin{definition}\label{defi:positivemap}
  A map $\partial \pi_1(S) \to G/B$ is said to be positive if it sends
  every positively oriented $k$-tuple in $\partial\pi_1(S)$ to a positive
  $k$-tuple of flags in $G/B$.
\end{definition}

\begin{remark}
A map $\partial \pi_1(S) \to G/B$ is positive if and only if it sends 
every positively oriented triple in $\partial\pi_1(S)$ to a positive tripe of flags in $G/B$.
\end{remark}

\begin{theorem}[{\cite[Theorem 1.6]{Fock_Goncharov}}]\label{thm:positivemap}
  Let $G$ be a split real simple Lie group and $B<G$ a Borel subgroup.
  A representation $\rho: \pi_1(S) \to G$ is positive if and only if
  there exists a $\rho$-equivariant positive map\index{positive map}\index{map!positive} $\phi:
  \partial\pi_1(S) \to G/B$.
\end{theorem}
\begin{remark}
Note that Fock and Goncharov choose a different identification of the boundary of 
$\pi_1(S)$ with a subset of $S^1$. For their identification the boundary map is indeed continuous. 
\end{remark}
Let us emphasize that Fock and Goncharov prove Theorem~\ref{thm:positivemap} when $S$ is a
noncompact surface. In the case when $S$ is a compact surface, they
use the characterization of Theorem~\ref{thm:positivemap} in order to
define positive representations $\rho: \pi_1(S) \to G$ by requiring
the existence of a $\rho$-equivariant positive map $\phi:
\partial\pi_1(S) = S^1 \to G/B$.  In order to prove the equality
$\RH \big(\pi_1(S), G \big) = \RP \big(\pi_1(S),G \big)$ for compact surfaces $S$, they
observe first that the set of positive representation is an open
subset of the Hitchin component. Then they study limits of positive
representations in order to prove that it is also closed. Hence as a nonempty open
 and closed subset it is a connected component and thus coincides with the Hitchin component. 

For $G = \PSL(n,\RR)$ there is an intimate relation between positive
maps of $\partial \pi_1(S)$ into the full flag variety $G/B$ and
convex maps\index{convex map}\index{map!convex} into the the partial flag variety $\RR\PP^{n-1}$. In
particular, the projection of a positive map into the flag variety 
to $\RR\PP^{n-1}$ is a
convex map, and convex maps (with some regularity) naturally lift to positive maps, 
see \cite[Theorem~1.3]{Fock_Goncharov} and \cite[Chapter~5]{Labourie_anosov}, \cite[Appendix B]{Labourie_McShane} 
for details.

\subsection{The symplectic group\index{symplectic group}}\label{subsubsec:symplectic}
The only simple groups which are both of Hermitian type as well as split real are
the real symplectic groups,  $G = \PSp(2n,\RR)$. 
When $n\geq 2$, Hitchin representations or positive representation,
 and maximal representation provide different generalizations of Teichm\"uller space in this situation. 
It is indeed not difficult to see that the Hitchin
component and the space of positive representations are properly
contained in the space of maximal representations.
%

Moreover, for the symplectic group the properties of the boundary maps required 
in Theorem~\ref{thm:positivemap} and in
Theorem~\ref{thm:5.21} are related in the following way.  In
this situation $\mathcal{F} = G/B$ is the flag variety consisting of
full isotropic flags and $\cs =G/Q$ is the partial flag variety
consisting only of Lagrangian (i.e. maximal isotropic) subspaces.
Positive triples of flags in $\mathcal{F}$ in the sense of
Definition~\ref{defi:positivemap} are mapped to maximal triples in the
Shilov boundary in the sense of Theorem~\ref{thm:5.21} under the
natural projection
 \bqn
 \ba \mathcal{F}\qquad &\to \cs\\
  \, (F_1, \cdots, F_n)&\mapsto F_n\,.  
\ea
\eqn

%

\section{Anosov structures}\label{subsec:anosov}
The notion of Anosov structure is a dynamical analogue of the concept
of locally homogeneous $(G,X)$-structures\index{$(G,X)$-structure} in the sense of Ehresmann,
introduced by Labourie in \cite{Labourie_anosov} to study
Hitchin representations into $\PSL(n,\RR)$. Holonomy representations of
Anosov structure are called {\em Anosov representations}.


The class of Anosov representations is much bigger than the higher Teichm\"uller spaces discussed above. 
Anosov representations of
fundamental groups of surfaces exist into any semisimple Lie group,
and they can be defined more generally also for fundamental groups of
arbitrary closed negatively curved manifolds.  Nevertheless, when $S$
is a compact surface, 
representations in higher Teichm\"uller spaces are examples
of Anosov representations and recent results about Anosov
representations provide important geometric information about higher
Teichm\"uller spaces.

\subsection{Definition, properties and examples}
From now on let $S$ be a compact connected oriented surface with a fixed hyperbolic metric.
Denote by $T^1S$ its unit tangent bundle and by $\phi_t$ the geodesic
flow on $T^1S$. The group $\pi_1(S)$ acts as group of deck
transformations on $T^1\widetilde{S}$, commuting with $\phi_t$.

Let $G$ be a semisimple Lie group, given a representation $\rho:
\pi_1(S) \to G$ we obtain a proper action of $\pi_1(S)$ on
$T^1\widetilde{S} \times G$ by \bqn \gamma_\ast(x,g)=\big(T_\gamma
x,\rho(\gamma)g\big) \eqn whose quotient $\pi_1(S)\backslash
(T^1\widetilde{S}\times G)$ is the total space $G(\rho)$ of
a (flat) principal $G$-bundle over $T^1S$. 
(Note that the bundle $G(\rho)$ defined here is the pullback of the bundle $G(\rho)$ over $S$, 
defined in \S~\ref{subsec:flat_gb},  under  the canonical projection $T^1 S \to S$.)
The geodesic flow lifts to a flow on $G(\rho)$ defined (with a slight abuse of notation) by 
$\phi_t(x,g) = (\phi_t(x), g)$ on $T^1\widetilde{S} \times G.$

Let $P_+, P_-<G$ be a pair of opposite parabolic subgroup of $G$. The
unique open $G$-orbit $\mathcal{O} \subset G/P_+ \times G/P_-$
inherits two foliations, whose corresponding distributions we denote
by $E^\pm$, i.e. $(E^\pm)_{(z_+, z_-)} \cong T_{z_\pm} G/P_\pm$.

\begin{definition}\label{def:anosov_def}\cite{Labourie_anosov}
Let $\mathcal{O}(\rho)$ be the associated $\mathcal{O}$-bundle of $G(\rho)$. 
An Anosov structure\index{Anosov structure} on $\mathcal{O}(\rho)$ is a continuous section $\sigma$ such that 
\begin{enumerate}
\item\label{one} $\sigma$ commutes with the flow, and 
\item\label{two} the action of the flow $\phi_t$ on $\sigma^* E^+$
    (resp. $\sigma^* E^-$) is contracting (resp. dilating),
    \emph{i.e.} there exist  constants $A,a >0$ such that 
    \begin{itemize}
   \item for
    any $e$ in $\sigma^* (E^+)_m$  and for any $t>0$ one has
    \begin{center}
      $\displaystyle \| \phi_{ t} e \|_{\phi_{ t} m} \leq A
      \exp(-at) \| e \|_m, $ 
      \end{center}
  \item for any $e$ in $\sigma^* (E^-)_m$  and for any $t>0$ one has
    \begin{center}
      $\displaystyle \| \phi_{- t} e \|_{\phi_{- t} m} \leq A
      \exp(-at) \| e \|_m, $ 
    \end{center}
    \end{itemize}
  where $\| \cdot \|$ is any continuous norm on $\mathcal{O}(\rho)$.
  \end{enumerate}
\end{definition} 
\begin{remark}
The definition of Anosov structure does not depend on the choice of the hyperbolic metric on $S$. 
\end{remark}
\begin{definition}
  A representation $\rho: \pi_1(S) \to G$ is said to be a $(P_+,  P_-)$-Anosov 
representation\index{Anosov representation}\index{representation!Anosov} 
if $\mathcal{O}(\rho)$
  carries an Anosov structure.
\end{definition}

The conditions on $\sigma$ are equivalent to requiring that ${\sigma}
(T^1{S})$ is a hyberbolic set for the flow $\phi_t$.  Stability of
hyperbolic sets implies stability for $(P_+, P_-)$-Anosov
representations:

\begin{proposition}[\cite{Labourie_anosov}]
The set of $(P_+, P_-)$-Anosov representations is open in $\R \big(\pi_1(S),G \big)$.
\end{proposition}

Since we fixed a hyperbolic structure on $S$ we can equivariantly identify 
the boundary $\partial \pi_1(S)$ with the boundary $S^1 = \partial \DD$ 
of the Poincar\'e disk as we did in \S~\ref{subsec:hyp_surf_quasiconj}. 
%

\begin{proposition}[\cite{Labourie_anosov}]\label{prop:boundarymap}
  To every Anosov representation $\rho: \pi_1(S) \to G$ there are
  associated continuous $\rho$-equivariant boundary maps 
  \bqn 
  \xi_\pm: S^1 \to G/P_\pm\,, 
  \eqn 
  with the property that for all $t,t' \in  S^1$ distinct, we have $\big(\xi_+(t), \xi_-(t') \big) \in \mathcal{O}$.
  Moreover, for every element $\gamma \in \pi_1(S) - \{1\}$ with fixed 
  points $\gamma^\pm \in S^1$ the point $\xi_\pm (\gamma^+)$ is the
  unique attracting fixed point of $\rho(\gamma)$ in $G/P_\pm$ and
  $\xi_\pm (\gamma^-)$ is the unique repelling fixed point of   $\rho(\gamma)$ in $G/P_\pm$.
\end{proposition}

\begin{proof}[Sketch of proof]
  Since the existence of these boundary maps play an important role in
  some results discussed above, let us sketch how these maps are
  obtained.  Recall that $T^1\widetilde{S}$ is naturally identified
  with the space of positively oriented triples in $S^1$, via the map
  \bqn
  \ba
  T^1\widetilde{S}\to&\quad (S^1)^{(3_+)}\\
  v \quad\mapsto& (v_+, v_0, v_-)\,,
  \ea 
  \eqn 
  where $v_\pm$ are the endpoints at
  $\pm \infty$ of the unique geodesic $g_v$ determined by
  $v$ and $v_0$ is the unique point which is mapped to the basepoint
  of $v$ under the orthogonal projection to the geodesic $g_v$ and
  such that $(v_+, v_0, v_-)$ is positively oriented.

  The existence of a continuous section $\sigma$ of
  $\mathcal{O}(\rho)$ is equivalent to the existence of
  a $\rho$-equivariant continuous map $F: T^1\widetilde{S} \to
  \mathcal{O}$. The Anosov condition (\ref{one}) on $\sigma$ is
  equivalent to $F$ being $\phi_t$-invariant.
  In particular, the map $F$ only depends of $(v_+, v_-) \in (S^1
  \times S^1 )- \textup{diag} =: (S^1)^{(2)}$. Thus we have a map \bqn
  F = (\xi_+, \xi_-) : (S^1)^{(2)} \to G/P_+ \times G/P_-.  \eqn It is
  not difficult to see that due to the contraction properties of the
  geodesic flow (see Anosov condition (\ref{two})) the map $\xi_+
  (v_+, v_-)$ only depends on $v_+$ and $\xi_- (v_+,v_-)$ only depends
  on $v_-$, and that $\xi_\pm$ satisfy the above properties.
\end{proof}

The property of a representation $\rho: \pi_1(S) \to G$ being
$(P_+,P_-)$-Anosov is indeed (almost) equivalent to the existence of such
continuous boundary maps.

\begin{proposition}[\cite{Guichard_Wienhard_dod}]\label{prop:boundary_Anosov}
  Let $\rho:\pi_1(S) \to G$ be a Zariski dense representation and assume that there
  exists $\rho$-equivariant continuous boundary maps $\xi_\pm : S^1
  \to G/P_\pm$ such that
\begin{enumerate}
\item for all  $t,t' \in S^1$ distinct, we have $\big(\xi_+(t), \xi_-(t') \big) \in \mathcal{O}$, and 
\item for all $t \in S^1$, the two parabolic subgroups stabilizing
  $\xi_+(t)$ and $\xi_-(t)$ contain a common Borel subgroup.
\end{enumerate}
Then $\rho$ is a $(P_+, P_-)$-Anosov representation.
\end{proposition}
The Anosov section $\sigma$ can be easily reconstructed from the
boundary map using the identification $T^1\widetilde{S} \cong
(S^1)^{(3_+)}$.
 
 Let us list several consequences of the existence of such boundary
 maps, for proofs see \cite{Labourie_anosov} and
 \cite{Guichard_Wienhard_dod}.
\begin{enumerate}
\item The representation $\rho$ is faithful with discrete image.
\item For every $\gamma \in \pi_1(S) -\{1\}$ the holonomy
  $\rho(\gamma)$ is conjugate to a (contracting) element in $H = P_+
  \cap P_-$.
\item The orbit map $\pi_1(S) \to X \,, \,  \gamma \mapsto \rho(\gamma)x_0$
  for some $x_0 \in X$ is a quasiisometric embedding with respect to
  the word metric on $\pi_1(S)$ and any $G$-invariant metric on the
  symmetric space $X= G/K$.
\item The representation $\rho$ is well-displacing. 
\end{enumerate}
The concepts of welldisplacing representations and quasiisometric
embeddings are discussed in more detail in \S~\ref{subsubsec:mapgroup}.

We already mentioned that representations in higher Teichm\"uller
spaces are Anosov representations. We now describe this in a little
more detail using the existence of special boundary map discussed in
\S~\ref{subsec:boundarymaps}.  We want to emphasize that this is not the
order in which results are proved. In many cases the proof of the
existence of a continuous boundary map is intertwined with the proof
of the Anosov property. For example one constructs first a not
necessarily continuous boundary map, establishes the contraction
properties for a (not continuous) section constructed out of this
boundary map. Then using the contraction property 
one can conclude that the map
is indeed continuous, hence defining a genuine Anosov section. 
(For an illustration
of this strategy for maximal representations into the symplectic group
we refer the reader to \cite{Burger_Iozzi_Labourie_Wienhard}.)
    The
special ``positivity conditions'' for the map usually are derived from
more specific properties of the representations.

\begin{enumerate}
\item Hitchin representations into $\PSL(n,\RR)$, $\PSp(2n,\RR)$,
  $\PO(n,n+1)$ are Anosov representations with $P_\pm$ being minimal
  parabolic subgroups \cite{Labourie_anosov}. Using the
  characterization of Hitchin representations via the existence of
  convex curves, we are given $\phi:S^1 \to \RR\PP^{n-1}$ and we can take
  $\phi^*:S^1\to (\RR\PP^{n-1})^*$ to be the dual curve, i.e. $\phi^*(t)$ is
 the unique osculating hyperplane of the curve $\phi$ containing $\phi(t)$. 
We set $\xi^+ = \phi$, $\xi^- = \phi^*$, then $(\xi^+,
  \xi^-): S^1 \to \RR\PP^{n-1} \times (\RR\PP^{n-1})^*$ satisfies the
  hypothesis of Proposition~\ref{prop:boundary_Anosov}, thus $\rho$ is
  Anosov with respect to the parabolic subgroup stabilizing a line in
  $\RR^n$. In order to see that Hitchin representations are actually
  Anosov with respect to the minimal parabolic subgroup, note that for
  any point on the convex curve we can consider the osculating flag
  and obtain maps $\xi_\pm = \xi: S^1\to G/P_{min}$. The convexity of
  $\phi$ implies the transversality condition on $\xi_\pm$ (see \cite[Chapter~5]{Labourie_anosov})
\item Maximal representations are Anosov representation with $P_\pm$
  being stabilizers of points in the Shilov boundary $\cs$ of the
  Hermitian symmetric space
  \cite{Burger_Iozzi_Labourie_Wienhard,Burger_Iozzi_Wienhard_anosov}.
This means in particular that in this case the boundary maps 
$\xi_\pm = \phi:S^1 \to \cs$ (Theorem~\ref{thm:5.21}) 
sending positively oriented triples to maximal triples are continuous. 
Since $(\xi_+, \xi_-)$ satisfies the transversality conditions required in Proposition~\ref{prop:boundary_Anosov}, 
maximal representations are $(P_+, P_-)$-Anosov.
\end{enumerate}


\subsection{Quotients of Higher Teichm\"uller spaces}\label{subsec:quotients}

\subsubsection{Action of the mapping class group\index{mapping class group}}\label{subsubsec:mapgroup}
The automorphism groups of $\pi_1(S)$ and of $G$ act naturally on
$\R \big(\pi_1(S),G \big)$ 
\bqn
\ba
\aut \big (\pi_1(S) \big) \times \aut(G) \times \R \big(\pi_1(S),G \big) &\to& \R \big(\pi_1(S),G \big)\\
(\psi, \alpha, \rho)\qquad\qquad\qquad &\longmapsto& \quad\alpha \circ\rho\circ\psi^{-1}\,.
\ea 
 \eqn
When we consider the quotient of the representation variety
$\R \big(\pi_1(S), G \big)/G$ this action descends to an action of the outer
automorphism group $\out \big (\pi_1(S) \big) = \aut \big(\pi_1(S) \big)/\inn \big(\pi_1(S) \big)$ on
$\R \big(\pi_1(S), G \big)/G$. As we discussed in \S~\ref{sec:hyp_str}, 
if $S$ is closed, the 
group of orientation preserving outer automorphisms of $\pi_1(S)$ is
isomorphic to the mapping class group $\map(S)$, and we will refer to
this action as the action of the mapping class group.

The components of $\R \big(\pi_1(S), G \big)/G$ which form higher Teichm\"uller
spaces, are preserved by this action.  In the case when $G =
\PSU(1,1)$ the action of the mapping class group on Teichm\"uller
space is properly discontinuous and the quotient $\Mm(S)$ is the
moduli space of Riemann surfaces.

Given a higher Teichm\"uller space it is natural to consider its
quotient by the action of the mapping class group, to study how it
relates to the moduli space of Riemann surface, and to investigate its
possible compactifications. The first question here is whether the action of
the mapping class group is properly discontinuous on higher
Teichm\"uller spaces.  

In order to answer this question the essential notion is that of a representation being well-displacing. 
For this let us introduce the translation lengths
$\tau$ and $\tau_\rho$ of an element $\gamma \in \pi_1(S) -\{1\}$:
\bqn
\ba
\tau(\gamma) =& \inf_{p\in \widetilde{S}} d(p, \gamma p)\\
\tau_\rho(\gamma) =& \inf_{z\in X} d_G(z, \rho(\gamma) z)\,, 
\ea
\eqn 
where $d$ is
the lift of a hyperbolic metric on ${S}$ and $d_G$ is a $G$-invariant
Riemannian metric on the symmetric space $X$.

\begin{definition}
  A representation $\rho: \pi_1(S) \to G$ is {\em well-displacing}\index{well-displacing} if
  there exist constants $A,B >0$ such that for all $\gamma \in
  \pi_1(S)$ \bqn A^{-1} \tau_\rho(\gamma) -B \leq \tau(\gamma) \leq A
  \tau_\rho(\gamma) +B.  \eqn
\end{definition}

The translation length $\tau$ depends on the choice of hyperbolic
metric on $S$. For any two choices the translations lengths are
comparable, thus the definition of well-displacing is independent of
the chosen hyperbolic metric on $S$.

It is shown in 
\cite{Labourie_energy, Wienhard_mapping, Hartnick_Strubel} 
that representations in higher Teichm\"uller spaces are 
  well-displacing. Then 
  a simple argument (see e.g. \cite{Labourie_energy, Wienhard_mapping})
shows that this implies that the action of the mapping class group on higher Teichm\"uller 
spaces is properly discontinuous. 
%
%


\begin{remark}
  The notion of well-displacing is related to the notion of
  quasiisometric embeddings. A representation $\rho: \pi_1(S) \to G$
  is a quasiisometric embedding\index{quasiisometric embedding} if there exist constants $A,B >0$
  such that for all $\gamma \in \pi_1(S)$ \bqn A^{-1}
  d_G \big(\rho(\gamma)z,z \big) -B \leq d(\gamma p,p) \leq A \,
  d_G \big(\rho(\gamma)z,z \big) +B, \eqn for some $p\in \widetilde{S}$ and some
  $z \in X$.
  
  Both notions can be defined more generally for representations of
  arbitrary finitely generated groups. The relation between the two
  notions is studied in \cite{Delzant_Guichard_Labourie_Mozes}.
  Representations in higher Teichm\"uller spaces are also
  quasiisometric embeddings \cite{Labourie_energy, Wienhard_mapping,
    Hartnick_Strubel, Guichard_Wienhard_dod}.
\end{remark}

\subsubsection{Relation to moduli space}
The notion of well-displacement also plays an important role when trying 
to obtain a mapping class group invariant projection from
higher Teichm\"uller spaces to classical Teichm\"uller space. To
describe this approach recall that given a representation $\rho:
\pi_1(S) \to G$ and hyperbolic metric $h\in \hyp(S)$, one can define
the energy of a $\rho$-equivariant smooth map $f : \widetilde{S} \to
X$ into the symmetric space $X = G/K$ as
$$E_\rho(f,h) := \int_S ||df||^2\, {\rm dvol},$$
where $||df||(p)$, $p\in S$ is the norm 
of the linear map $df_p$ with respect 
to the hyperbolic metric on $S$ and the $G$-invariant Riemannian metric on $X$.

The map $f$ is said to be harmonic
if and only if it minimizes the energy in its $\rho$-equivariant homotopy class.  
Setting $E_\rho(h): = \inf _{f} E_\rho(f,h)$, where $f$ ranges over all $\rho$-equivariant 
smooth maps $\widetilde{S} \to X$, we obtain a function 
\bqn
E_\rho: {\mathcal{F}(S)} = \operatorname{Diff}^+_0(S)\backslash\hyp(S) \to \RR, 
\eqn
called the {\em energy functional}\index{energy functional} associated to the representation
$\rho: \pi_1(S) \to G$.  The energy functional is a smooth function on the Fricke space 
${\mathcal{F}(S)}$. In the case when $G
=\PSU(1,1)$ and $\rho$ is a discrete embedding, it is known that
$E_\rho$ has a unique minimum \cite{Fischer_Tromba,Wolf}, namely the
hyperbolic structure determined by $\rho$.

In the general case, one would like to construct a mapping class group
invariant projection from higher Teichm\"uller spaces to classical
Teichm\"uller space by showing that the energy functional has a unique
minimum.  
As a first step we have 
\begin{theorem}[{\cite[Theorem 6.2.1]{Labourie_energy}}]
  Let $\rho: \pi_1(S) \to G$ be a well-displacing representation,
  then the energy functional $E_\rho$ is a proper function on 
  ${\mathcal{F}(S)}$.
\end{theorem}
%

In \cite{Labourie_energy} Labourie describes an approach to realize the Hitchin component 
for $\PSL(n,\RR)$ as a vector bundle over Teichm\"uller space in a equivariant
way with respect to the mapping class group. 
Recall for this that the isomorphism (see (\ref{eq:hitchin_map})) 
\bqn
h_{j}: \RH \big(\pi_1(S), \PSL(n,\RR) \big)/\PSL(n,\RR) \to 
H^0 \big(S,\oplus_{k=2}^n \Omega_{(S,j)}^{k} \big)\,,
\eqn 
where $H^0 (S, \Omega_{(S,j)}^k)$ is
the vector space of holomorphic differentials (with respect to the
fixed complex structure $j$ on $S$) of degree $k$, is not mapping class group invariant. 
Consider the vector bundle $\mathcal{E}^n$ over Teichm\"uller space $\mathcal{T}(S)$ realized 
as space of complex structures on $S$, where the fiber over the complex structure $j$ equals 
$H^0 \big(S,\oplus_{k=3}^n \Omega_{(S,j)}^{k} \big)$. Then $\mathcal{E}^n$ has the same dimension as 
 $\RH \big(\pi_1(S), \PSL(n,\RR) \big)/\PSL(n,\RR)$. Labourie defines the Hitchin map 
 \bqn
 (j, \omega) \longmapsto h_j(0,\omega)\, , 
 \eqn
 where $j \in \mathcal{T}(S)$ is a complex structure and $\omega \in H^0 \big(S,\oplus_{k=3}^n \Omega_{(S,j)}^{k} \big)$. 
 This map is equivariant with respect to the mapping class group action. 
Labourie proves that it is surjective \cite[Theorem~2.2.1]{Labourie_energy}  
and conjectures that $H$ is a homeomorphism, which would imply  
 
  \begin{conjecture}[{\cite[Conjecture~2.2.3]{Labourie_energy}}]\label{conj:bundle}
   The quotient of  the Hitchin component $\RH \big(\pi_1(S), \PSL(n,\RR) \big)/\PSL(n,\RR)$ by the mapping class
   group is a vector bundle over the moduli space of Riemann surfaces
   with fiber being the space of holomorphic $k$-differentials 
   $H^0 \big(S,\oplus_{k=3}^n \Omega_S^{k} \big)$.
  \end{conjecture}
  
  In order to prove this conjecture it would be sufficient to show that for a Hitchin representation 
  $\rho\in \RH \big(\pi_1(S), \PSL(n,\RR) \big)/\PSL(n,\RR)$
 the energy functional $E_\rho$ has a nondegenerate minimum.  

Conjecture~\ref{conj:bundle} has been proved for $n=2$ and $n=3$. 
The proof for $\PSL(3,\RR)$ is independently due to Labourie \cite{Labourie_projective} and Loftin \cite{Loftin}.  
They rely on the description of $\RH \big(\pi_1(S), \PSL(3,\RR) \big )$ as deformation
space of convex real projective structures due to Choi and Goldman,
and use the theory of affine spheres developed in 
\cite{Cheng_Yau_1,Cheng_Yau_2} in order to prove 

\begin{theorem}[\cite{Labourie_projective,Loftin}]
  The quotient 
  \bqn
  \map(S) \backslash \RH \big(\pi_1(S), \PSL(3,\RR) \big)/\PSL(3,\RR)
  \eqn 
  is a vector bundle over the moduli space of
  Riemann surface with fiber being the space of cubic holomorphic
  differentials on the surface.
\end{theorem}

  
  For maximal representations the quotients by the mapping class group
  are expected to look more complicated due to the fact that
\begin{enumerate}
\item the connected components consisting of maximal representations
  might have singularities, and
\item the space of maximal representations has several connected
  components which need to be treated separately. (We will come back
  to this problem in \S~\ref{subsec:invariants}.)
\end{enumerate}

\subsubsection{Compactifications}
Mapping class group equivariant compactifications of higher
Teichm\"uller spaces are partially understood.

A general construction to compactify the space of discrete, injective
nonparabolic representations of a finitely generated group into $G$
using the generalized marked length spectrum is given in
\cite{Parreau}. This construction applies to give compactifications of
higher Teichm\"uller spaces. Boundary points in this compactification
can be interpreted as actions on $\RR$-buildings \cite{Paulin}.

For the Hitchin component $\RH \big(\pi_1(S), \PSL(3,\RR) \big)/\PSL(3,\RR)$ the
identification with the deformation space of convex real projective
structures allows to obtain a better understanding of this
compactification, see e.g.
\cite{Kim,Loftin_compactifications,Loftin_limits,Cooper_etal}. Through
the study of degenerations of convex projective structures, Cooper et
al. \cite{Cooper_etal} obtain a description of boundary points as
mixtures of measured laminations and special Finsler metrics
(Hex metrics) on $S$.

Fock and Goncharov construct tropicalizations of the spaces of positive representations 
which they expect to provide (partial) completions \cite{Fock_Goncharov} when $S$ is
noncompact. But except for the case when $G = \PSL(2,\RR)$ (treated
in \cite{Fock_Goncharov_htt}), they do not define a topology of the union of 
space of positive representations and its tropicalized counterpart. 

\subsubsection{Crossratios}
Realizing $\partial\DD \subset \CC\PP^1$, the restriction of
the classical crossratio function\index{crossratio function}\index{crossratio} on $\CC\PP^1$  
\bqn 
c(x, y, t, z) =
\frac{x-y}{x-t} \frac{z-t}{z-y} 
\eqn 
gives a
continuous real valued $\PSU(1,1)$-invariant function on 
\bqn
(\partial\DD)^{4*}:= \big\{(x, y, z, t) \in (\partial \DD)^4 \, :\, x\neq t\, y \neq z \big\}\,.
\eqn
This crossratio and several
generalizations (see e.g. \cite{Otal,Ledrappier}) play an important
role in the study of negatively curved manifolds and hyperbolic
groups.  

Given a hyperbolic element $\gamma \in \PSU(1,1)$ its period is
defined as \bqn l_c(\gamma) = \log c(\gamma^-, z, \gamma^+, \gamma
z)\,, \eqn where $\gamma^+ $ is the unique attracting fixed point and $\gamma^- $ 
the unique repelling fixed point of $\gamma$ in $\partial \DD$ 
and $z\in \partial \DD- \{\gamma^\pm\}$ is arbitrary. The period of
$\gamma$ equals the translation length 
$\tau(\gamma) = \inf_{p \in  \DD} d_{\DD} (p, \gamma p)$.

Given a discrete embedding $\rho: \pi_1(S) \to \PSU(1,1)$, let 
\bqn
c_\rho:= \phi^*c: (S^1)^{4*} \to \RR
\eqn
be the pullback of $c$ by some $\rho$-equivariant boundary map,  be the associated crossratio
function. Then $c_\rho$ contains all information about the marked
length spectrum of $S$ with respect to the hyperbolic metric defined
by $\rho$. In particular, two discrete embeddings $\rho_1, \rho_2$ are
conjugate if and only if $c_{\rho_1} = c_{\rho_2}$.

 A generalized crossratio function is a $\pi_1(S)$-invariant continuous functions 
\bqn
(S^1)^{4*} = \big\{(x,y,z,t) \in (S^1)^4 \, :\,  x
\neq t\, y \neq z \big\} \to \RR 
\eqn
satisfying the following relations \cite[Introduction]{Labourie_crossratio}:
\begin{enumerate}
\item (Symmetry) $c(x,y,z,t) = c(z,t,x,y)$
\item (Normalization) \begin{enumerate}
    \item[ ] $c(x,y,z,t)  = 0$ if and only if $x=y$ or $z=t$
    \item[ ] $c(x,y,z,t) = 1$ if and only if $x=z$ or $y=t$
      \end{enumerate}
\item (Cocycle identity) \begin{enumerate} 
\item[ ] $c(x,y,z,t) = c(x,y,z,w) c(x,w,z,t)$
\item[ ] $c(x,y,z,t) = c(x,y,w,t) c(w,y,z,t)$.
\end{enumerate}
\end{enumerate}
Among such functions crossratios arising from a discrete embedding
$\pi_1(S) \to \PSU(1,1)$ are uniquely characterized by the functional
equation $ 1-c(x,y,z,t) = c(t,y,z,x)$.

The study of generalized crossratio functions associated to higher
Teichm\"uller spaces has been pioneered by Labourie. In particular, he
associates a generalized crossratio function to any Hitchin
representation into $\PSL(n,\RR)$ and shows that crossratio functions
arising from a Hitchin representation into $\PSL(n,\RR)$ are
characterized by explicit functional equations \cite{Labourie_crossratio}.
 
In \cite{Labourie_McShane} Labourie and McShane establish generalized
McShane identities for the crossratios associated to Hitchin
representations into $\PSL(n,\RR)$.

\begin{remark}
  Related crossratio functions of four partial flags consisting of a
  line and a hyperplane are used in the work of Fock and Goncharov
  \cite{Fock_Goncharov} in order to construct explicit coordinates for
  the space of positive representations into $\PSL(n,\RR)$.
\end{remark}

In the context of maximal representations crossratio functions have
been defined and studied by Hartnick and Strubel
\cite{Hartnick_Strubel}. They construct crossratio functions defined
on a suitable subset $\cs^{4*}$ of the fourfold product of the Shilov
boundary of any Hermitian symmetric space of tube type. They show that
there is a unique such crossratio function which satisties some
natural functorial properties.
Given a maximal representation $\rho: \pi_1(S) \to G$ a concrete implementation of the continuous
boundary map $\phi: S^1 \to \cs$ (see Theorem~\ref{thm:5.21}) allows
to pullback this crossratio function to a generalized crossratio
function on $(S^1)^{4*}$.  The well-displacing property of
representations in higher Teichm\"uller spaces can be easily deduced
from the existence of generalized crossratio functions.  In all works
investigating crossratio functions, the existence of boundary maps
with special positivity properties (as discussed in
\S~\ref{subsec:boundarymaps}) play an important role.

\subsection{Geometric structures\index{geometric structures}}\label{subsec:geom_struc}
We already mentioned that Hitchin had asked in \cite{Hitchin} about
the geometric significance of Hitchin components, and one might raise
the same question for maximal representation,
even though the picture there seems to be more complicated due to the
fact that the space of maximal representations has singularities and
multiple components.

Interpreting higher Teichm\"uller spaces as deformation spaces of
geometric structures is not just of interest in itself. Any such
interpretation gives an important tool to study these spaces, their
quotients by the mapping class group, their relations to the moduli
space of Riemann surfaces as well as their compactifications. This is
illustrated by the fact that the deeper understanding of these
questions for the Hitchin component of $\PSL(3,\RR)$ relies on the
Theorem by Choi and Goldman, which we already mentioned above:

\begin{theorem}[\cite{Goldman_convex,Goldman_Choi}]\label{thm:goldman_choi}
  The Hitchin component 
  \bqn
  \RH \big(\pi_1(S), \PSL(3,\RR) \big)/\PSL(3,\RR)
  \eqn
  parametrizes convex real projective structures\index{convex real projective structure}\index{projective structure} on $S$.
\end{theorem}

The original proof of this theorem relied on Goldman's work on convex
projective structures on surfaces \cite{Goldman_convex}, which implied
that the deformation space of these structures is an open domain in
$\RH \big(\pi_1(S), \PSL(3,\RR) \big)/\PSL(3,\RR)$. Goldman and Choi \cite{Goldman_Choi} 
then proved
that this subset is furthermore closed, establishing the above
theorem.

In terms of the properties of Hitchin representations we have
discussed so far, Theorem~\ref{thm:goldman_choi} is basically
equivalent to the characterization of Hitchin representations into
$\PSL(3,\RR)$ by the existence of a  convex map from $S^1$ into $\RR\PP^2$
(Theorem~\ref{thm:convexmap}). We give a sketch of how
Theorem~\ref{thm:goldman_choi} follows from
Theorem~\ref{thm:convexmap} when $n=3$.

\begin{proof}[Sketch of a proof of Theorem~\ref{thm:goldman_choi} assuming Theorem~\ref{thm:convexmap}]
A convex real projective structure on $S$ is a pair $(N, f )$,
where $N$ is the quotient
$\Omega /\Gamma$ of a strictly convex domain $\Omega$ in $\RR\PP^{2}$ by a
discrete subgroup $\Gamma$ of $\PSL(3,\RR)$, and $f : S \rightarrow N$ is
a diffeomorphism.

Starting from a representation $\rho: \pi_1(S) \to \PSL(3,\RR)$ in the Hitchin
component, let $\Omega_\xi \subset \RR\PP^{2}$ be the strictly
convex domain bounded by the convex curve $\xi(S^1)
\subset \RR\PP^2$. 
Then $\rho \big(\pi_1(S) \big)$ acts freely and properly
discontinuously on $\Omega_\xi$. The quotient $\Omega_\xi/\rho \big(\pi_1(S) \big)$
is a real projective convex manifold, diffeomorphic to $S$. Conversely
given a real projective structure on $S$, we can $\rho$-equivariantly
identify $S^1$ (identified with the boundary of $\pi_1(S)$) 
with the boundary of $\Omega$ and get a
convex curve $\xi: S^1 \to \partial \Omega \subset
\RR\PP^{2}$.
\end{proof}

Inspired by this proof and with Theorem~\ref{thm:convexmap} at hand
for arbitrary $n$, one might try to follow a similar strategy on order
to find geometric structures parametrized by the Hitchin component for
$\PSL(n,\RR)$. This works for $n=4$, where we obtain the following
\begin{theorem}[\cite{Guichard_Wienhard_convex}]\label{thm:guichard_wienhard_convex}
 \label{thm:guichard_wienhard_PGL4}
  The Hitchin component for $\PSL(4,\RR)$ is naturally homeomorphic to the
  moduli space of properly convex foliated projective structures on $T^1S$.
\end{theorem}  

Properly convex foliated
projective structures are locally homogeneous 
$\big( \PSL(4, \RR),\RR\PP^3 \big)$-structures on $T^1 S$ 
satisfying the following additional conditions:
\begin{itemize}
\item every orbit of the geodesic flow is locally a projective line,
\item every (weakly) stable leaf of the geodesic flow is locally a projective
  plane and the projective structure on the leaf obtained by restriction is
  convex.
\end{itemize} 

Using the convex curve provided by Theorem~\ref{thm:convexmap} one can
consider the corresponding discriminant surface $\Delta \subset
\RR\PP^3$, i.e. the union of all its tangent lines. The complement
$\RR\PP^3 -\Delta$ consists of two connected components, on both of
which $\rho \big(\pi_1(S) \big)$ acts properly discontinuous.  The
quotient of one of the connected components by $\pi_1(S)$ is
homeomorphic to $T^1 S$, equipped with a properly convex foliated
projective structure.
%
The main work goes into
establishing the converse direction, i.e. showing that the holonomy
representation of a properly convex foliated projective structure on
$T^1 S$ lies in the Hitchin component -- this is rather tedious.

\begin{remark}   
  The above theorem implies that the Hitchin component for
  $\PSp(4,\RR)$ is naturally homeomorphic to the moduli space of
  properly convex foliated projective contact structures on the unit
  tangent bundle of $S$.
\end{remark}

For $n\geq 5$ the above strategy seems to fail in general. 


The first step in the strategy described above to find geometric
structures which are parametrized by representations $\rho: \pi_1(S)
\to G$ in higher Teichm\"uller spaces is to find domains of
discontinuity for such representations in homogeneous spaces, more
precisely in generalized flag varieties associated to $G$, on which
$\pi_1(S)$ is supposed to act with compact quotient. This problem
becomes more difficult the bigger $G$ gets, since $\pi_1(S)$ is a
group of cohomological dimension $2$, whereas the dimension of the
generalized flag varieties grows as $G$ get bigger.  So it comes a bit
as a surprise that finding domains of discontinuity with compact
quotient can be accomplished in the very general setting of Anosov
representations.

\begin{theorem}[\cite{Guichard_Wienhard_dod_ann, Guichard_Wienhard_dod}]\label{thm:domainsofdiscontinuity}
  Let $G$ be a semisimple Lie group and assume that no simple factor of $G$ is locally isomorphic to $\PSL(2,\RR)$.
  Let $\rho: \pi_1(S) \to G$ be a $(P_+,
  P_-)$-Anosov representation.  
  Let $P=MAN$ be the minimal parabolic subgroup of $G$. 
Then there exists an open non-empty set $\Omega_\rho \subset G/AN$, 
  on which $\pi_1(S) $ acts freely, properly discontinuous and with
  compact quotient.
\end{theorem}

\begin{remark}
The homogeneous space $G/AN$ is the maximal compact quotient of $G$. 
In many cases the domain $\Omega_\rho$ descends to a domain of discontinuity in $G/P$. 
\end{remark}

\begin{remark}
  Theorem~\ref{thm:domainsofdiscontinuity} holds more general for Anosov representations
  of convex cocompact subgroups of Hadamard
  manifolds of strictly negative curvature or even of hyperbolic groups. 
The reader interested in the more general statement is
  referred to \cite{Guichard_Wienhard_dod}.
\end{remark}

The main tool in order to define the domain of discontinuity
$\Omega_\rho$ are the $\rho$-equivariant continuous boundary maps
$\xi_\pm: S^1 \to G/P_\pm $ associated to the $(P_+, P_-)$-Anosov
representation (see Proposition~\ref{prop:boundarymap}).

There is some evidence that -- at least in the case of higher
Teichm\"uller spaces -- the quotients $\Omega_\rho / \rho \big(\pi_1(S) \big)$ 
are homeomorphic to the total spaces of bundles over $S$ with compact fibers.
This has been established for maximal representation into $\Sp(2n,\RR)$ 
as well as for Hitchin representations into $\SL(2n,\RR)$. 

\begin{theorem}[\cite{Guichard_Wienhard_dod}]
\begin{enumerate}
\item The Hitchin component for $\SL(2n,\RR)$ parametrizes 
real projective structures on a compact manifold $M$, 
which is topologically a $\O(n)/\O(n-2)$-bundle over the surface $S$.
\item Maximal representations into $\Sp(2n,\RR)$ parametrize real 
projective structures on a compact manifold $M$ 
homeomorphic to
an $\O(n)/\O(n-2)$-bundle over the surface $S$. 
Its isomorphism type depends on the connected component 
containing the representation.  
\end{enumerate}
\end{theorem}

\subsection{Topological invariants}\label{subsec:invariants}
The Hitchin component is by definition a single 
connected component, but the space of maximal representations is a priori
only a union of connected components, and their might be more than
one. In many cases the exact number of connected components of the
space of maximal representations has been computed using methods from
the theory of Higgs bundles\index{Higgs bundle} \cite{Gothen, GarciaPrada_Mundet,
  GarciaPrada_Gothen_Mundet, Bradlow_GarciaPrada_Gothen,
  Bradlow_GarciaPrada_Gothen_survey}.  And the most interesting family
in terms of the number of connected components are maximal
representations into symplectic groups $\Sp(2n,\RR)$: there are
$3\times 2^{2g}$ connected components when $n\geq 3$
\cite{GarciaPrada_Gothen_Mundet} and $(3 \times 2^{2g} + 2g-4)$
connected components when $n=2$ \cite{Gothen}.
  
Invariants to distinguish these connected components can be derived
from the associated Higgs bundles, but topological invariants to
distinguish the different connected components also arise from
considering maximal representations as Anosov representations.

Recall that in the definition of Anosov structures one considers the
flat $G$-bundle $G(\rho)$ over $T^1 S$ and the
associated bundle
$\mathcal{O}(\rho)$.  
The first part of the data of an
Anosov structure is a section $\sigma$ of
$\mathcal{O}(\rho)$. Since $\mathcal{O}(\rho)$ is the $G/H$-bundle associated to $G(\rho)$ its 
sections are in one-to-one
correspondence with reductions of the structure group of
$G(\rho)$ from $G$ to $H$. In general there is no canonical
section, but in the case of Anosov structures we have
\begin{proposition}[\cite{Guichard_Wienhard_invariants}]
  If a section $\sigma$ of $\mathcal{O}(\rho)$ with the
  properties required in Definition~\ref{def:anosov_def} exists, then
  it is unique.
\end{proposition}

As a consequence an Anosov representation $\rho: \pi_1(S) \to G$ gives
a canonical reduction of the $G$-principal bundle $G(\rho)$
to an $H$-principal bundle. This $H$-bundle is
in general not flat; its characteristic classes give topological
invariants of the Anosov representation $\rho$ which live in
$\textup{H}^*(T^1S)$.

In the situation of maximal representations $\rho: \pi_1(S) \to
\Sp(2n,\RR)$, we have that $H= \GL(n,\RR)$, embedded into $\Sp(2n,\RR)$ as
the stabilizer of two transverse Lagrangian subspaces.  The
topological invariants of significance are first and second
Stiefel--Whitney classes\index{Stiefel--Whitney class}, as well as an Euler class\index{Euler class} if $n=2$.

\begin{theorem}[\cite{Guichard_Wienhard_invariants}]
  The topological invariants distinguish the connected components of
  $\RM \big(\pi_1(S), \Sp(2n,\RR) \big) \setminus\RH \big(\pi_1(S), \Sp(2n,\RR) \big)$.
  
  Considering Hitchin representations as $(P_{min},
  P^{opp}_{min})$-Anosov representations there is an additional first
  Stiefel--Whitney class, which distinguishes the connected components
  of $\RH \big(\pi_1(S), \Sp(2n,\RR \big)$.
\end{theorem}

The invariants constructed using the Anosov property of maximal
representations are in principle computable for a given representation
$\rho: \pi_1(S) \to \Sp(2n,\RR)$. Explicit computations for various
representations allows us to describe model representations in any
connected component.  This is of particular interest for $\Sp(4,\RR)$
as there are $2g-3$ connected components in which every representation
is Zariski dense (see also
\cite{Bradlow_GarciaPrada_Gothen_sp4}).

Besides the {\em irreducible Fuchsian representation\index{Fuchsian representation}\index{representation!Fuchsian}}
which were introduced to define the Hitchin component, there are two
other kinds of model representations: A {\em twisted diagonal
  representation}\index{twisted diagonal
  representation}\index{diagonal
  representation}\index{representation!diagonal}\index{representation!twisted diagonal} is a maximal representation 
  \bqn
  \rho_\theta = (\iota\otimes \theta): \pi_1(S) \to \SL(2,\RR)\times\O(n) \subset
\Sp(2n,\RR)\,,
\eqn 
where $\iota: \pi_1(S) \to \SL(2,\RR)$
is a discrete embedding and $\theta: \pi_1(S) \to \O(n)$ is an orthogonal
representation; $\SL(2,\RR)\times\O(n)$ sits in $\Sp(2n,\RR)$ as the
normalizer of the diagonal embedding 
\bqn
\SL(2,\RR) \to \SL(2,\RR)^n\subset \Sp(2n,\RR)\,.
\eqn  
A {\em hybrid representation}\index{hybrid representation}\index{representation!hybrid} is a maximal
representation 
\bqn
\rho_k = \rho_1* \rho_2: S = S_1 \cup_{\gamma}S_2 \to\Sp(2n,\RR)\,,
\eqn
 $k= 3-2g, \cdots, -1$, which is obtained by amalgamation
of an irreducible Fuchsian representation on $\pi_1(S_1)$ and a
suitable deformation of an (untwisted) diagonal representation on
$\pi_1(S_2)$. The subscript $k$ indicates the Euler characteristic of
$S_1$. The construction of hybrid representations relies on the
additivity of the Toledo number and the Euler characteristic under
gluing (see Proposition~\ref{prop:5.8}).

\begin{theorem}[\cite{Guichard_Wienhard_invariants}]
  When $n\geq 3$ any maximal representation $\rho: \pi_1(S) \to
  \Sp(2n,\RR)$ can be deformed either to an irreducible Fuchsian
  representation or to a twisted diagonal representation.
  
  When $n= 2$ there are $2g-3$ connected components $\Hh_k$,
  $k=1,\cdots, 2g-3$ of $\RM \big(\pi_1(S), \Sp(4,\RR) \big)$ in which every
  representation has Zariski dense image. Representations in $\Hh_k$ 
can be deformed to $k$-hybrid representations.
\end{theorem}

The information about model representations in each connected
component can be used to obtain further information about the
holonomies of maximal representations.

For representations in the Hitchin components,
Theorem~\ref{thm:Labourie} implies that for every $\gamma \in \pi_1(S)
- \{1\}$ the image $\rho(\gamma)$ is diagonalizable over $\RR$ with
distinct eigenvalues.  This does not hold for the other
components of maximal representations. 

The Anosov property implies that for every  $\gamma \in \pi_1(S)\setminus \{e\}$ 
the image $\rho(\gamma)$ is conjugate to an element in $\GL(n,\RR) < \Sp(2n,\RR)$. 
More precisely, we have:
\begin{theorem}[\cite{Guichard_Wienhard_invariants}\label{thm:holonomy}]
Let $\hH$ be a connected component of 
\bqn
\RM \big(\pi_1(S),\Sp(2n,\RR) \big) \setminus \RH \big(\pi_1(S), \Sp(2n,\RR) \big)\,,
\eqn
and let $\g \in \pi_1(S)-\{1\}$ be an element corresponding to a
simple curve.  
Then there exist 
\begin{enumerate}
\item a representation $\rho \in \hH$ such that the Jordan
  decomposition of $\rho(\g)$ in $\GL(n,\RR)$ has a nontrivial
  parabolic component.
\item a representation $\rho' \in \hH$ such that the Jordan
  decomposition of $\rho(\g)$ in $\GL(n,\RR)$ has a nontrivial
  elliptic component.
\end{enumerate}
\end{theorem}

This results indicates that understanding the structure of the space
of maximal representations is much more complicated than understanding
the structure of Hitchin components, since already the conjugacy
classes in which the holonomy of one element can lie in might differ
from connected component to connected component.

\section{Open questions and further directions}\label{subsec:questions}
In the previous sections we already mentioned some open questions regarding the quotients of 
higher Teichm\"uller spaces, their compactifications as well as their geometric significance. 
In this section we want to conclude our survey with mentioning some further directions in the study of 
higher Teichm\"uller spaces which to our knowledge have not yet been explored. 

\subsection{Positivity\index{positivity}, causality\index{causality} and other groups}
As we pointed out in \S~\ref{subsec:boundarymaps}
 an underlying common structure of higher Teichm\"uller spaces is 
that the homomorphisms in them admit equivariant boundary maps 
which satisfy some positivity or causality property. 

The relation between positive triples in the full flag variety
and the weaker notion of 
maximal triples in the space of Lagrangians  discussed in \S~\ref{subsubsec:symplectic} is very special. 
It
would be very interesting to discover weaker notions of positivity of
$k$-tuples in (partial) flag varieties which then extends to other
groups which are neither split real forms nor of Hermitian type. 
Such notions of positivity might lead to discovering higher
Teichm\"uller spaces for other Lie groups $G$, which are again characterized by
the existence of special boundary maps. 

A first family of groups to look at could be $G=\PO(p,q)$, which is of
Hermitian type if $(p,q) = (2,q)$ and a split real form if $(p,q) =
(n, n+1)$ or $(p,q) = (n,n)$.

Every time there is a notion of positivity or cyclic ordering, 
the images of boundary maps tend to be more regular, namely rectifiable circles.
This contrasts with the case of quasifuchsian deformations into $\PSL(2,\CC)$
of compact surface groups in $\PSL(2,\RR)$,
where in fact the limit set, or -- what is the same -- the image of the boundary map,
is a topological circle with Hausdorff dimension larger than 1, unless the deformed group 
is Fuchsian.\footnote{This is due to \cite{Bowen}:  see also the footnote on p.~76 of Fricke's address
in Chicago in 1893, \cite{Fricke}.}
This suggests to study the deformations of the homomorphism
\bqn
i\circ\rho:\pi_1(S)\to\gG(\CC)\,,
\eqn
where $i:G=\gG(\RR)^\circ\to\gG(\CC)$ is the natural inclusion and $\rho:\pi_1(S)\to G$ is 
either a maximal representation into a group of Hermitian type or a Hitchin representation 
into a real split Lie group.  Observe that $i\circ\rho$ is Anosov for a suitable pair of parabolic
subgroups and, as a result, small deformations of $i\circ\rho$ are as well.

\subsection{Coordinates and quantizations \\ for maximal representations}
Fock and Goncharov describe explicit coordinate for the space of positive representations.
 For $\PSL(n,\RR)$ these coordinates have a particular nice form. 
Based on the explicit coordinate system they describe the cluster variety structure and 
quantizations of the space of positive representations. 

It would be interesting to construct similar explicit coordinate
systems for the space of maximal representations, in particular when
$G = \Sp(2n,\RR)$. Theorem~\ref{thm:holonomy} gives a hint that
constructing coordinates for the space of maximal representations is
more involved. The structure of the coordinates also needs to be more
complicated as they have to model the singularities of the space of
maximal representations.

The additivity of the Toledo number on the other hand implies that the
space of maximal representations of a compact surface $S$ can be built
out of the space of maximal representations of a pair of pants.

Having coordinates at hand, one might also ask for quantizations of
the space of maximal representations or try to express the symplectic
form on the space of maximal representations explicitly in
coordinates.

\frenchspacing
\bibliographystyle{abbrv}

\vskip1cm
\printindex
\end{document}